\documentclass[reqno,12pt,letterpaper]{amsart}
\usepackage[proof]{newmzmacros1}
\usepackage[all,cmtip]{xy}
\newcommand{\e}{\epsilon}
\newcommand{\RR}{{\mathbb R}}
\newcommand{\NN}{{\mathbb N}}
\newcommand{\CC}{{\mathbb C}}
\newcommand{\ZZ}{{\mathbb Z}}
\newcommand{\CI}{{C^\infty}}

\newcommand{\TT}{{\mathbb T}}
\usepackage[normalem]{ulem}

\renewcommand{\e}{\epsilon}
\newcommand{\mc}[1]{\mathcal{#1}}
\newcommand{\spec}{\operatorname{spec}}

\newcommand{\specppL}{{\rm{spec}}_{{\rm{pp}},L^2}}

\title{Viscosity limits for 0th order pseudodifferential operators}
\author{Jeffrey Galkowski}
\email{j.galkowski@ucl.ac.uk}
\address{Department of Mathematics, University College London, WC1H 0AY, UK}

\author{Maciej Zworski} 
\email{zworski@math.berkeley.edu}
\address{Department of Mathematics, University of California, Berkeley, CA 94720}
\begin{document}

%\begin{center}
%\includegraphics[width=6.65in]{hesj4d_cover}
%\end{center}

\begin{abstract}
Motivated by the work of Colin de Verdi\`ere and Saint-Raymond \cite{CS} on 
spectral theory {for} 0th order pseudodifferential operators on tori we consider
 viscosity limits in which 0th order operators{,} $ P $, are replaced by 
$ P + i \nu \Delta $, $ \nu > 0 $. By adapting the Helffer--Sj\"ostrand theory of scattering resonances \cite{HS}, we show that{,} in a complex neighbourhood of the continuous spectrum, eigenvalues of $ P + i \nu \Delta $ have limits as {the} viscosity{,} $ \nu $, goes to 0. In the simplified setting of tori{,} this justifies claims made in the physics literature -- see for instance \cite{Rieu}.
\end{abstract}

\maketitle

\section{Introduction}
\label{s:intr}

Spectral properties of 0th order pseudo-differential operators 
arise naturally in the problems of fluid mechanics -- for an early example see  
Ralston \cite{Ra}. Recently, 
Colin de Verdi\`ere and Saint-Raymond \cite{CS}, \cite{C} investigated such 
operators under natural dynamical conditions motivated by the study of
(linearized) internal waves -- see the review article of Dauxois et al 
\cite{Da} and the introduction to \cite{CS} for a physics perspective and references. Dyatlov--Zworski \cite{DyZw19} provided  proofs
of the results of \cite{CS} based on the analogy to scattering theory --
see Melrose--Zworski \cite{mz}, Hassell--Melrose--Vasy \cite{hmv} and \cite{dizzy}. This analogy was developed further by Wang \cite{JW} who defined and described a  
scattering matrix in this setting. Tao \cite{ZTao} constructed an 
example of an embedded eigenvalue.

\renewcommand\thefootnote{\dag}%  

Motivated by the physics literature -- see for instance Rieutord et al \cite{Rieu} --  we consider here operators with a viscosity term
\[  P_\nu := P + i \nu \Delta ,\]
where $ P $ is a 0th order pseudodifferential operator {on the torus}~\eqref{eq:defP0} {satisfying}~\eqref{e:SymbolAnalyticity0} and the dynamical assumption~\eqref{eq:HpG0}. The operator $ \Delta $ is the usual Laplacian on the torus.
The assumption~\eqref{eq:HpG0} guarantees continuity of 
the spectrum at $ 0 $  \cite{CS}, \cite{DyZw19}.  We then show that as $ \nu \to 0 + $ the eigenvalues of $ P_\nu$ 
 in a complex neighbourhood of $ 0$ tend to a discrete set associated to $ P $ alone -- see Figure \ref{f:1} for a numerical illustration. {This} justifies 
claims seen in related models of the physics literature\footnote{For example a claim from \cite{Rieu}:
``The aim of this paper is to present what we believe to be the asymptotic limit of inertial modes in a spherical shell when viscosity tends to zero."}. Our approach is again based on analogy to scattering theory, 
in this case to the general approach to scattering resonances due to 
Helffer--Sj\"ostrand \cite{HS}.
\begin{center}
 \begin{figure}
\includegraphics[width=13.5cm]{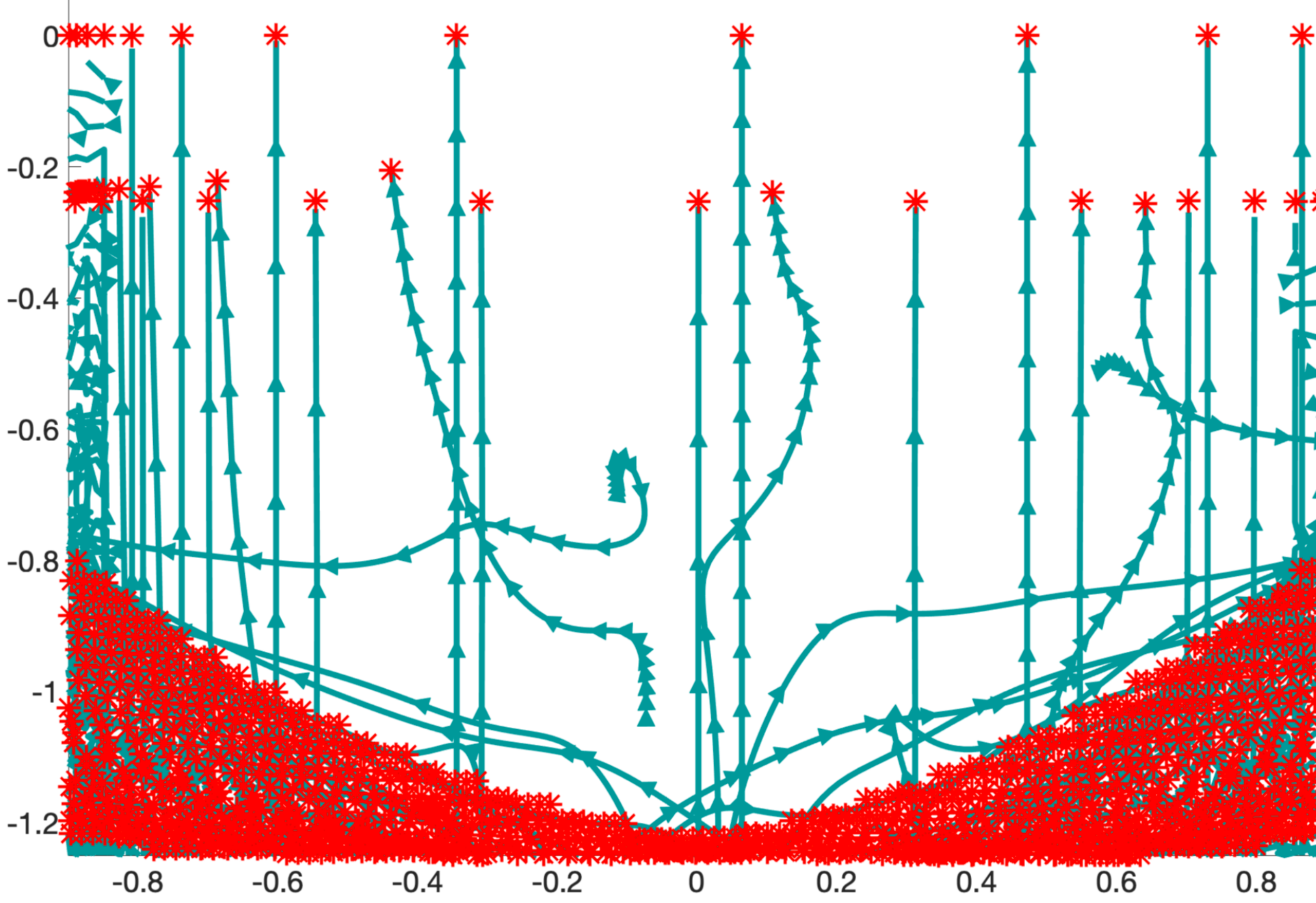}
\caption{\label{f:1}We display the resonances of $ P $ as red stars (a full explanation using the deformed operator $ P_\theta $ is given in Appendix B). The paths of the eigenvalues of $P+i\nu \Delta$ as $\nu \to 0+$ are shown by the green curves with the arrows denoting the direction of the path as $\nu$ decreases. $P$ is chosen as in~\eqref{e:numeric} with $V_a=\tfrac{1}{2}(\xi^3-1)e^{-\xi^2}$ and $V_m=(1 + (e-1) (\xi-2)^2 )e^{-(\xi-2)^2}$. For an animated version of this figure see
\url{https://math.berkeley.edu/~zworski/vis_dynam.mov}.}
\end{figure}
\end{center}

To state our results precisely{,} we start with the class of pseudodifferential operators: 
\begin{equation}
\label{eq:defP0}
Pu(y):=\frac{1}{(2\pi h)^n}\int_{\RR^n}\int_{\RR^n}  e^{\frac{i}{h}\langle y-y',\eta\rangle}p(y,\eta)u(y')dy'd\eta
\end{equation}
where 
$p\in S^m ( T^* \TT^n ) $, $ \TT^n := \RR^n/ 2\pi \ZZ^n $, has an analytic continuation from $T^*\mathbb{T}^n$ satisfying
\begin{equation}
\label{e:SymbolAnalyticity0}
|p(z,\zeta)|\leq M, \ \text{ for }  \ |\Im z|\leq a, \ \ |\Im \zeta|\leq  b\langle \Re \zeta\rangle .
\end{equation}
The integral in the definition \eqref{eq:defP0} of $ Pu $  is considered in the sense of oscillatory integrals (see for instance \cite[\S 5.3]{zw}) 
and we extend both $ y \mapsto u ( y ) $ and $ y \mapsto p ( y, \eta ) $
to periodic functions on $ \RR^n $.

The dynamical assumption is formulated using an {\em escape function}:
\begin{equation}
\exists \, G\in S^1(T^*\mathbb{T}^n), \ C>0 \ \  H_pG(x,\xi)>0, \ \ \text{ for } (x,\xi)\in \{p=0\}\cap \{|\xi|>C\}.
\label{eq:HpG0}
\end{equation}
(For the definition of the symbol class $ S^1 ( T^* \TT^n ) $ see 
\eqref{e:controlDeformation} and \cite[\S E.1]{dizzy} and for a discussion  of escape functions \cite[\S 6.4]{dizzy}.) 
We make our assumption at $ p= 0$ but {the value} $ 0$ can be replaced by any real number $ \lambda $ by changing the operator {to} $ P - \lambda $.
We could also replace $ p $ in \eqref{eq:HpG0} by the principal symbol of
$ P $. Examples of operators satisfying our assumptions are given in 
Appendix B (see also \cite{DyZw19} and \cite{ZTao}). 

We denote by $ \Delta = \sum_{j=1}^n \partial_{x_j}^2 $ the usual Laplacian on $ \TT^n $
and state a precise version of our main result:
\begin{theo}
\label{t:viscoscity}
Suppose that $ P $ is given by \eqref{eq:defP0} with $ p $
satisfying \eqref{e:SymbolAnalyticity0} and \eqref{eq:HpG0}. 
Then there exist an open neighbourhood of $ 0 $ in $ \CC $, $ U $ and 
a set 
$$ \mathscr R ( P ) \subset \{ \Im \omega \leq 0 \} \cap U  $$ 
such that
for every $ K \Subset U $, $ \mathscr R ( P ) \cap K $ is discrete, and 
\begin{equation}
\label{eq:t1}  \spec_{L^2 } ( P + i \nu \Delta ) 
\longrightarrow \mathscr R ( P )  , \ \ \nu \to 0 + ,
\end{equation}
uniformly on  $K $. 
\end{theo}

Numerical illustrations of this theorem are presented in Appendix B.

Another way to state the theorem is to say that $ \mathscr R ( P )  = \{ \omega_j \}
_{ j = 1}^N  $ (where $N = \infty $ is allowed) and 
$ \spec_{L^2} ( P + i \nu \Delta ) = \{ \omega_j ( \nu ) \}_{j=1}^\infty $
then (after suitable re-ordering)
\[ \omega_j ( \nu ) \to \omega_j , \ \ \nu \to 0 + , \]
 uniformly on compact sets and with agreement of multiplicities. In fact, the proof gives a more precise statement implying smoothness of 
 projectors acting on spaces $ \mc{X} $ of Theorem \ref{t:meromorphy} -- see \cite[Proposition 5.3]{DZ3}. Since the statement is essentially the same we do not reproduce it here.
 
The Laplacian $ \Delta $ can be replaced by any second order (or any order) 
elliptic differential operator with analytic coefficients and the set 
$ \mathscr R ( P ) $ is independent of that choice. The next theorem shows that  $\mathscr{R}(P)$ is defined intrinsically for operators satisfying our assumptions:
\begin{theo}
\label{t:meromorphy}
Suppose that $ P $ satisfies the assumptions of Theorem \ref{t:viscoscity}
and $ U $ is the open set presented there. Then 
there exists a Hilbert space $ \mc{ X } $ such that for $ \omega \in U $, 
\[  P - \omega : \mc{X} \to \mc{X} \ \text{ is a Fredholm operator }, \]
and 
$
\mc{R} (\omega):=(P-\omega)^{-1}:\mc{X}\to \mc{X}, 
$
forms a meromorphic family of operators with poles
of finite rank. The set of these poles in $ U $ is the set $ \mathscr R ( P ) $ in 
Theorem \ref{t:viscoscity} (with inclusion according to multiplicity). 
Moreover, 
$$ \mathscr R ( P ) \cap \RR = \specppL(P)\cap U . $$
\end{theo}

The space $ \mc{X} = H_\Lambda $ is defined in \S \ref{s:dFBI} and
for some $ \delta > 0 $, 
\[   \mathscr A_\delta \subset \mc{X} \subset \mathscr A_{-\delta } , \ \ 
 \]
where, for $ s \in \RR $, the spaces $ \mathscr A_s $ is given by formal Fourier series with Fourier coefficients bounded by $ e^{ - |n|s } $, $ n \in \ZZ^n $. Hence $ \mc{X} $ contains the space of analytic functions extending to a sufficiently large complex neighbourhood of $ \TT^n $ and is contained in the dual of such space -- 
see \eqref{e:analytic} for precise definitions. 

We briefly recall similar results in different settings. Dyatlov--Zworski 
\cite{DZ3} showed that if $ X $ is the generator of an Anosov flow on a compact manifold and $ Q $ is a self-adjoint second order elliptic  operator{,} then the eigenvalues of $ X + i \nu Q $ converge to the 
{\em Pollicott--Ruelle resonances} of the Anosov flow. These resonances
appear in expansions of correlations -- see \cite{DZ3} for a discussion and references.
Drouot \cite{Dr2} proved an analogue of this result for {\em kinetic Brownian motion} in which $ X $ is a generator of an Anosov geodesic flow and $ Q $ is the ``spherical Laplacian" on the fibers. Dang--Rivi\`ere
\cite{dar} showed that for Morse--Smale gradient flows, the eigenvalues of 
$ \mathcal L_{\nabla_g f } + i \nu \Delta_g $ (which agree with the eigenvalues of the {\em Witten Laplacian}) converge to the Pollicott-Ruelle resonance of the gradient flow. That generalized a result
of Frenkel--Losev--Nekrasov \cite{FLN} who, motivated by quantum field theory, considered the case of the height function on the sphere.

The {\em complex absorbing potential} method (see \cite[\S 4.9]{dizzy} for 
a description and references) is also related to viscosity limits: to obtain discrete complex spectrum a complex potential, say $ - i \epsilon |x|^2 $, is added to a 
Schr\"odinger operator. In cases where scattering resonances can be defined, the spectrum of this new operator converges to the resonances -- see \cite{Xi}, \cite{z18}.

The essential ingredient in  the proof of Theorems \ref{t:viscoscity} and \ref{t:meromorphy} is the theory of complex microlocal deformations
inspired by works of Sj\"ostrand \cite{sam},\cite{Sj96} and Helffer--Sj\"ostrand \cite{HS}. The starting object there is an 
FBI transform. In our case we need an FBI transform which respects the analytic structure of the underlying compact analytic manifold.
Hence{,} if $ M $ is a compact analytic manifold{,} we define (using a measure
on $ M$ coming from a real analytic metric)
\begin{equation}
\label{eq:genT}  T u ( x, \xi , h ) := h^{-\frac{3n}2} \int_M K ( x, \xi  , y , h ) u ( y ) dy , 
\end{equation}
where 
\[   ( x, \xi , y ) \to K ( x, \xi , y , h ) \]
is holomorphic 
in a fixed complex conic neighbourhood of $ T^* M \times M $ and{,}
uniformly in that neighbourhood{,}
\begin{equation}
\label{eq:defKex}
\begin{gathered}  K ( x, \xi , y , h ) = \chi ( x, y) a ( x, \xi, y , h ) e^{ \frac ih \varphi ( x, \xi, {y}, h ) } + \mathcal O ( e^{ - \langle  \Re \xi \rangle/ Ch } ) , \\
\varphi ( x, \xi, x ) =0 , \ \ d_y \varphi ( x, \xi, y )|_{y = x  } = - \xi, \ \  \Im d_y^2 \varphi ( x, \xi, y )|_{y = x   } \sim \langle \Re \xi \rangle I.
\end{gathered}
\end{equation}
Denoting by $ \widetilde M $ a complex neighbourhood of $ M$, $ \chi $ satisfies
\begin{gather*}  \chi \in \CI ( U ), \ \
\chi|_V  \equiv 1 , %\ \ \bar \partial_x \bar \partial_y \chi = 
%\mathcal O ( |\Im x|^\infty + |\Im y |^\infty ) , 
\ \
V \Subset U \subset \widetilde M \times \widetilde M \text{ are small neighbourhoods of $ \Delta ( \widetilde M ) $,} 
\end{gather*}
and $ a $ is analytic symbol of order $ n/4 $ in $ \xi $. (Here $ 
\Delta ( \widetilde M ) $ denotes the diagonal $\{ ( x, x ) : 
x \in \widetilde M \} $.)

Existence of such kernels $ K $ can be obtained by choosing a real analytic metric with exponential map $ T_x M \ni ( x, v ) \mapsto 
\exp_x ( v ) $, and then  putting
\[ \varphi ( x, \xi , y ) = - \xi ( \exp^{-1}_x ( y ) ) + 
\tfrac i 2 \langle \xi \rangle d ( x, y )^2. \]
We can then solve the $ \bar \partial $-equation with the right hand side given by $ \bar \partial_{x,y} $ applied to the first term on the right hand side of \eqref{eq:defKex}.

In this paper, in view of our applications and for the sake of clarity, we consider an explicit $ K ( x, \xi, y , h ) $ available in the case of tori, $ \TT^n := \RR^n/ ( 2 \pi \ZZ)^n $:
\begin{equation}
\label{e:exactFBI}  K ( x, \xi , y , h ) = 
c_n \langle \xi \rangle^{ \frac n 4} \sum_{ k \in \ZZ} 
e^{\frac i h ( \langle  x - y - 2 \pi k , \xi \rangle + 
\frac i 2 \langle \xi \rangle ( x - y - 2 \pi k )^2 ) } . \end{equation}
Although the analysis works in the more general setting of analytic compact manifolds and FBI transforms satisfying \eqref{eq:genT} and \eqref{eq:defKex}{,} we can avoid additional complications such as the study of analytic symbols when the inverse of $ T $ is not exact (see 
Proposition \ref{p:invert}) and of operators annihilating $ T u $
which do not commute exactly (see Proposition \ref{p:ZLa}) {by using~\eqref{e:exactFBI}}. One motivation for this project was to present the theory of 
{\em exponential weights} which are not compactly supported -- see \S \ref{s:dFBI}.  The expository 
article \cite{GZ} is intended as an introduction to these methods in the 
simpler setting of compactly supported weights, see also Martinez \cite{M} and Nakamura \cite{N} for a very clear approach to compactly supported weights in 
$ \RR^n $ (or more generally weights $ \psi $ satisfying $ 
\partial^\alpha \psi \in L^\infty $ for $ |\alpha| > 0 $).

In an independent development Guedes Bonthonneau--J\'ez\'equel \cite{Guj} 
presented a similar theory in a more general setting of Gevrey regularity
and arbitrary real analytic compact manifolds. 
Their motivation came 
from microlocal study of dynamical zeta functions and trace formulas for
Anosov flows, see \cite{DZ16},\cite{Je} and references given there.

The paper is organized as follows:

\begin{itemize}

\item In \S \ref{s:FBItor} we define an FBI transform, $ T $ on tori and construct
its exact left inverse $ S$. The FBI transform takes function{s} on $ \TT^n $ to  functions on $ T^* \TT^n $.

\item The geometry of complex deformation and their relation to exponential weights is reviewed in \S \ref{s:geom}. The 
complex deformation of our FBI transform, $ T_\Lambda $,  is then investigated in 
\S \ref{s:dFBI} where the space $ \mc{X} = H_\Lambda $ is {also} defined. 
Here, $ \Lambda $ is a complex deformation of $ T^* \TT^n $ associated to 
$ G $ in \eqref{eq:HpG0} using \eqref{e:LambdaDef}.

\item \S \ref{s:asy} is motivated by the study of Bergman kernels by Boutet de Monvel--Sj\"ostrand \cite{BS}, \cite{Sj96} and of Toeplitz operators by 
Boutet de Monvel--Guillemin \cite{BG}: we construct a parametrix for the
orthogonal projector onto the image of $ \mc{X} $ under $ T_\Lambda $. 

\item The action of pseudo-differential {operators} of the form \eqref{eq:defP0} 
on the space $ \mc{X} $ is described in \S \ref{s:deforpsi}. We also present {the}
compactness and trace class properties needed in {our} proofs of the Fredholm property and of the viscosity limit {for $P$ and $P+i\nu\Delta$}.

\item Finally, \S \ref{s:deforpsi} is devoted to the proofs of Theorems 
\ref{t:viscoscity} and \ref{t:meromorphy}. 

\item Appendix A reviews some aspects almost analytic machinery of Melin--Sj\"ostrand \cite{mess}, see also \cite[\S 5]{GZ}. In Appendix B we
discuss the (very) special case of escape functions which are linear in 
$ \xi $. In that case we can use an analogue of the method of complex scaling -- see \cite[\S\S 4.5,4.7]{dizzy} and references given there. 
{This} method lends itself to numerical experimentation and some  results of that are presented in Appendix B as well.

\end{itemize}

%\noindent{\bf List of Notation}

\medskip\noindent\textbf{Acknowledgements.}
We would like to thank Semyon Dyatlov for many enlightening discussions and 
Johannes Sj\"ostrand for helpful comments on the first version of \cite{GZ}. 
Partial support by the National Science Foundation grants
DMS-1900434 and DMS-1502661 (JG) and 
 DMS-1500852 (MZ) is also
gratefully acknowledged.

\section{A semiclassical FBI transform on $\mathbb{T}^n=\RR^n/2\pi \mathbb{Z}^n$}
\label{s:FBItor}

We start by defining an FBI transform on $\mathbb{T}^n$ {which respects the} \emph{real analytic} structure of $\mathbb{T}^n$ and is invertible with error
exponentially small in $ h $ and in frequency.

As stated in \S \ref{s:intr} we achieve this with the following transform:
\begin{equation}
\label{e:periodized}
\begin{gathered}
Tu(x,\xi):=  h^{-\frac{3n}{4}}\int_{\mathbb{T}^n} \sum_{k\in \mathbb{Z}^n} e^{\frac{i}{h}\varphi(x,y-2\pi k,\xi)}\langle \xi\rangle^{\frac{n}{4}}u(y)dy, \ \ u\in \CI (\mathbb{T}^n), \\
\varphi(x,\xi,y):=\langle x-y,\xi\rangle +\tfrac{i}{2}\langle \xi\rangle (x-y)^2.\end{gathered}
\end{equation}
This sum is rapidly convergent since $\Im \varphi \geq \langle \xi\rangle|x-y|^2/2$.

\noindent
{\bf Remark.} 
As already emphasized in \S \ref{s:intr}, 
the crucial feature of $T$ is the structure of its integral kernel, $K(x,\xi,y)$, which is analytic in all variables and is given by 
\begin{gather*}
e^{\frac{i}{h}\varphi(x,\xi,y)}a(x,\xi,y)\chi(d(x,y))+ \mathcal O(e^{-\langle \xi\rangle/h}), \\
 \varphi(x,\xi,y)=\langle \exp_{y}^{-1}(x),\xi\rangle +\tfrac{i}{2}\langle \xi\rangle d(x,y)^2,
\end{gather*}
where $a$ is a classical analytic symbol and $\chi\in C_c^\infty(\mathbb{R})$ is supported in a small neighbourhood of $0$ and is equal to $1$ near $ 0$.

 Extending $u$ to $\RR^n$ as a $2\pi \mathbb{Z}^n$ periodic function, we observe that 
$$
Tu(x,\xi)=h^{-\frac{3n}{4}}\int_{\RR^n} e^{\frac{i}{h}\varphi(x-y,\xi)}\langle \xi\rangle^{\frac{n}{4}}u(y)dy
$$
and, moreover, $Tu(x,\xi)$ is $2\pi \mathbb{Z}^n$ periodic in $x$.

\begin{lemm}
The operator $T:C^\infty(\mathbb{T}^n)\to C^\infty(T^*\mathbb{T}^n)$ extends to an operator 
$$ T: L^2(\mathbb{T}^n)\to L^2(T^*\mathbb{T}^n), \ \ \|T \|_{
L^2(\mathbb{T}^n)\to L^2(T^*\mathbb{T}^n) } \leq C , $$
with $ C $ independent of $ h $.
\end{lemm}
\begin{proof}
Suppose that $v\in C_c^\infty(T^*\mathbb{T}^n)$. We extend $v$ periodically in $ x $ and consider
$$
TT^*v(x,\xi)=h^{-\frac{3n}{2}}\langle \xi\rangle^{\frac{n}{4}}\int_{\RR^{3n}}\langle \eta\rangle^{\frac{n}{4}} e^{\frac{i}{h}
\Phi } v(y,\eta)dyd\eta dw,
$$
where
\[ \Phi:= \langle x-w,\xi\rangle +\langle w-y,\eta\rangle +\tfrac{i}{2}(\langle \xi\rangle (x-w)^2+\langle \eta\rangle(y-w)^2) . \]
Completing the square and integrating in $w$, we then obtain
$$
TT^*v(x,\xi)=h^{-n}\int_{T^* \TT^n} \frac{\langle \xi\rangle^{\frac{n}{4}}\langle \eta\rangle^{\frac{n}{4}}}{(\langle \xi\rangle+\langle \eta\rangle)^{\frac{n}{2}}} \sum_{ k \in \ZZ^n}
e^{\frac i h \Psi ( x - y + 2 \pi k , \xi , \eta)  }v(y,\eta)dyd\eta .
$$
where
$$
\Psi ( z, \xi, \eta)  :=\frac{i}{2}\frac{(\xi-\eta)^2}{\langle \xi\rangle +\langle \eta\rangle}+\frac{i}{2}\frac{ \langle \eta\rangle \langle \xi\rangle z^2 }{\langle \xi\rangle +\langle \eta\rangle}+\frac{\langle \eta\rangle \xi+\langle \xi\rangle \eta}{\langle \xi\rangle +\langle \eta\rangle}\cdot z .
$$
Schur's test for boundedness together with density of $C_c^\infty(T^*\mathbb{T}^n)$ in $L^2(T^*\mathbb{T}^n)$ complete the proof of the lemma.
\end{proof}

Our next goal is to find an inverse for $T$. To do this, we define
\begin{equation}
\label{eq:defS}
\begin{gathered}
Sv(y)=h^{-\frac{3n}{4}}\int_{T^*\mathbb{T}^n}\sum_{k\in \mathbb{Z}^n}e^{-\frac{i}{h}\varphi^*(x-2\pi k,y,\xi)}b(x-y - 2\pi k,\xi)v(x,\xi)dxd\xi\\
\varphi^*(x,\xi,y)= \bar{\varphi}(x,\xi,y).
\end{gathered}
\end{equation}
Then, as before, extending $v$ periodically in $x$, 
$$
Sv(y)=h^{-\frac{3n}{4}}\int_{T^*\RR^n}e^{\frac{i}{h}\varphi^*(x,y,\xi)}b(x-y,\xi)v(x,\xi)dxd\xi.
$$
We then have
\begin{prop}
\label{p:invert}
Putting 
\begin{equation}
\label{eq:defb} b ( w, \xi ) = 2^{\frac n 2} (2\pi)^{-\frac{3n}{2}}\langle \xi\rangle^{\frac{n}{4}}(1+\tfrac{i}{2} \langle w, {\xi}/{\langle \xi\rangle}\rangle), \end{equation}
in 
\eqref{eq:defS} gives 
\begin{equation}
\label{e:ST1}
ST u = u , \ \  u \in L^2 ( \RR^n ) . 
\end{equation}
\end{prop}
\begin{proof}
Using definition \eqref{e:periodized} and \eqref{eq:defS} 
we have
\begin{equation}
\label{e:ST}
\begin{aligned}
STu&=h^{-\frac{3n}{2}}\int_{\RR^n\times \RR^n\times \RR^n} e^{\frac{i}{h}(\langle x-y,\xi\rangle +\frac{i}{2}\langle \xi\rangle (z-y)^2 +\frac{i}{2}\langle\xi\rangle (x-z)^2)}\langle\xi\rangle^{\frac{n}{4}}b(x-z,\xi)u(y)dydzd\xi\\
&= h^{-n}\int_{\RR^n\times \RR^n} e^{\frac{i}{h}\langle x-y,\xi\rangle -\frac{\langle \xi\rangle}{4h}(x-y)^2 }  a(x, y,\xi;h)u(y)dyd\xi.
\end{aligned}
\end{equation}
For our choice of $ b $ we have 
\begin{equation}
\label{eq:axyxi} 
\begin{split}
 a(x, y,\xi;h)&= h^{-\frac{n}{2}}e^{ \frac{\langle \xi\rangle}{4h} ( x- y )^2  }\int_{\RR^n}e^{-\frac{\langle \xi\rangle }{2h}[(x-z)^2+(z-y)^2]}\langle\xi\rangle^{\frac{n}{4}}b(x-z,\xi)dz\\
&= h^{-\frac{n}{2}} e^{\frac{\langle \xi\rangle}{4h}(x-y)^2 }\int_{\RR^n}e^{-\frac{\langle \xi\rangle }{2h}[(x-w-y)^2+w^2]}\langle\xi\rangle^{\frac{n}{4}}b(x-w-y,\xi )dw \\
& = h^{-\frac{n}{2}} \langle\xi\rangle^{\frac{n}{4}} \int_{\RR^n}e^{-\frac{\langle \xi\rangle }{h} v^2 }b \left(\tfrac 12 ( x- y ) - v,\xi , h  \right) d{v} \\
& = (2\pi)^{-n}(1+\tfrac{i}{4} \langle x-y, {\xi}/{\langle \xi\rangle}\rangle) . 
\end{split}
\end{equation}
The proof is now concluded using \eqref{eq:Leb} below. \end{proof}

For the reader's convenience we include the derivation of Lebeau's inversion formula used in the proof of Proposition \ref{p:invert}  (see
\cite[(9.6.7)]{H1}): 
\begin{lemm}
\label{l:identity}
For $u\in C_c^\infty ( \RR^n ) $, 
\begin{equation}
\label{eq:Leb}
u(x)=(2\pi h)^{-n}\int_{\RR^{2n}} e^{\frac{i}{h}\langle x-y,\xi+ i a \langle \xi\rangle (x-y)\rangle}  (1+ i a  \langle x-y, {\xi}/{\langle \xi\rangle}\rangle)   u(y)dy d\xi, \ \ a > 0 . 
\end{equation}
\end{lemm}
\begin{proof}
For $u\in C_c^\infty ( \RR^n ) $ the Fourier inversion formula gives
$$
u(x)=(2\pi h)^{-n}\lim_{\epsilon\to 0+}\int e^{\frac{i}{h}(\langle x-y,\xi\rangle +i\epsilon \langle \xi\rangle )}u(y)dy d\xi,
$$
where the integral converges absolutely for $ \epsilon > 0 $.
We deform the contour of integration in $ \xi $ to 
$ \Gamma_a (x,y)$ given by 
$$
\xi \mapsto \eta := \xi + a {i}  \langle \xi\rangle (x-y), \ \ 
\xi \in \RR^n , \ \ 0 < a \ll 1 .
$$
This deformation is justified since on $\Gamma$, 
$$
\Im \langle x-y,\eta\rangle \geq c\langle \eta\rangle (x-y)^2.
$$
and for $ a $ sufficiently small, $ \langle \eta \rangle := 
( 1 + \eta^2 ) ^{\frac12 } $ has an analytic branch with positive real part.
In particular, we have, using that $d\langle \xi\rangle =\sum_i \langle \xi\rangle^{-1}\xi_id\xi_i.$
\begin{align*}
u(x)&=(2\pi h)^{-n}\lim_{\epsilon \to 0}\int_{\Gamma_a}  \int_{\RR^n}e^{\frac{i}{h}(\langle x-y,\eta\rangle +i\epsilon \langle \eta \rangle ) }u(y)dy d\eta_1\wedge d\eta_2\wedge \dots\wedge d\eta_n\\
&=(2\pi h)^{-n}\lim_{\epsilon\to 0}\int_{\RR^{2n}} e^{\frac{i}{h}(\langle x-y,\xi+ i a (x-y)\rangle+i\epsilon \langle{\eta}\rangle )} \det (\eta_{\xi})\,  u(y)dy d\xi\\
&=(2\pi h)^{-n}\int_{\RR^{2n}} e^{\frac{i}{h} \langle x-y,\xi+ i a (x-y)\rangle} (1+i a  \langle x-y, {\xi}/{\langle\xi\rangle}\rangle)   u(y)dy d\xi .
\end{align*}
Since the right hand side is analytic in $ \{ a \in \CC :  \Re a >0 \} $ it follows that
the formula remains valid for all $ a > 0 $.
\end{proof}

\section{Geometry of complex deformations}
\label{s:geom}

Following \cite{HS} and \cite{Sj96} we will study the FBI transform~\eqref{e:periodized} when $T^*\mathbb{T}^n$ is replaced by {an} $I$-Lagrangian $\mathbb{R}$-symplectic manifold submanifold of 
\[ \widetilde{T^*\mathbb{T}^n} :=\left\{(z,\zeta)\mid z\in \mathbb{C}^n/2\pi \mathbb{Z}^n,\,\zeta\in \mathbb{C}^n \right\} \simeq T^*( \mathbb{C}^n/2\pi \mathbb{Z}^n) . \]
We recall that $ \widetilde{T^*\mathbb{T}^n}  $ is equipped with the
complex symplectic form 
\[ \sigma := d\zeta \wedge dz := \sum_{ j=1}^{ n } d \zeta_j \wedge dz_j = 
d ( \zeta \cdot dz). \]
For a real $2n$-dimensional submanifold of $ \widetilde{T^*\mathbb{T}^n}  $, $ \Lambda$,  we 
say 
\[  \text{$\Lambda $ is $I$-Lagrangian} \ \Longleftrightarrow \
\Im(\sigma|_\Lambda ) \equiv 0 , \]
and that
\[  \text{$\Lambda $ is $\RR$-symplectic} \ \Longleftrightarrow \
\text{ $\Re (\sigma|_\Lambda) $ is non-degenerate.} \]

The specific submanifolds used here are given as follows.
For a function $G(x,\xi)\in \CI (T^*\mathbb{T}^n; \RR )$, assume that
for some sufficiently small $ \epsilon_0 $ (to be chosen in the constructions below), 
\begin{equation}
\label{e:controlDeformation}
\sup_{ |\alpha| + |\beta| \leq 2 } \langle \xi\rangle^{-1+|\beta|} | \partial_x^\alpha \partial_\xi ^\beta G(x,\xi) | \leq \epsilon_0, \ \ 
 | \partial_x^\alpha \partial_\xi ^\beta G(x,\xi) | \leq C_{\alpha \beta} \langle \xi\rangle^{1-|\beta|}.
\end{equation}
(The second condition merely states that $ G \in S^1 ( T^* \TT^n ) $ 
in the standard notation of \cite{H3}.) 
We then define 
\begin{equation}
\label{e:LambdaDef}
\begin{gathered}
\Lambda:=\{(x+iG_\xi(x,\xi),\xi-iG_x(x,\xi))\mid (x,\xi)\in T^*\mathbb{T}^n\}\subset \widetilde{T^*\mathbb{T}^n} .
\end{gathered}
\end{equation}
By consider{ing} $ G ( x, \xi ) $ as a period{ic} function of $ x $, we can also think of $ \Lambda $ as a submanifold of $ T^* \CC^n $. 

A submanifold given by \eqref{e:LambdaDef} is always $I$-Lagrangian: 
\[  \begin{split} \zeta \cdot dz|_\Lambda & = ( \xi - i G_x ) \cdot d (x + i G_\xi ) \\
& = 
\xi \cdot dx + G_x \cdot G_{\xi \xi} d \xi + G_x \cdot G_{ \xi x } dx 
+ i ( - G_x \cdot dx + {\xi \cdot dG_\xi}   ) ,
\end{split}  \]
and 
\[  d ( - G_x \cdot dx + {\xi\cdot dG_\xi}) = 
d \xi \wedge G_{\xi x } dx + dx \wedge G_{x \xi } d \xi = 0 . \]
The smallness of $ \epsilon_0 $ enters for the first time in guaranteeing that $ \Lambda $ is $ \RR$-symplectic:
\[ \begin{split}  d ( \xi \cdot dx + G_x \cdot G_{\xi \xi} d \xi + G_x \cdot G_{ \xi x } dx  )  & = 
d \xi \wedge d x + G_{xx} dx \wedge G_{\xi \xi } d \xi + 
G_{x \xi } d \xi \wedge G_{ \xi \xi } d\xi \\
& \ \ \ \ + G_{ x \xi} d \xi \wedge 
G_{ \xi x } dx + {G_{xx}} dx \wedge G_{\xi x } dx .\\ 
\end{split}
\] 
The left hand side is non-degenerate if $ \epsilon_0 $ in \eqref{e:controlDeformation} is small enough.

Since $ \Im \zeta \cdot dz |_ \Lambda $ is closed, there exists 
$ H \in \CI ( \Lambda; \RR ) $ such
\begin{equation}
\label{eq:defH}  d H  = - \Im \zeta d z|_{ \Lambda } ,
\end{equation}
with the normalization $ H \equiv 0 $ when $ G \equiv 0 $. 
Using the parametrization \eqref{e:LambdaDef} we have the following 
explicit expression for 
\begin{equation}
\label{eq:Hpedest}
 H ( x, \xi ) = G ( x, \xi ) - \xi \cdot G_\xi ( x, \xi)  .\end{equation}

Any $I$-Lagrangian and $\RR$-symplectic manifold is automatically maximally totally real in the sense that
\[  T_\rho \Lambda \cap i T_\rho \Lambda = \{0 \}, \ \ \rho \in \Lambda. \]
In fact, suppose that $ X , i X  \in T_\rho \Lambda$, then for all
$ Y \in T_\rho \Lambda $, $ \Re \sigma ( Y , i X ) = - \Im \sigma ( Y , X ) = 0 $, as $ \Im \sigma $ vanishes on $ T_\rho \Lambda $. But then the 
non-degenerary of $ \Re \sigma $ shows that $ X = 0 $.  

The real symplectic form on $ \Lambda $ defines a natural 
volume form $  d m ( \alpha) = ( \sigma |_{\Lambda})^n /n!$. If $ ( z, \zeta ) = ( x + i G_\xi , \xi - i G_x ) $
we sometimes write 
\begin{equation}
\label{eq:dmLa} d m_\Lambda ( \alpha ) = dz d\zeta = d \alpha , \ \ 
\alpha = ( z, \zeta ) \in \Lambda , \ \  \beta = \Re \alpha . 
\end{equation}

Let $ \Gamma $ be a small conic connected neighbourhood of $ T^* \TT^n $
in $ T^* \CC^n/{\ZZ^n} $ and let $ \widetilde G ( z, \zeta ) $ be 
a symbolic almost analytic extension of $ G ( x, \xi ) $ supported in 
$ \Gamma $:
\begin{gather*}   | \bar \partial_z \widetilde G ( z, \zeta ) | + 
\langle \Re \zeta \rangle \bar\partial_\zeta \widetilde G ( z, \zeta ) |
\leq \langle \Re \zeta \rangle \mathcal O ( | \Im z |^\infty + 
| \Im \zeta/ \langle \Re \zeta \rangle |^\infty ), \\
\sup_{ |\alpha| + |\beta| \leq 2 } | \partial_z^\alpha \partial_\zeta ^\beta \widetilde G(z ,\zeta) | \leq C \epsilon_0\langle \Re \zeta \rangle^{1-|\beta|}, \ \ 
 | \partial_z^\alpha \partial_\zeta ^\beta G(x,\xi) | \leq C_{\alpha \beta} \langle \Re \zeta \rangle^{1-|\beta|}, 
\end{gather*}
for $ ( z, \zeta ) \in \Gamma $ -- 
see \cite[Theorem 1.3]{mess} (for a brief review of basic concepts of 
almost analytic machinery see Appendix A). 

We use an almost analytic change of variables in $ \Gamma $ to 
identify the totally real submanifold $ \Lambda $ with $ T^* \RR^n $ (on 
$ \Lambda $ the differentials of that transformation are complex linear): it is the inverse of the map
\begin{equation}
\label{eq:paraLaG} F:  ( z , \zeta ) \mapsto  ( w, \omega ) := ( z + i \widetilde G_\zeta ( z, \zeta ) , \zeta - i \widetilde G_z( z, \zeta ) ) , 
\end{equation}
 Using this identification we define
\begin{equation}
\label{eq:conjLa}  C_\Lambda ( w  , \omega ) =  F (  \overline { F^{-1} ( w,
\omega) } ) , \ ( w, \omega ) \in \Gamma , \ \  C_\Lambda|_\Lambda = I_\Lambda. 
\end{equation}

 We also denote by $ \sigma_\Lambda $ the almost analytic extension of $  \sigma|_{{ \Lambda}} $ to 
$ \Gamma $. 

\noindent
{\bf Notation:} The different identifications lead to potentially confusing 
notational issues. We will typically use coordinates 
\[  \alpha = ( \alpha_x, \alpha_\xi ) = ( x, \xi ) \mapsto \beta = ( 
\beta_x , \beta_\xi ) = 
( x +  i G_\xi ( x, \xi) , \xi - i G_x ( x, \xi ) ) \in \Lambda \]
and consider the complexification of $ \alpha $ using the identification
\eqref{eq:paraLaG}. In that case for $ \alpha \in \Gamma $,
$ \bar \alpha $ denotes $ C_\Lambda ( \alpha ) $. It is {\em not} 
given by taking $ ( z, \zeta ) \mapsto ( \bar z, \bar \zeta ) $ in the 
original coordinates on $ \widetilde T^* \TT^n $ (for one thing, it would not be the identity on $ \Lambda $). Sometimes it is convenient to 
use $ \beta \in \Lambda $ as the variable in formulae and integrations.
The choice should be clear from the context.

\section{Complex deformations of the FBI transform}
\label{s:dFBI} 

For $ \Lambda $ given by \eqref{e:LambdaDef} 
we define an operator $T_\Lambda$ by prescribing its 
Schwartz kernel:
$$
T_\Lambda(z,\zeta, y) :=T((z,\zeta),y)|_{(z,\zeta)\in \Lambda}.
$$
We then define an operator $S_{\Lambda}$  by
$$
S_{\Lambda }v(y):=\int_{\Lambda} S(x, \beta )v( \beta ) d \beta, 
\ \ \beta = ( z, \zeta ) \in \Lambda, \ \ d \beta = dz\wedge d\zeta|_{\Lambda},
$$
where $S(x,z,\zeta)$ is the kernel of the operator $S$:
\begin{equation}
\label{eq:defkerS}
S ( x, z, \zeta ) := 
h^{-\frac{3n}{4}}\sum_{k\in \mathbb{Z}^n}e^{-\frac{i}{h}\varphi^*(z -2\pi k,x ,\zeta)}b(z - x - 2\pi k,{\zeta}), 
\end{equation}
with $ b $ given in \eqref{eq:defb}.

 Note that {if we parametrize $\Lambda$ as in~\eqref{e:LambdaDef} with $\alpha=(x,\xi)$} we may also write 
$$
S_{\Lambda}v(y) :=\int_{T^*\mathbb{T}^n} S_{\Lambda}(x,z(\alpha ),\zeta(\alpha ))v(\alpha)dm _{\Lambda}(\alpha ) , 
$$
where $dz\wedge d\zeta|_{\Lambda}=d m_\Lambda (\alpha) $. Finally, we sometimes write $\alpha_x=z(\alpha)$ and $\alpha_\xi=\zeta(\alpha)$.

In order to make sense of the composition $S_\Lambda T_\Lambda$, we start by analyzing $T_\Lambda$ on a space of analytic functions on $\mathbb{T}^n$.  
For $ \delta \geq 0 $ let 
\begin{equation}
\label{e:analytic}
\begin{gathered} 
\mathscr{A}_\delta = \{ u \in L^2 ( \TT^n ) : \| u \|_{ \mathscr A_\delta}^2  := 
\sum_{ n \in \ZZ^n } | \widehat u ( n ) |^2 e^{ 4 |n| \delta} < \infty \}, \\
\widehat u ( n ) := \frac{1 }{ (2 \pi)^n } \int_{\TT^n} u ( x ) e^{- i \langle x, n \rangle} d x . 
\end{gathered}
\end{equation}
Let also $\mathscr{A}_{-\delta}$ denote the dual space of $\mathscr{A}_\delta$. Note that $\mathscr{A}_{-\delta}$  is a space of hyperfunctions but on tori it can be 
identified with formal Fourier series with coefficients satisfying 
\[  \sum_{ n \in \ZZ^n } | \widehat u ( n ) |^2 e^{ - 4 |n| \delta} < \infty   . \]
(In that case $ \widehat u (n ) $ can be defined using the pairing of the hyperfunction $ u $ with the analytic function $ x \mapsto e^{ - i \langle x , n \rangle }/(2 \pi)^n  $.)
We note that $ u \in \mathscr A_\delta $ extends to a (periodic) holomorphic function
in $ |\Im z  | <  {2}\delta $ and (by the Fourier inversion formula and the Cauchy--Schwartz inequality),
\begin{equation}
\label{eq:boundonu} \forall \, \delta' < {2}\delta \,\, \exists \, C \, \text{ such that for } \ 
u \in \mathscr A_\delta , \  \ \sup_{ |\Im z | < \delta' } 
| u ( z ) | \leq C \| u \|_{ \mathscr A_\delta } .\end{equation}

\begin{lemm}
\label{l:expDecay}
Define 
\[ \Omega_\delta := 
\{ ( z,  \zeta)  \in T^* (\CC^n/2 \pi \ZZ^n)  : |\Im \zeta|\leq  \delta \langle \Re\zeta \rangle,\,|\Im z|\leq \delta,\,|\Re \zeta| \geq 1 \} .\]
There exist $ c_0 , \delta_0 > 0  $ such that for $ ( z, \zeta ) \in \Omega_{ \delta}$ and $ 0 < \delta < \delta_0 $, 
%$\delta_0>0$ and $c_0>0$ small enough such that for $0<\delta<\delta_0$ there exists $C_0>0$ such that for $u\in \mathscr A_ \delta$, 
\begin{equation}
\label{eq:TuStu}
|T u (z,\zeta)|\leq  e^{-\delta  c_0 |\zeta|/ h}\|u\|_{\mathscr A_\delta},  %\ \ ( z, \zeta ) \in \Omega_{ \delta}, 
%$$
%and
%$$
\ \ \ \ |S^tu (z,\zeta)|\leq  e^{- c_0 \delta |\zeta|/ h}\|u\|_{\mathscr A_\delta}%,\ \ ( z, \zeta ) \in \Omega_{ \delta } 
.
\end{equation}
where $S^t$ is defined by 
\[  S^t u ( z, \zeta ) := \int_{\TT^n  } S ( y , z , 
\zeta )  u ( y ) d y \]
where the kernel $ S $ is defined in \eqref{eq:defkerS}.
\end{lemm}
\begin{proof}
Extended $u$ to a periodic function on $\RR^n$ we write
$$
Tu(z,\zeta)=h^{-\frac{3n}{4}}\int_{\RR^n}  e^{\frac{i}{h}(\langle z-y,\zeta\rangle +\frac{i}{2}\langle \zeta\rangle (z-y)^2)}\langle \zeta\rangle^{\frac{n}{4}}u(y)dy. 
$$
Since $u$ is analytic on $|\Im y|\leq \delta$, we may deform the contour in the $y$ integration to $ \Gamma(z,\zeta) $ given by 
$$
 w \mapsto y ( w) =  w + z - i\delta \frac{\Re\zeta}{\langle \Re\zeta\rangle},  \ \ w \in 
 \RR^n .
$$
Then, 
$$
Tu(z,\zeta)=h^{-\frac{3n}{4}}\int e^{\frac{i}{h}(\langle - w+i\delta\frac{\Re\zeta}{\langle \Re\zeta\rangle },\zeta\rangle +\frac{i}{2}\langle \zeta\rangle(w - i\delta\frac{\Re\zeta}{\langle \Re\zeta\rangle })^2) }\langle\zeta\rangle^{\frac{n}{4}}u\left(
%z-w-i\delta\tfrac{\Re\zeta}{|\Re\zeta|}
y ( w ) \right) dw.
$$

For $ | \Im \zeta |\leq \delta \langle \Re \zeta \rangle $, $ |\Re \zeta | \geq 1 $, 
with $ \delta $ small enough, 
\[  \Re \langle \zeta \rangle \geq \tfrac 12 |  \zeta |  , \ \ 
 | \Im \langle \zeta \rangle | \leq \tfrac 1{16} | \zeta | , \ \ 
 |\zeta | \geq \tfrac12  .\]
Hence for 
$ w \in \RR $ and $ ( z, \zeta ) \in \Omega_\delta $ the real part of 
the phase in the integral above is bounded  by 
\[  %\Re ( i \langle - w+i\delta\tfrac{\Re\zeta}{|\Re\zeta|},\zeta\rangle -\tfrac{1}{2}\langle \zeta\rangle(w - i\delta\tfrac{\Re\zeta}{|\Re\zeta|})^2 ) \leq 
- \tfrac12 \delta | \zeta| + \tfrac 1{16} \delta |w| |\zeta | - 
\tfrac14 ( |w|^2 - \delta^2 ) |\zeta | + \tfrac 1{16} \delta |\zeta| 
\leq - c_0 | \zeta| - c_0 |w|^2 , \ \ c_0 > 0 . 
\]
In view of \eqref{eq:boundonu} the integrand is then bounded by 
$ \exp ( - c_0  ( |\zeta| + |w|^2 ) /h ) \| u \|_{ \mathscr A_\delta } $
which gives the first bound in \eqref{eq:TuStu}.
The proof for $S^t$ is identical since the phase agrees with that of $T$.
\end{proof}

A natural Hilbert space on the FBI transform side is defined by the norm
$$
\|v\|_{L^2(\Lambda)}^2=\int_{\Lambda} |v ( \alpha ) |^2e^{-2H( \alpha ) /h} d \alpha .
$$
The next lemma gives boundedness of $ S_\Lambda $ and $ T_\Lambda^t $ 
on exponentially decaying functions on $ \Lambda $:
\begin{lemm}
\label{l:expDecay2}
There exist $\delta_0>0$ and $C_0>0$ big enough such that 
for $0<\e_0<\delta_0$ in \eqref{e:controlDeformation} we have 
$$ S_\Lambda: e^{-C_0\delta\langle \xi\rangle/h}L^2(\Lambda)\to \mathscr{A}_\delta,\qquad
T^*_\Lambda: e^{-C_0\delta \langle \xi\rangle/h}L^2(\Lambda)\to \mathscr{A}_\delta.
$$
for all $0<\delta<\delta_0$, where the adjoint $ T^*_\Lambda $ is defined using the $ L^2 ( \Lambda, e^{ -2 H / h } ) $ inner product.
\end{lemm}
\begin{proof}
Let $v\in e^{-c\langle \xi\rangle/h}L^2(\Lambda)$ and $|\Im y|\leq {a}\delta$. Then,
$$
 S_{\Lambda}v(y)=h^{-\frac{3n}{4}}\int_\Lambda  e^{\frac{i}{h}(\langle y-\alpha_x,\alpha_\xi\rangle +\frac{i}{2}\langle \alpha_\xi\rangle (\alpha_x-y)^2)}b(y-\alpha_x) v(\alpha)d \alpha. 
$$
where $ b $ is given in \eqref{eq:defb}. Therefore, by the Cauchy--Schwartz inequality, 
\[
| S_{\Lambda}v(y)|^2\\
\leq C h^{-\frac{3n}{2}} J( y ) \|e^{C_0\delta\langle \xi\rangle/h}v\|^2_{L^2(\Lambda)} , 
\]
where
\[ J ( y ) := 
\int_\Lambda 
 e^{-2\Im(\langle y-\alpha_x,\alpha_\xi\rangle +\frac{i}{2}\langle \alpha_\xi\rangle (\alpha_x-y)^2)/h{+}2H(\alpha)/h}\langle |y-\alpha_x|\rangle^2e^{-2C_0\delta \langle |\alpha_\xi|\rangle/h}d\alpha\] 
Writing $ \beta = \Re \alpha $ we now estimate
\[ \begin{split} - \Im \langle y - \alpha_x , \alpha_\xi \rangle & = 
\langle G_\xi -   \Im y   , \beta_\xi \rangle
- \langle \beta_x - \Re y , G_x \rangle \\
& \leq ( {a}\delta + \epsilon_0 ) | \beta_\xi| + \epsilon_0 \langle \beta_\xi \rangle  | \beta_x - \Re y | .
\end{split} \]
Similarly, 
\[ \Re ( \langle \alpha_\xi \rangle ) ( \alpha_x - y )^2 
\leq - ( 1 - C\epsilon_0 ) \langle \beta_\xi \rangle ( 
| \beta_x - \Re y |^2 - C \epsilon^2 - C {a^2}\delta^2 ) ,
\]
and $  2 H ( \alpha ) \leq C \epsilon_0 \langle \beta_\xi \rangle $ (see
\eqref{eq:Hpedest}).
Hence for $ C_0 \gg C $, the phase in $ J ( y ) $ is bounded by 
\[  - C_1 \delta \langle \beta_\xi \rangle \langle 
\Re y - \beta_x \rangle^2 , \ \ C_ 1 > 0 .\]
That proves $ S_\Lambda v $ in analytic and uniformly 
bounded in $ |\Im y | \leq {a}\delta $. In particular $ S_\Lambda v \in 
\mathscr A_\delta $.
  A similar argument applies to $T_\Lambda^*.$ %\marginpar{{Note that to get to $\mathscr{A}_\delta$ need bounds on $|\Im. y|\leq 2\delta$}}
\end{proof}

Together, Lemmas~\ref{l:expDecay} and~\ref{l:expDecay2} imply that there are $\delta_1,\delta_2>0$ such that  $S_\Lambda T_\Lambda$ is well as an operator $\mathscr{A}_{\delta_1}\to \mathscr{A}_{\delta_2}$ and as an operator $\mathscr{A}_{-\delta_2}\to \mathscr{A}_{-\delta_1}$.

We can now show that $S_{\Lambda}T_\Lambda$ is  the identity on $\mathscr{A}_{\delta}$ and $\mathscr{A}_{-\delta}$ for $\delta>0$ small enough.
\begin{prop}
There is $\delta_1>0$ such that for all $0<|\delta|<\delta_1$, $S_\Lambda$ and $T_\Lambda$ as above,
$$
S_\Lambda T_\Lambda =\Id :\mathscr{A}_{\delta}\to \mathscr{A}_{\delta}. $$
\end{prop}
\begin{proof}
Assume first that $ \delta > 0 $ and let $v\in \mathscr{A}_\delta$. Then, by Lemma~\ref{l:expDecay} for $\delta>0$ small enough, $T_\Lambda v\in e^{-c\delta  | \alpha_\xi | }L^2(\Lambda)$ and is given by 
$$
T_\Lambda v(\alpha)=\int_{\mathbb{T}^n} T_\Lambda(\alpha,y)v(y)dy.
$$
Then, again for $\delta>0$ small enough, Lemma~\ref{l:expDecay2}, shows that $S_{\Lambda}T_\Lambda v$ is well defined and given by 
\begin{equation}
\label{e:ST2}
S_\Lambda T_\Lambda v(x)=\int_{\Lambda}\int_{\mathbb{T}^n} S_\Lambda (x,\alpha)T_{\Lambda}(\alpha,y)v(y)dyd\alpha.
\end{equation}
The decay in $ | \alpha_\xi {|}$ allows a contour deformation in $\alpha$ in~\eqref{e:ST2} and then an application of Proposition~\ref{p:invert}.
This gives, 
$$
S_\Lambda T_\Lambda v(x)=\int_{T^*\mathbb{T}^n}\int_{\mathbb{T}^n} S (x,\alpha)T(\alpha,y)v(y)dyd\alpha=v(y), \ \ v \in \mathscr A_\delta. 
$$

To define $T_{\Lambda}v$ for $v\in \mathscr{A}_{-\delta}$, $ \delta > 0 $,  we note that Lemma~\ref{l:expDecay2} shows that if $w\in e^{-C\delta \langle \xi\rangle/h}L^2(\Lambda)$, then $T_{\Lambda}^* w\in \mathscr{A}_{\delta}$. Therefore, 
$$
\langle T_{\Lambda}v,w\rangle_{ L^2 ( \Lambda ) }:=\langle v,T_\Lambda^* w\rangle_{L^2 ( \TT^n ) } 
$$
is well defined and $T_\Lambda:\mathscr{A}_{-\delta}\to e^{C\delta\langle \xi\rangle/h}L^2(\Lambda)$. 

For $u\in \mathscr{A}_{c_1 \delta}$, $ c_1 \gg 1 $, $ c_1 \delta < \delta_0 $ (with $ \delta_0 $ of Lemma \ref{l:expDecay}), 
we formally have
\begin{equation}
\label{eq:STpair}
\langle S_\Lambda T_\Lambda v,u\rangle_{ L^2 ( \TT^n ) } :=\langle T_\Lambda v, S_\Lambda^* u\rangle_{ L^2 ( \Lambda ) } .
\end{equation}
Since $ S_\Lambda^* u = \overline{ S^t|_\Lambda \bar u } e^{ 2 H ( \alpha)/h } $,  and $ H ( \alpha ) \leq C \epsilon_0 \langle \Re \alpha_\xi \rangle $, 
Lemma~\ref{l:expDecay} shows that 
\[ S_\Lambda^* u\in e^{C \epsilon_0 \langle \xi \rangle /h-c_0 c_1 \delta|  \xi | {/h}} L^2(\Lambda) .
 \]
 and hence for $c_1 >0 $ large enough (and $ \delta_1 $ small enough 
 so that $ c_1 \delta_1 < \delta_0 $), the pairing on the right hand side of \eqref{eq:STpair} is well defined and
$$
\langle S_\Lambda T_\Lambda v,u\rangle=\langle v,T_\Lambda^* S_\Lambda^* u\rangle.
$$
We can now deform the contour in the the $ \alpha $ integral 
which gives 
$$
T_\Lambda^* S_\Lambda^* u(x)=\int_{\mathbb{T}^n}\int_{\Lambda} T_\Lambda(\alpha,x)S_\Lambda(y,\alpha)u(y)dy d\alpha=u(y) . $$
Hence for $ v \in \mathscr A_\delta $ and $ u \in \mathscr A_{c_1 \delta }
$, $ \langle S_{{\Lambda}} T_{{\Lambda}} v, u \rangle_{ L^2 ( \TT^n ) }
= \langle v, u \rangle_{ L^2 ( \TT^n ) } $. Since $ \mathscr A_{c_1 \delta } $, 
$ c_1 \geq 1 $ is dense in $ \mathscr A_{\delta} $, the claim follows.
\end{proof}

We now define natural spaces on which $T_\Lambda$, $S_\Lambda$ act:

\noindent
{\bf Definition}. Let $ \delta_0 $ be as in Lemma \ref{l:expDecay}. 
We define the Sobolev space of order $ t $ adapted to $ \Lambda $ as
\begin{equation}
\label{e:deformedSpace}
H^t_{\Lambda}:=\overline{\mathscr{A}_{\delta_0}}^{_{\|\cdot \|_{H^m_{\Lambda}}}},\qquad \|u\|^2_{H^t_{\Lambda}}:=\int_{\Lambda} \langle \Re \alpha_{{\xi}} \rangle^{2t}|T_{\Lambda}u(\alpha)|^2e^{-2H(\alpha) /h} d \alpha \end{equation}
where we used the notation from \eqref{eq:dmLa} and \eqref{eq:defH}.
We then have an isometry
\[  T_\Lambda : H^t_\Lambda \to \langle \xi \rangle^{-t} L^2 ( \Lambda ) , \]
where the notation on the right hand side is the shorthand for 
$ \langle \Re \alpha_\xi \rangle^{-t} $. \qed

\medskip

\noindent{\bf{Remarks:}} 1. There exists $\delta>0$ such that
$$\mathscr{A}_{\delta}\subset H_{\Lambda}^m\subset \mathscr{A}_{-\delta}.$$ The left inclusion is immediate from the definition. On the other hand, for $u\in H_\Lambda^m$,  $T_\Lambda u\in \langle \xi\rangle^m L^2(\Lambda)$ and in particular, by Lemma~\ref{l:expDecay2} ${S}_{\Lambda}T_\Lambda u\in \mathscr{A}_{-\delta}$ for some $\delta>0$. But, ${S}_{\Lambda}T_\Lambda u=u$ and hence $u\in \mathscr{A}_{-\delta}$. 

\noindent 2.{Let $\Pi_\Lambda$ denote the orthogonal projection from $L^2(\Lambda)\to T_{\Lambda}(H^0_{\Lambda})$.} The properties of $ \Pi_\Lambda $ show that
$  T_\Lambda ( H^t_ \Lambda  ) = \Pi_\Lambda ( \langle \xi \rangle^{-t} 
L^2 ( \Lambda ) ) $.

\medskip

Lemmas \ref{l:expDecay} and \ref{l:expDecay2} show that (with $ h $ dependent norms and changing $ c_0 $ to $ c_0 /2 $), 
\[  T_\Lambda : \mathscr A_\delta \to e^{ - \delta c_0 \langle \xi \rangle  } L^2 ( \Lambda ) , \ \ \ 
S_\Lambda : e^{ - \delta C_0 \langle \xi \rangle / h }  L^2 ( \Lambda ) 
\mapsto \mathscr A_\delta . \]
This means that
\begin{equation}
\label{eq:TLaS} 
T_\Lambda S_\Lambda : e^{ - \delta C_0 \langle \xi \rangle / h }
L^2 ( \Lambda ) \to e^{ - \delta c_0 \langle \xi \rangle / h}
L^2 ( \Lambda )  .
\end{equation}

\begin{prop}
\label{p:prelim-Projector}
The operator $T_\Lambda {S}_\Lambda$ in \eqref{eq:TLaS} extends to an 
operator
$$
T_\Lambda {S}_\Lambda = \mathcal O ( 1 ) : \langle \xi\rangle^mL^2(\Lambda)\to \langle \xi\rangle^mL^2(\Lambda).
$$
Moreover, there are $ k \in S^0 ( \Lambda \times \Lambda ) $,  $ \alpha, \beta \in \Lambda $, $ \chi \in \CIc ( \RR ) $ such that for all $\delta>0$, there is $\e_1>0$ such that for $G$ satisfying~\eqref{e:controlDeformation} with $\e_0<\e_1$,
\[  T_\Lambda S_\Lambda = K_\Lambda + O_N(e^{-C_\delta/h})_{\langle \xi\rangle^{N}L^2(\Lambda)\to \langle \xi\rangle^{-N}L^2(\Lambda)} , 
\]
where the Schwartz kernel of $ K_\Lambda $ is given by 
\begin{equation}
\label{e:kernel-Projector}
\begin{gathered}
K_{\Lambda}(\alpha,\beta)=h^{-n}e^{\frac{i}{h}\Psi(\alpha,\beta)}k(\alpha,\beta)\tilde{\chi}(\alpha,\beta)\\
\tilde{\chi}(\alpha,\beta):=\chi(\delta^{-1}d(\Re\alpha_x,\Re\beta_x ))\chi(\delta^{-1}\min(\langle\Re \beta_\xi\rangle,\langle \Re \alpha_\xi\rangle)^{-1}|\Re \alpha_\xi-\Re \beta_\xi|) , 
\end{gathered}
\end{equation}
and
\begin{equation}
\label{e:tempProjectorPhase}
\Psi=\frac{i}{2}\frac{(\alpha_\xi-\beta_\xi)^2}{\langle \alpha_\xi\rangle +\langle \beta_\xi\rangle}+\frac{i}{2}\frac{ \langle \beta_\xi\rangle \langle \alpha_\xi\rangle (\alpha_x-\beta_x)^2}{\langle \alpha_\xi\rangle +\langle \beta_\xi\rangle}+\frac{\langle \beta_\xi\rangle \alpha_\xi+\langle \alpha_\xi\rangle \beta_\xi}{\langle \alpha_\xi\rangle +\langle \beta_\xi\rangle}\cdot (\alpha_x-\beta_x).
\end{equation}
\end{prop}

We will prove the proposition in two lemmas which for future use are formulated in greater generality. We first study the kernel of the composition $T_{\Lambda}S_{\Lambda}$.
\begin{lemm}
\label{l:compositePhase}
Let $\Lambda_{G_1}$ and $\Lambda_{G_2}$ be given by~\eqref{e:LambdaDef} with $G_i$ satisfying~\eqref{e:controlDeformation}. Then, there are $ \chi \in \CIc ( \RR ) $ and
$ k \in S^0 ( \Lambda_{G_2} \times \Lambda_{G_1}{)} $ such that for all $\delta>0$, there is $\e_1>0$ such that for $G_1$ and $G_2$ satisfying~\eqref{e:controlDeformation} with $\e_0<\e_1$,
\[ T_{\Lambda_{G_2}}S_{\Lambda_{G_1}} = K + O_N(e^{-C_\delta/h})_{\langle \xi\rangle^{N}L^2(\Lambda_1)\to \langle \xi\rangle^{-N}L^2(\Lambda_2)}, 
\]
where the Schwartz kernel of $K $ 
is given by 
$$
h^{-n}e^{\frac{i}{h}\Psi(\alpha,\beta)}k(\alpha,\beta)\chi(\delta^{-1}d(\Re \alpha_x, \Re \beta_x) )\chi(\delta^{-1}\min(\langle\Re\alpha_\xi \rangle,\langle \Re\beta_\xi\rangle)^{-1}|\Re \alpha_\xi-\Re\beta_\xi|), $$
$ (\alpha, \beta ) \in \Lambda_{G_2} \times \Lambda_{G_1}  $ and where $\Psi$ is as in~\eqref{e:tempProjectorPhase}. 
\end{lemm}
\begin{proof}
The kernel of $T_{\Lambda_{G_2}}S_{\Lambda_{G_1}}$ (again extending everything to be periodic on $\RR^n$ and using integration with 
respect $ d \beta = (\sigma|_\Lambda )^n/n!$) is given by 
$$
h^{-\frac{3n}{2}}\int_{\mathbb{R}^n} e^{\frac{i}{h}(\varphi(\alpha,y)-\varphi^*(\beta,y))}\langle \alpha_\xi\rangle^{\frac{n}{4}}b(\beta_x-y,\beta_\xi)dy, 
$$
where $\alpha\in \Lambda_{G_1}$, $\beta\in \Lambda_{G_2}$, and $b
$ is given by \eqref{eq:defb}. To analyse it, we first  observe that for $\e_{0}$  small enough
\begin{multline*}
\Im \varphi(\alpha,y)-\varphi^*(\beta,y)\geq \frac{1}{4}|\langle 
\alpha_\xi\rangle|(|\Re (\alpha_x-y)|^2-|\Im (\alpha_x-y)|^2)\\+\frac{1}{4}|\langle \beta_\xi\rangle|(|\Re (\beta_x-y)|^2-|\Im(\beta_x-y)|^2 ) +\Im \langle \alpha_x-y,\alpha_\xi\rangle +\Im \langle y-\beta_x,\beta_\xi\rangle
\end{multline*}
Now, fix $\delta>0$, and assume that $|\Im y|\leq \delta$. Then for $\e_0\leq\delta$ in~\eqref{e:controlDeformation}, we have
\begin{multline*}
%\label{e:ImPhase1}
\Im \varphi(\alpha,y)-\varphi^*(\beta,y)\geq c|\langle 
\alpha_\xi\rangle|| \Re \alpha_x-\Re y|^2+c|\langle \beta_\xi\rangle|| \Re \beta_x-\Re y|^2-C\delta^2(|\langle \alpha_\xi\rangle|+|\langle \beta_\xi\rangle| )\\
+\Im \langle \alpha_x-y,\alpha_\xi\rangle +\Im \langle y-\beta_x,\beta_\xi\rangle.
\end{multline*}
Therefore, deforming the contour in $y$ using 
$$
y\mapsto y+\frac{i\delta (\overline{\beta_\xi-\alpha_\xi})}{\langle \Re( \beta_\xi-\alpha_\xi)\rangle }, \ \ \ y \in \RR^n , 
$$
we have (on the new contour)
\begin{equation*}
%\label{e:ImPhase1}
\begin{split}
\Im \varphi(\alpha,y)-\varphi^*(\beta,y) & \geq c|\langle 
\alpha_\xi\rangle|| \Re \alpha_x-\Re y|^2+c|\langle \beta_\xi\rangle|| \Re \beta_x-\Re y|^2+\delta\frac{|\alpha_\xi-\beta_\xi|^2}{\langle \Re (\beta_\xi-\alpha_\xi )\rangle} \\ & \ \ \ \ \ \ \ 
-C\delta^2(|\langle \alpha_\xi\rangle|+|\langle \beta_\xi\rangle| ) 
+\Im \langle \alpha_x,\alpha_\xi\rangle -\Im \langle \beta_x,\beta_\xi\rangle.
\end{split}
\end{equation*}
Using 
$$|\Im \beta_x|+|\Im \alpha_x|+|\langle \alpha_\xi\rangle|^{-1}|\Im \alpha_\xi|+|\langle \beta_\xi\rangle|^{-1}|\Im \beta_\xi|\leq C\e_0\ll\delta,$$
we then obtain
\begin{equation*}
%\label{e:ImPhase1}
\begin{split}
\Im \varphi(\alpha,y)-\varphi^*(\beta,y) &  \geq c|\langle 
\alpha_\xi\rangle||\Re \alpha_x-y|^2+c|\langle \beta_\xi\rangle||\Re  \beta_x-y|^2
\\ 
& \ \ \ \ \ \ +c\delta\frac{|\alpha_\xi-\beta_\xi|^2}{\langle \Re (\beta_\xi-\alpha_\xi)\rangle}-C\delta^2(|\langle \alpha_\xi\rangle|+|\langle \beta_\xi\rangle| )
\end{split}
\end{equation*}
In particular when
$$
|\Re \alpha_x-\Re \beta_x|\geq \delta \ \text{ or } \  |\Re \alpha_\xi-\Re \beta_\xi|\geq 2c\delta\min(\langle \Re \alpha_\xi\rangle,\langle\Re \beta_\xi\rangle)/C,
$$
the integrand is bounded by
$$
e^{-(\langle \Re \alpha_\xi\rangle+\langle \Re \beta_\xi\rangle)(1+|\Re \alpha_x -\Re\beta_x|)/Ch}.
$$
Therefore, modulo an $O_N(e^{-C/h})_{\langle \xi\rangle^{N}L^2(\Lambda_1)\to \langle \xi\rangle^{-N}L^2(\Lambda_2)}$ error, the kernel is given by
\begin{equation*}
\begin{gathered}
k ( \alpha, \beta ) := h^{-\frac{3n}{2}}\int_{ \RR^n}  e^{\frac{i}{h}(\varphi(\alpha,y)-\varphi^*(\beta,y)){h}} k_1 (\alpha,\beta,y)\widetilde \chi dy\\
\widetilde \chi := \chi(\delta^{-1}d{(}\alpha_x,\beta_x))
\chi ( \delta^{-1}\min(\langle \Re \alpha_\xi\rangle,\langle \Re \beta_\xi\rangle)^{-1}|\alpha_\xi-\beta_\xi|), 
\end{gathered}
\end{equation*} 
where $ \chi $ is a suitable cut-off function and $ k_1\in \langle \Re \alpha_\xi \rangle^{\frac n4} \langle \Re \beta_\xi \rangle^{\frac n4}S^{0} ( \Lambda_{G_2} \times \Lambda_{G_1} 
\times \RR^n )$, and the dependence on the last variable is periodic {and holomorphic on $|\Im y|\leq c.$}

We claim that $ k ( \alpha, \beta ) $  is given by 
\begin{equation}
\label{e:finalPhase}
h^{-n}e^{\frac{i}{h}\Psi(\alpha,\beta)}k(\alpha,\beta)\tilde{\chi}, 
\end{equation}
where $k\in S^0 ( \Lambda_{G_2} \times \Lambda_{ G_2 } ) $. To see this we note that the critical point in $y$ is given by
$$
y_c =\frac{i(\beta_\xi-\alpha_\xi) +\langle \alpha_\xi\rangle \alpha_x+\langle \beta_\xi\rangle \beta_x}{\langle \alpha_\xi\rangle +\langle \beta_\xi\rangle}.
$$
We then deform the contour to $
y\mapsto y+y_c 
$. 
The phase  becomes
$$
(\alpha_x-\beta_x)\frac{\langle \alpha_\xi\rangle\beta_\xi+\alpha_\xi\langle \beta_\xi\rangle}{\langle \alpha_\xi\rangle +\langle \beta_\xi\rangle}+\frac{i(\langle \alpha_\xi\rangle +\langle \beta_\xi\rangle)}{2}y^2+\frac{i}{2}\frac{\langle \alpha_\xi\rangle\langle \beta_\xi\rangle (\beta_x-\alpha_x)^2}{\langle \alpha_\xi\rangle +\langle \beta_\xi\rangle}+\frac{i}{2}\frac{(\beta_\xi-\alpha_\xi)^2}{\langle \alpha_\xi\rangle +\langle \beta_\xi\rangle}.
$$
%$$
%(x-y)\xi+(y-w)\omega+i\frac{(\omega-\xi)^2}{\langle \xi\rangle +\langle \omega\rangle}+\frac{i}{2}(\langle \omega\rangle (y+i\frac{\omega-\xi}{\langle \xi\rangle +\langle \omega\rangle}-w)^2+\langle \xi \rangle (y+i\frac{\omega-\xi}{\langle \xi\rangle +\langle \omega\rangle}-x)^2)
%$$
%\begin{gather*}
%(x-y)\xi+(y-w)\omega +\frac{i}{2}(\langle \omega\rangle [(y-w)^2+2i(y-w)\frac{\omega-\xi}{\langle \xi\rangle +\langle \omega\rangle}]\\
%+\langle \xi \rangle [(y-x)^2+2i(y-x)\frac{\omega-\xi}{\langle \xi\rangle +\langle \omega\rangle}])\\
%x\xi-w\omega +\frac{i}{2}(\langle \omega\rangle [(y-w)^2-2iw\frac{\omega-\xi}{\langle \xi\rangle +\langle \omega\rangle}]+\langle \xi \rangle [(y-x)^2-2ix\frac{\omega-\xi}{\langle \xi\rangle +\langle \omega\rangle}])\\
%(x-w)\frac{\langle \xi\rangle\omega+\xi\langle \omega\rangle}{\langle \xi\rangle +\langle \omega\rangle}+\frac{i}{2}(\langle \omega\rangle (y-w)^2+\langle \xi \rangle (y-x)^2)\\
%(x-w)\frac{\langle \xi\rangle\omega+\xi\langle \omega\rangle}{\langle \xi\rangle +\langle \omega\rangle}+\frac{i(\langle \xi\rangle +\langle \omega\rangle)}{2}[(y-\frac{w\langle \omega\rangle+x\langle \xi\rangle}{\langle \xi\rangle +\omega\rangle})^2+\frac{\langle \omega\rangle w^2+\langle \xi \rangle x^2}{\langle \xi\rangle +\langle \omega\rangle} -\Big(\frac{w\langle \omega\rangle +x\langle \xi\rangle }{\langle \xi\rangle +\langle \omega\rangle}\Big)^2]
%\end{gather*}
and the method of steepest descent gives~\eqref{e:finalPhase}. 
\end{proof}

The next lemma gives the first part of Proposition \ref{p:prelim-Projector}:
\begin{lemm}
\label{l:TSbounded}
For all $m\in \RR$, there are $C, h_0>0$ such that for $0<h<h_0$, 
$$
\|T_{\Lambda}S_{\Lambda}\|_{\langle \xi\rangle^mL^2({\Lambda})\to \langle \xi\rangle^mL^2(\Lambda)}\leq C.
$$
\end{lemm}
\begin{proof}
By Lemma~\ref{l:compositePhase}, we need to show uniform 
boundedness of $ K $ with the kernel given by 
$$
K (\alpha,\beta)=h^{-n}e^{\frac{i}{h}\Psi(\alpha,\beta)}k(\alpha,\beta)\chi(\delta^{-1}d( \Re\alpha_x, \Re\beta_x ) )\chi\Big(\frac{|\Re\alpha_\xi-\Re\beta_\xi|}{\delta\min(\langle\Re \alpha_\xi\rangle,\langle \Re\beta_\xi\rangle)}\Big) .
$$
where $\Psi$ is as in~\eqref{e:tempProjectorPhase}, $k\in S^0 $. 

In particular, conjugating by $\langle \Re \alpha_\xi\rangle^me^{H(\alpha)/h}$, we need to show that the operator with the kernel 
\begin{gather*}
h^{-n}e^{\frac{i}{h}(\Psi(\alpha,\beta)-iH(\beta)+iH(\alpha))}\tilde{k}(\alpha,\beta)
\end{gather*}
is  bounded on $L^2 ( \Lambda, dm(\alpha) )$ where 
$$\tilde{k}(\alpha,\beta):=\left( \frac{ \langle \Re \alpha_{\xi}\rangle}{\langle \Re \beta_\xi\rangle}\right)^{m}k(\alpha,\beta)\chi(\delta^{-1}|\Re\alpha_x-\Re\beta_x|)\chi\left( \frac{ 
| \Re\alpha_\xi-\Re\beta_\xi|}{\delta \min(\langle\Re\alpha_\xi\rangle,\langle \Re\beta_\xi\rangle)}\right).$$

To establish this we define 
\begin{equation}
\label{eq:defPhi}
\Phi(\alpha,\beta):=\Psi(\alpha,\beta) - iH(\beta) + iH(\alpha),  
\end{equation}
where we see that $\Phi(\alpha,\alpha)=0$. 
Next, we note that 
\begin{align}
\label{eq:dalPh} 
d_{\alpha}\Phi|_{\alpha=\beta}&=\alpha_\xi d\alpha_x + id_\alpha H.
\end{align}
Therefore (see \eqref{eq:defH}), $\Im d_\alpha \Phi|_{\alpha=\beta}=0$. Similarly, $\Im d_\beta \Phi|_{\alpha=\beta}=0$ and hence $\Im \Phi$ vanishes quadratically at $\alpha=\beta$. 

In the case of no deformation (that is, for $ \Lambda = T^* \TT^n $)
$$
\Im \Phi\geq  c\langle \alpha_\xi\rangle |\alpha_x-\beta_x|^2+c\langle \alpha_\xi\rangle^{-1}|\alpha
_\xi-\beta_\xi|^2,\qquad \alpha, \,\beta \in T^{*}\TT^n .
$$
Since $\Lambda$ is a small conic perturbation of $T^*\mathbb{T}^n$, this remains true on $\Lambda$. Hence, 
\begin{gather*}
| K ( \alpha, \beta ) | \leq Ch^{-n}e^{  (c\langle\Re \alpha_\xi\rangle |\Re\alpha_x-\Re\beta_x|^2+c\langle \Re \alpha_\xi\rangle^{-1}|\Re\alpha_\xi-\Re\beta_\xi|^2) /h }  \langle\Re \alpha_\xi \rangle^{\frac n 4} \langle\Re \beta_\xi \rangle ^{\frac n 4 } \widetilde \chi , \\ \,
\widetilde \chi = \chi(\delta^{-1}d( \Re\alpha_x, \Re\beta_x ) )
\chi(  \delta^{-1}\min(\langle\Re\alpha_\xi\rangle,\langle \Re\beta_\xi\rangle)^{-1} 
| \Re\alpha_\xi-\Re\beta_\xi|) . 
\end{gather*} 
The Schur's test for boundedness on $L^2$ then shows that 
$ K $ is uniformly bounded on $L^2 ( \Lambda ) $. 
 \end{proof}

The following lemma shows that compact changes of the Lagrangian $\Lambda$ change the norm on $L^2(\Lambda)$ but not the elements in the space.
\begin{lemm}
\label{l:changeLagrangian}
Let $G_1$ and $G_2$ satisfy~\eqref{e:controlDeformation}. Then, for all $M,N>0$,
$$
\indic_{|\xi|\leq M}T_{\Lambda_{G_2}} {S}_{\Lambda_{G_1}}  = 
\mathcal O_h ( 1 ) :\langle \xi\rangle^{N}L^2(\Lambda_{G_1})\to \langle \xi\rangle^{{-}N}L^2(\Lambda_{G_2}). 
$$
\end{lemm}
\begin{proof}
By Lemma~\ref{l:compositePhase}  we only need to show that the operator 
$ \indic_{|\xi| \leq {M} } K $ is bounded. However, the structure of the
Schwartz kernel described in that lemma shows that the kernel 
of $ \indic_{|\xi| \leq {M} } K $  is smooth and compactly supported. 
Except for a loss in the constant due to different weights the boundedness follows.
\end{proof}

\section{Asymptotic description of the projector}
\label{s:asy}

The main part of this section consists of a construction of a parametrix for 
the orthogonal projector onto the (closure of the) image of $ T_\Lambda $.
It is inspired by \cite[\S 1]{Sj96} which in turn followed ideas of
\cite{mess}, \cite{BS}, \cite{BG} and \cite{HS}. A detailed presentation in 
a simpler case of compactly supported weights can be found
in \cite[\S 6]{GZ} and it can be used as a guide to the more notationally involved case at hand. We then use the argument from \cite{BG} and \cite{Sj96} to relate the parametrix to the exact projector.

\subsection{The structure of the parametrix} 
We seek an operator of the following form 
\begin{equation}
\label{eq:kerBLa}
\begin{gathered}  B_\Lambda u ( \alpha ) = h^{-n} \int_{T^* \TT^n}  
e^{ i  \psi ( \alpha, \beta )/h -   2 H ( \beta )/h } a ( \alpha, \beta, h  ) u ( \beta ) d m_\Lambda ( \beta)  ,  \\ 
d m_\Lambda ( \beta ) := (\sigma|_\Lambda)^n / n! = 
d \alpha ,  \ \beta = \Re \alpha , \ \ \alpha \in \Lambda ,  
\end{gathered}
\end{equation}
where $ \psi $ and $ a $ satisfy (for all $  k,k',\ell,\ell' \in \NN^n  $) 
\begin{equation}
\label{eq:propsi}
\begin{gathered} 
\supp \psi, \ \supp a \subset \{ ( \alpha, \beta ) :
d ( \alpha_x , \beta_x ) \leq \epsilon , \ \
| \alpha_\xi - \beta_\xi | \leq \langle \alpha_\xi \rangle \epsilon \},  \\
\partial_{\alpha_x}^k \partial_{\alpha_\xi}^\ell 
\partial_{\beta_x}^{k'} \partial_{\beta_\xi}^{\ell'} \psi ( \alpha, \beta ) 
= \mathcal O ( \langle \alpha_\xi \rangle^{1-|\ell|-|\ell'|}) , 
\ \ \overline{ \psi ( \alpha, \beta ) } = - \psi ( \beta, \alpha ) , 
\end{gathered}
\end{equation}
and
\begin{equation}
\label{eq:propaa}
\begin{gathered} 
a ( \alpha, \beta , h ) \sim \sum_{j=0}^\infty ( \langle \alpha_\xi \rangle^{-1} h)^j a_j ( \alpha, \beta ) , \ \ 
\overline{ a ( \alpha, \beta ) } = a ( \beta, \alpha ) ,\\
\partial_{\alpha_x}^k \partial_{\alpha_\xi}^\ell 
\partial_{\beta_x}^{k'} \partial_{\beta_\xi}^{\ell'} a_j ( \alpha, \beta ) 
= \mathcal O ( \langle \alpha_\xi \rangle^{-|\ell|-|\ell'|}). 
\end{gathered}
\end{equation}
The basic properties we need are self-adjointness and idempotence:
\begin{equation}
\label{eq:propBLa}
B_\Lambda = B_\Lambda^{*,H} , \ \ 
B_\Lambda \equiv  B_\Lambda^2 , 
\end{equation}
where $ A \equiv B $ means that $ A - B =  \mathcal O ( h^N )_{ \langle \xi\rangle^NL^2(\Lambda)\to \langle \xi\rangle^{-N}L^2(\Lambda)} $ for all $ N $. 

The deeper requirement comes from relating the image of $ B_\Lambda $ 
to that of $ T_\Lambda $:
\begin{prop}
\label{p:ZLa}
Suppose that $ Z_j $, differential operators with holomorphic coefficients in $ \Gamma $, are defined by 
\[  Z_j :=   \langle\zeta\rangle^{-1} ( hD_{z_j} - \zeta_j ) +\tfrac 12 
\langle \zeta \rangle ^{ -3} \zeta_j (h D_z - \zeta)^2  - i h D_{\zeta_j } - \tfrac {n} 4 h \langle \zeta\rangle^{-2} \zeta_j .
 \]
If 
\begin{equation}
\label{eq:defZLa} Z_j^\Lambda = Z_j|_\Lambda \end{equation}
in the sense of restriction of holomorphic operators to totally real submanifolds, then for $ u \in \mathscr A_\delta $, 
\begin{equation}
\label{eq:ZLa}
Z_j^\Lambda T_\Lambda u ( \alpha ) = 0 , \ \ j =1, \cdots, n .
\end{equation}
\end{prop}
\begin{proof}
Putting 
\[ W_j = 
 \langle \zeta \rangle^{-\frac n 4} Z_j \langle \zeta \rangle^{\frac n 4} 
 = \langle\zeta\rangle^{-1} ( hD_{z_j} - \zeta_j )  + \tfrac 12 
\langle \zeta \rangle ^{ -3} \zeta_j (h D_z - \zeta)^2  - i h D_{\zeta_j } - \tfrac n 2  h \langle \zeta\rangle^{-2} \zeta_j  ,\]
we check that 
\[ W_j ( e^{ \frac i h (  \langle z - y + 2 \pi k , \zeta \rangle + 
\frac i 2 \langle \zeta \rangle ( z - y + 2 \pi k )^2 )} ) = 0 , \]
for all $ y \in \TT^n $ and $ k \in \ZZ^n $. 
The definition of $ T_\Lambda $ then immediately gives \eqref{eq:ZLa}. 
\end{proof}

%We recall the operators approximately annihilating the image of 
%$ T_\Lambda$:
%\[
%^Z_j^\Lambda := Z_j^\Lambda (\alpha , h D_\alpha , h ) , \ \ 
We note that $ Z_j $'s commute and hence we also have
\[ [ Z_j^\Lambda , Z_k^\Lambda ] = 0. \]

We write 
\[ z_j^\Lambda := \langle \zeta \rangle ^{-1} ( z_j^* - \zeta_j ) +\tfrac 12 
\langle \zeta \rangle^{ -3} ( z^* - \zeta)^2 \zeta_j - i \zeta_j^* |_\Lambda , \ \ 
\{ z_j^\Lambda , z_k^\Lambda \} = 0 \]
for the principal symbol of $ Z_j^\Lambda $ (in a sense
which will be explained after the rescaling below).
The vanishing of the Poisson bracket reflects the 
fact that  
$ z_j^\Lambda $ vanish on the involutive manifold 
$ \{ ( \alpha, d_\alpha \varphi ( \alpha, y ) : \alpha \in \Lambda \,,
\ y \in \TT^n \} $ -- see Lemma \ref{l:zetac} below.

{Since $B_\Lambda$ is supposed to be a parametrix for} a self-adjoint projection onto the image $ T_\Lambda $, Proposition \ref{p:ZLa} shows that we should have
\begin{equation}
\label{eq:ZLaB} Z^\Lambda_j B_\Lambda \equiv 0 , \ \ B_\Lambda (Z^\Lambda_j)^{*, H} 
\equiv 0 , \end{equation}
where the definition of $ \equiv $ is given in \eqref{eq:BLaq} below.

To explore the second condition in terms of the kernel of 
 $ B_\Lambda $ we denote by $ A^*$ the formal adjoint
 of an operator $ A $ on $ L^2 ( \Lambda , dm_\Lambda ) $ (no weight). We also define a transpose of $ A $ by 
\[  \int_\Lambda A u ( \alpha ) v ( \alpha ) dm_\Lambda ( \alpha ) =
\int_\Lambda u ( \alpha ) A^t v ( \alpha ) dm_\Lambda ( \alpha ) . \]
We note the general fact 
$ ( A^*)^t = J \circ A \circ J $, $  J u := \bar u $. 
Then, with $ K_\Lambda ( \alpha , \beta ) := h^{-n} e^{ i  \psi ( \alpha, \beta )/h} a ( \alpha, \beta, h  ), $
\[ \begin{split}   ( A B_\Lambda )^{*, H}  u ( \alpha ) &  = \int_{\Lambda}  K_\Lambda ( \alpha, \beta ) A^* ( e^{ -  2H ( \bullet)/h } u ( \bullet ) ) ( \beta ) 
d m_\Lambda ( \beta)  \\
& =  \int_\Lambda  (A^*)^t \left( K_\Lambda ( \alpha, \bullet ) \right) ( \beta )  
e^{ -  2H( \beta)/ h } u ( \beta ) d m_\Lambda( \beta)  \\
& =  \int_\Lambda (J \circ A \circ J ) \left( K_\Lambda ( \alpha, \bullet ) \right) ( \beta )  
e^{ -  2H( \beta)/ h } u ( \beta ) d m_\Lambda( \beta) .
\end{split}
\]
Using \eqref{eq:ZLaB} and the above calculation with $ A = Z_j^\Lambda $ gives \begin{equation}
\label{eq:JZJ} 
\widetilde Z_j^\Lambda ( K_\Lambda ( \alpha, \bullet ) ) \equiv 0 , \ \ 
 \widetilde Z_j^\Lambda := J \circ Z_j^\Lambda \circ J . \ \ 
J u := \bar u , \end{equation}
The principal symbols are given by 
\begin{equation}
\label{eq:JzJ} \widetilde z_j^\Lambda ( \beta, \beta^* )  = 
\bar z_j^\Lambda ( \beta , - \beta^*) , \ \ 
\bar z_j^\Lambda := \overline{ z_j^\Lambda  ( \bar \beta,  \bar \beta^* )} , \end{equation}
and by almost analytic continuation are defined in $ \Gamma$. 

\noindent {{\bf{Remark:}} Here we recall that the complex conjugation of $\beta$ and $\beta^*$ is defined as in~\eqref{eq:conjLa}.}

Lemma \ref{l:zetac} will discuss some properties of $ z_j^\Lambda $ 
and $ \bar z_j^\Lambda $ after a linear rescaling. Here we point out that
$ z_j^\Lambda $ is a restriction to $ \Lambda $ of a holomorphic function in $ \Gamma $ but $ \bar z_j^\Lambda ( \alpha, \alpha^* ) = \overline{
\zeta_j^\Lambda ( \alpha, \alpha^* ) }$, $ ( \alpha, \alpha^* ) \in 
T^* \Lambda $, is {\em not}.

\subsection{A general construction}
Here we establish the following
\begin{prop}
\label{p:BLa}
Let $ Z_j^\Lambda $ and $ \widetilde Z_j^\Lambda $ be given by 
\eqref{eq:defZLa} and \eqref{eq:JZJ} respectively.  Suppose that
$ b = b ( \alpha, h ) $ satisfies \eqref{eq:propaa} (with no dependence on 
$ \beta $).

Then there exist $ \psi ( \alpha, \beta ) $ and $ a ( \alpha, \beta , h ) $
satisfying \eqref{eq:propsi} and \eqref{eq:propaa} and such that
\begin{equation}
\label{eq:BLaq}
\begin{gathered}
\psi ( \alpha, \alpha ) = - 2 i H ( \alpha ) ,  \ \ \ 
a_j ( \alpha, \alpha ) = b_j ( \alpha) , 
\\
 e^{ - \frac i h \psi ( \alpha, \beta ) } Z_j^\Lambda ( \alpha, 
h D_\alpha, h ) \left( e^{  \frac i h \psi ( \alpha, \beta ) } a( \alpha, \beta, h ) \right) = \mathcal O_\infty  \\  
 e^{ - \frac i h \psi ( \alpha, \beta ) } \widetilde Z_j^\Lambda ( \beta, 
h D_\beta, h ) \left( e^{  \frac i h \psi ( \alpha, \beta ) } a( \alpha, \beta, h ) \right) = \mathcal O_\infty \end{gathered}, 
\end{equation}
where
\[ \mathcal O_\infty :
= \mathcal O \left( d ( \alpha_x, \beta_x )^\infty + 
( \langle \alpha_\xi \rangle^{-1}  | \alpha_\xi - \beta_\xi|)^\infty + 
( \langle \alpha_\xi \rangle^{-1} h) ^\infty \right). \]

The phase $ \psi ( \alpha, \beta ) $ and amplitudes
$ a_j ( \alpha, \beta ) $
are uniquely determined by $ b_j ( \alpha ) $ up to 
$ \mathcal O_\infty $ and 
\begin{equation}
\label{eq:BLap}
- H ( \alpha ) - \Im \psi ( \alpha, \beta ) - H( \beta ) \leq 
 - (d ( \alpha_x, \beta_x )^2  +  \langle \alpha_\xi \rangle^{-1}  | \alpha_\xi - \beta_\xi|^2 )/C ,
\end{equation}
for some $ C > 0 $.
\end{prop}

{We will see that $a$ and $\psi$ are essentially determined by their values on the diagonal in $\Lambda\times \Lambda$. Therefore, t}he construction of $ \psi $ and $ a $ can be done locally and
we now work near $ \alpha^0 = ( x^0, \xi^0 ) \in \Lambda $, where
we identify $ \Lambda $ with $ T^* \TT^n $ as in \eqref{eq:paraLaG}.

For $ \alpha, \beta $ in a conic neighbourhood of $ \alpha^0 $ we 
rescale $ Z_j $ using the following change of variables:
\begin{equation}
\label{eq:tildealbe} \begin{split} 
& \tilde \alpha_x : = \alpha_x - \alpha^0_x , \ \ 
\tilde \alpha_\xi := \langle \alpha_\xi^0 \rangle^{-1} ( \alpha_\xi - 
\alpha_\xi^0 ) ,\\
& \tilde \beta_x : = \beta_x - \alpha^0_x , \ \ 
\tilde \beta_\xi := \langle \alpha_\xi^0 \rangle^{-1} ( \beta_\xi - 
\alpha_\xi^0 ) .
\end{split}
\end{equation}
In this new coordinates the operators $ Z_j^\Lambda $ become
\begin{equation}
\label{eq:tildeZLa}  \begin{gathered} Z_j^\Lambda = \zeta_j^\Lambda ( \tilde \alpha, \tilde h D_{\tilde \alpha } ) + \tilde h \zeta_j^1 ( \tilde \alpha , \tilde h 
D_{\tilde \alpha } ) + \tilde h^2 \zeta_j^2 ( \tilde \alpha ) , 
%-
%\tfrac12 \tilde h \lambda(\tilde \alpha_{\xi} ) ^2 \tilde \alpha_{\xi_j} -
%\tfrac12 \tilde h \rho_j , 
\ \ \zeta_j^\Lambda = z_j |_\Lambda , 
\\ z_j ( z, \zeta, z^*, \zeta^* ) := \lambda( \zeta ) ( z_j^* - \zeta_j {-}
\theta_j ) {+} \tfrac12 \lambda ( \zeta)^{3} ( z^* - \zeta {-} \theta )^2 (\zeta_j + \theta_j ) 
- i \zeta_j^* ,
 \\
\tilde h := \frac{h}{\langle \alpha_\xi^0
\rangle }, \  \ \
\lambda ( \zeta ) := \frac{\langle \alpha^0_\xi \rangle}{ 
\langle \langle \alpha_\xi^0 \rangle ( \zeta + \theta ) \rangle } , \ \ \  {{\theta}} :=   
\frac{ \alpha_{\xi}^0 }{ \langle \alpha_\xi^0 \rangle}, 
\end{gathered} \end{equation}
where we still have $ \{ \zeta_j^\Lambda , \zeta_k^\Lambda \} = 0 $. The operators
$ \widetilde Z_j^\Lambda $ are defined using \eqref{eq:JZJ}.

We now define the rescaled phase and amplitudes:
\begin{equation}
\label{eq:tildepha}
\begin{gathered}
\tilde \psi ( \tilde \alpha, \tilde \beta ) := 
\langle \alpha_\xi^0 \rangle^{-1} \psi ( \alpha, \beta ) , \ \ 
\tilde H ( \tilde \alpha ) := \langle \alpha_\xi^0 \rangle^{-1} 
H ( \alpha ) , \ \  {\tilde{G}(\tilde{\alpha})=\langle \alpha_\xi^0\rangle^{-1}G(\alpha), }\\
\tilde a_j ( \tilde \alpha, \tilde \beta ) : = 
\langle \alpha_\xi^0 \rangle^j a_j ( \alpha, \beta ) , 
\ \ \tilde b_j ( \tilde \alpha ) :=  \langle \alpha_\xi^0 \rangle^j b_j ( \alpha) , 
\end{gathered}
\end{equation}
so that
\[  a ( \alpha, \beta ) \sim \sum_{ j=0}^\infty \tilde h^j \tilde a_j 
( \tilde \alpha, \tilde \beta ) , \ \ 
b ( \alpha ) \sim \sum_{ j=0}^\infty \tilde h^j \tilde b_j 
( \tilde \alpha) .
 \]
Hence \eqref{eq:BLaq} becomes
\begin{equation}
\label{eq:BLaq1}
\begin{gathered}
\tilde \psi ( \tilde \alpha, \tilde \alpha ) = - 2 i \tilde H ( \alpha ) ,  \ \ \ 
\tilde a_j ( \tilde \alpha, \tilde \alpha ) = \tilde b_j ( \tilde \alpha) , 
\\
 e^{ - \frac i {\tilde h}  \tilde \psi ( \tilde \alpha, \tilde \beta ) } Z_j^\Lambda  \left( e^{  \frac i {\tilde h} \tilde \psi (  \bullet, \tilde \beta  ) } \tilde a(  \bullet, \tilde \beta , \tilde h ) \right) ( \tilde \alpha ) = \mathcal O \left(   | \tilde \alpha- \tilde \beta|)^\infty + 
{\tilde h} ^\infty \right), \\  
 e^{ - \frac i {\tilde h } \tilde \psi ( \tilde \alpha, \tilde \beta ) } \widetilde Z_j^\Lambda \left( e^{  \frac i {\tilde h } \psi ( \tilde \alpha,
 \bullet ) } \tilde a( \tilde \alpha, \bullet, \tilde h ) \right) ( \tilde \beta) = \mathcal O \left(   | \tilde \alpha- \tilde \beta|)^\infty + 
{\tilde h} ^\infty \right), \end{gathered}
\end{equation}
where now $ \tilde \psi $ and $ \tilde a_j $ are smooth functions in a neighbourhood of $ 0 \in \RR^{2n} \times \RR^{2n} $. 
\begin{equation}
\label{eq:convent}
\text{\bf  To simplify notation we now drop $ \tilde{} \, $ in $ \tilde h $, $ \tilde \psi $, $ \tilde H $, ${\tilde{G}}$ 
and $ \tilde a $.}
\end{equation}
This will apply until the end of the construction of the phase and the amplitude.

\subsubsection{Eikonal equations}
Here we work in the coordinates \eqref{eq:tildealbe} and use
the convention \eqref{eq:convent}. Hence we assume that $ \Lambda $ is
a neighbourhood of $ 0 $ in $ T^* \RR^n $.

Let $ \zeta_j^\Lambda  $ and $ \widetilde \zeta_j^\Lambda $ be 
the principal symbols of $ Z_j^\Lambda $ and $ \widetilde Z_j^\Lambda$ respectively -- see \eqref{eq:tildeZLa}. {The} eikonal equations we want to solve are
\begin{equation}
\label{eq:Laeik}
\begin{gathered}
\zeta_j^\Lambda ( \alpha, d_\alpha \psi ( \alpha, \beta ) ) = 
\mathcal O ( | \alpha - \beta |^\infty ) , \\ 
\widetilde \zeta_j^\Lambda ( \beta, d_\beta \psi ( \alpha, \beta ) ) 
= \mathcal O ( | \alpha - \beta |^\infty ) ,  \end{gathered}
\end{equation}
for $ \alpha, \beta \in \Lambda$. 
We also put 
\[ \bar{ \zeta}_j^\Lambda ( \alpha, \alpha^* ) := 
\widetilde {\zeta}_j^\Lambda ( \alpha, - \alpha^* ) , \]
see \eqref{eq:JZJ} and \eqref{eq:JzJ}. We note that
for $ ( \alpha, \alpha^* ) \in T^* \Lambda $, 
$ \bar{\zeta}_j^\Lambda ( \alpha, \alpha^*) = 
\overline{ \zeta_j^\Lambda ( \alpha, \alpha^* )}$. The next lemma
records the Poisson bracket properties of $ \zeta_j^\Lambda $ on 
$ \Lambda $:

\begin{lemm}
\label{l:zetac}
Let $ \{ \bullet, \bullet \} $ denote the Poisson bracket on 
$ T^* \RR^n   $ defined using the (real) symplectic form $ 
\sigma_\Lambda := (\sigma_{T^* \CC^n })|_\Lambda  $ and coordinates 
\eqref{eq:paraLaG}.  Let 
\[ \Sigma := \{ \zeta_j^\Lambda ( \rho ) = 0 : \rho \in 
 T^* \RR^{2n}  ,  \ | x^*  - \xi - \theta  | < { \lambda ( \xi )^{-1} }  \} ,  \ \ \rho = ( x, \xi, x^*, \xi^* ),   \]
where $ \lambda $ is defined in \eqref{eq:tildeZLa}. 

Then, for $ \zeta_j^\Lambda $ defined above we have 
%\begin{equation}
%\label{eq:zetajLa} 
$\{ \zeta_j^\Lambda, \zeta_k^\Lambda \} = 0 $, 
%\end{equation}
and, for $ \| G \|_{ C^2 } \ll 1 $, 
\begin{equation}
\label{eq:zetajLab}   
\left( \tfrac 1 {2i}  \{ \zeta_j^\Lambda , \bar \zeta_k^\Lambda \} ( \alpha, \alpha^* ) \right) _{1\leq j,k\leq n } \gg c I , \ \ c > 0 , 
\end{equation}
for $ ( \alpha, \alpha^*) \in \Sigma  \cap \nbhd_{ T^* \RR^{2n} } ( 0 ) $.
\end{lemm}

\noindent
{The positivity condition in Lemma~\ref{l:zetac} will be used in two places. First, it is used to guarantee that the Lagrangian used to construct the phase solving~\eqref{eq:Laeik} is strictly positive (see \eqref{eq:strpos}). Next, when $G$ is only smooth, this condition will be crucial when proving~\eqref{eq:sJt} (see also~\cite[(6.29)]{GZ}) and hence that the Lagrangian we construct is almost analytic. The proof of the Lemma will also show that there are solutions to $\zeta_j^\Lambda(\rho)=0$ with $|x^*-\xi-\theta|\geq  \lambda ( \xi )  $ {($ \lambda ( \xi ) \sim 1 $ for $ \xi $ in 
a neighbourhood of $ 0 $)}. However,~\eqref{eq:zetajLab} may not be satisfied at these points and hence (at leais \emph{not} appropriately positive.}

\begin{proof}
It is enough to check \eqref{eq:zetajLab} for $ G = 0$. In that 
case $ \Sigma $ is contained in in $ \{ ( \alpha, d_\alpha \varphi ( \alpha, y ) : \alpha \in \RR^{2n} , y \in\CC^n \} $ where 
$ \varphi ( \alpha, y ) $ is the rescaled phase of our FBI transform
(this follows from the fact that $ \zeta_j^\Lambda $ are pri{n}cipal symbols of operators annihilating $ T $). Hence,
\begin{equation}
\label{eq:phiSig} \begin{gathered} 
\begin{split} \Sigma & := \{ ( x, \xi , \varphi_x  , 
 \varphi_\xi ) : y \in \CC^n {, \ 
 | \xi + \theta - \varphi_x  | < 
1 \ } 
  , \ x, \xi \in \RR^n \} \cap T^* \RR^{2n} \\
 & = 
\{ ( x, \xi , \xi + {{\theta}} , 0 ) \},
\ \ 
\varphi = \varphi ( x, \xi , y ) := \langle x - y ,  \xi + {{\theta}}  \rangle + 
\tfrac i 2 \lambda ( \xi )^{{-1}} ( x - y )^2  ,
\end{split} 
\end{gathered} 
\end{equation}
where $ \lambda $ was defined in \eqref{eq:tildeZLa}. 
{(We just check that if $ x^* = \varphi_x ( x , \xi, y ) $, 
 $ \xi^* = \varphi_\xi ( x, \xi, y ) $ then $ y = x - i \lambda ( \xi + \theta - x^* ) $ 
and $ \xi^* = i \lambda ( \xi + \theta - x^* ) 
+ \tfrac12 i \partial_\xi \lambda  ( \xi + \theta - x^* )^2 $. As $ x^* $ and $ \xi^* $ are real we obtain
that either $ x^* = \xi + \theta $ as claimed or 
 \[ \begin{split} \lambda | \xi + \theta - x^* | & = 2 \lambda^2
| \partial_\xi \lambda |^{-1}  = 2 \lambda^{-1} | \xi + \theta |^{-1} 
%\\ & 
=
2 \frac{ \langle R ( \xi + \theta ) \rangle } { R | \xi + \theta| }
\geq 2 
\end{split} \]
which contradicts the condition in \eqref{eq:phiSig}. Hence
$\xi = x^* - \theta  $ and $ y = x $.)} 

Since
$ \{ \zeta^\Lambda_j , \zeta_k^\Lambda \} = 0 $ we see
that 
\[ 
\begin{split}
\tfrac 1 {2i}  \{  \zeta_j^\Lambda , \bar \zeta_k^\Lambda \} & = \{ \Im \zeta_j^\Lambda, \Re \zeta_k^\Lambda \} \\
& = -
\partial_{ \xi_j } \left( \lambda( \xi) ( x_{{k}}^* - \xi_{{k}} {-}
{{\theta}}_{{k}} ) {+} \tfrac12 \lambda ( {\xi})^{3} ( x^* - \xi {-} {{\theta}} )^2 (\xi_{{k}} + {{\theta}}_{{k}} ) \right) = \lambda ( \xi ) \delta_{jk} , \end{split} \]
when evaluated at $ x^* = \xi + {{\theta}} $. Hence, for $ G = 0 $, 
the matrix is \eqref{eq:zetajLab} is given by $ \lambda ( \xi ) I $
and $ \lambda ( \xi ) \sim 1 $ for $ \xi $ bounded. Hence for $ G $ small
the matrix stays positive definite.
\end{proof}

From the geometric point of view, the framework for 
construction of the phase is the same as in \cite[\S 2.2]{GZ} (see also
\cite[\S 6.1]{GZ} for a presentation in a simpler case). It is convenient
to remove the weight by putting
\[ \psi_H ( \alpha, \beta ) := i H ( \alpha ) + \psi ( \alpha, \beta ) 
+ i H ( \beta ) . \]
We also define,
\begin{equation}
\label{eq:defzetH}   \begin{split} \zeta_j^H ( \alpha, \alpha^* ) & :=
\zeta_j^\Lambda ( \alpha, \alpha^* - i dH( \alpha ) ), 
\\  \bar \zeta_j^H  ( \alpha , \alpha^* ) & := 
\bar  \zeta_j^\Lambda  ( \alpha, \alpha^* + i dH( \alpha ) )
= \overline{ \zeta_j^H ( \bar \alpha, \bar \alpha^* ) }
 . \end{split} \end{equation}
(Here again the $ \bar \alpha $ and $ \bar \alpha^* $ are defined after an identification of $ \Lambda $ with $ T^* \RR^n $.)
Lemma \ref{l:zetac} remains valid for $ \zeta_j^H $.

The eikonal equations \eqref{eq:Laeik} become
\begin{equation}
\label{eq:LaeikH}
\begin{gathered}
\zeta_j^H ( \alpha, d_\alpha \psi_H ( \alpha, \beta ) ) = 
\mathcal O ( | \alpha - \beta |^\infty ) , \\ 
\bar \zeta_j^H ( \beta, - d_\beta \psi_H ( \alpha, \beta ) ) 
= \mathcal O ( | \alpha - \beta |^\infty ) ,  
\end{gathered} \end{equation}
for $ \alpha, \beta \in \Lambda$. Since we demand that $ 
\psi ( \alpha, \alpha ) = - 2 i  H ( \alpha ) $, 
%$ \psi ( \alpha, \beta ) =  - \overline{ \psi ( \bar \beta  , \bar \alpha )} $  
it follows that $ \psi_H ( \alpha, \alpha ) = 0 $, and by differentiation,
\begin{equation}
\label{eq:diagcond}
0 = d_\alpha  ( \psi_H ( \alpha, \alpha  )) = 
d_\alpha \psi_H ( \alpha, \beta ) |_{ \beta = \alpha} + d_\beta \psi_H ( \alpha, \beta )|_{\beta = \alpha }  , \ \ \alpha \in \Lambda. 
\end{equation}
To construct $ \psi_H $ we will construct $ \mathscr C_H $, a Lagrangian relation for which $ \psi_H $ will be the generating function: 
\begin{equation}
\label{eq:defCH}  \mathscr C_H = \{ ( \alpha, d_\alpha \psi_H ( \alpha, \beta ) , 
\beta, - d_\beta \psi_H ( \alpha, \beta ) ) : ( \alpha, \beta)  \in 
\nbhd_{\CC^{4n} } ( \Diag ( \Lambda \times \Lambda ) ) \}. \end{equation}
We first assume that $ G $, and hence $ H $, are real analytic and have 
holomorphic extensions. 

Writing $ \rho = ( x, \xi , x^*, \xi^* ) $,  the eikonal equations require that we should have (up to 
equivalence of almost analytic manifolds and exactly on $ T^* \Lambda $)
\begin{equation}
\label{eq:defSbS}  \begin{gathered}  \mathscr C_H \subset S \times \overline S , \ \ 
S := \{ \rho : \zeta_j^H ( \rho ) = 0 , \ \rho \in \nbhd_{ \CC^{4n}} 
(\RR^{4n} ), \ { | x^* - \xi - \theta | < 1 }   \} , \\ \overline S := \{ \bar \rho : \rho \in S \}=
\{ \rho : \bar \zeta_j^H ( \rho ) = 0 , \ \rho \in \nbhd_{ \CC^{4n}} 
( \RR^{4n}  ), \ { | x^* - \xi - \theta | < 1 } \} . \end{gathered} \end{equation}

The condition \eqref{eq:diagcond} means that 
\begin{equation}
\label{eq:diagoncond1}  \mathscr C_H \cap
\pi^{-1} (\Delta_{ \CC^{2n} \times \CC^{2n}} ) = 
\Delta (  ( S \cap \bar S ) \times ( S \cap \bar S )),  \end{equation}
where $ \Delta (  A \times A  ) := \{ ( a, a ) :
a \in A\} $. In fact, \eqref{eq:diagcond} shows that this must be true
for $ \mathscr C_H \cap
\pi^{-1} (\Delta_{ {\RR}^{2n} \times {\RR}^{2n}} ) $ and then it follows 
by analytic continuation (or an equivalence of almost analytic manifolds
once we move to the $ \CI $ category). We have the following additional property which comes from the choice of the weight $ H$:
\begin{lemm}
\label{l:SSR}
Let $ S$ and $ \bar S $ be defined in \eqref{eq:defSbS}. The for 
$ H $ satisfying \eqref{eq:defH} we have
\begin{equation}
\label{eq:SSR} 
( S \cap \bar S )_\RR = S_\RR = \{ ( \alpha, \Re ( z  
d\zeta |_\Lambda ) : \alpha \in \nbhd_{\RR^{2n} } ( 0 ) \}.
\end{equation}
\end{lemm}
\begin{proof}
As in the proof of Lemma \ref{l:zetac} it is useful to go to the origins of the symbols
$ \zeta_j^H $ \eqref{eq:defzetH}: $ Z_j^\Lambda $'s, with 
symbols $ \zeta_j^\Lambda $ annihilate the phase in $ T_\Lambda $ 
and that shows that, after switching to $ \zeta_j^H $, 
\begin{gather*} S_\alpha := S \cap T_\alpha^* \Lambda^\CC = 
\left\{ ( \alpha, d_\alpha \varphi ( \alpha, y ) 
+ i dH ( \alpha ) ) \
 :  \ y \in \CC^n \right\}, \\ 
\varphi ( \alpha, y ) := \langle z - y , \zeta + \theta \rangle + 
\tfrac i 2 \lambda^{{-1}} ( \zeta ) ( z - y )^2   ,\\
z = \alpha_x + i G_\xi ( \alpha_x, \alpha_\xi) , \ \
\zeta = \alpha_\xi - iG_x ( \alpha_x  , \alpha_\xi).
\end{gather*}
where $ \lambda ( \zeta ) $ and $\theta $ were defined in \eqref{eq:tildeZLa}.

In the case $ G = 0 $ (and hence $ H = 0 $),
$ S_\alpha  $ and $ \bar S_\alpha := \bar S \cap T^*_\alpha
\Lambda ^\CC  $ intersect transversally 
in one point. This remains true for a small perturbation induced by $ G $ with $ \| G \|_{C^2 } \ll 1 $ (this corresponds to symbolic norms before rescaling). Hence we 
are looking for a solution to 
\begin{equation}
\label{eq:dalH}  d_\alpha \varphi ( \alpha, y ) + i d H ( \alpha ) = 
\overline{ d_\alpha \varphi ( \alpha, y' ) } - i d H ( \alpha ) .\end{equation} 
Now, at $ y = y' = \alpha_x $ we have 
$ d_\alpha \varphi  ( \alpha, y ) = \zeta d z|_\Lambda $
and in view of the definition of $ d H $ in \eqref{eq:defH}, 
\eqref{eq:dalH} holds. It follows that  {for $\alpha\in \Lambda$}, that is 
for $ \alpha $ real, 
\[  S_\alpha \cap \bar S_\alpha = \{ ( \alpha, \Re ( z  
d\zeta |_\Lambda ) \} = S_\alpha \cap T^* \Lambda, \ \Lambda \simeq 
\nbhd_{T^* \RR^n } ( 0 ) .  \]
But this proves \eqref{eq:SSR}.
\end{proof}

Since 
\[ \mathscr C_H \subset \bigcap_{ j = 1}^n (\pi^*_L\zeta^H)^{-1} ( 0 ) \cap 
(\pi_R^* \bar \zeta_j^H)^{-1} ( 0 ) , \ \ 
\pi_L ( \rho, \rho' ) := \rho, \ \ \pi_R ( \rho, \rho' ) := \rho' , \]
it follows that the complex vector fields $ H_{\pi^*_L \zeta_j^H} $ and $ H_{\pi_R^* \bar \zeta_j^H}$
are tangent to $ \mathscr C_H$. By checking the case of $ T^* \Lambda = 
T^* \RR^n $ (no deformation and hence $ H \equiv 0 $) we have (see
\cite[\S 2.2]{GZ}) that $ S \cap \bar S $ is a symplectic submanifold 
(with respect to the complex symplectic form) of complex dimension 
$ {2} n $. The independence of $ H_{\zeta_k^H } $, $ H_{\bar \zeta_j^H } $
, $ j , k = 1, \cdots n $ (again easily seen in the unperturbed case)
shows that 
\[  B_{\CC^n} ( 0 , \epsilon ) \times B_{\CC^n} ( 0 , \epsilon ) 
\times ( S \cap \bar S ) \ni ( t, s, \rho)  \mapsto 
  ( \exp \langle t, H_{\zeta^H} \rangle  ( \rho ) , 
\exp \langle s , H_{ \bar \zeta_H} \rangle  ( \rho ) ) \in 
\CC^{8n} , \]
is a bi-holomorphic map to an embedded (complex) $ 4n$ dimensional
submanifold. 
This implies that
\begin{equation}
\label{eq:analCH} \mathscr C_H =  \left\{ (
\exp \langle t , H_{\zeta^H} \rangle ( \rho) , 
\exp \langle s , H_{\bar \zeta^H} \rangle ( \rho )) : \rho \in 
S \cap \bar S , \ \ t , s \in B_{ \CC^n } ( 0 , \epsilon ) \right\}  , \end{equation}
where $ \langle t, H_{\bullet^H } \rangle := 
\sum_{ k=1}^n t_k H_{\bullet_k^H } $, $ \bullet = 
\zeta, \bar \zeta $. {Checking again in the unperturbed case, we have that for $\rho \in S\cap \bar{S}$}
\begin{equation}
\label{e:project}
{\pi_*:  T_\rho \mathscr C_H \to T_{\pi(\rho)} \mathbb{C}^{4n} \ \text{ is onto.}}
\end{equation}

We now explain how to use 
almost analytic extensions off $ \Lambda $ in the $ \CI $ case. We first identify 
$ \Lambda $ with $ T^* \RR^n $ using \eqref{eq:paraLaG} and extending 
$ G $ almost analytically to $ \CC^{4n} $. The symplectic form is 
now the almost analytic extension of the symplectic form 
$ d \zeta \wedge dz |_\Lambda $. 
Hence we define (see the Appendix for the definitions)
\[ \mathscr C_H = 
\left\{ \left( \exp{ \widehat{ \langle t, H_{\zeta^H } \rangle }} (\rho), 
\exp{ \widehat {\langle s, H_{\bar \zeta^H } \rangle } } (\rho) \right) :
\rho \in S \cap \bar S  , \ \ t, s \in B_{\CC^n } ( 0 , \epsilon ) \right\} . \]
We claim that
{ \begin{equation}
\label{eq:sJt}  |\Im \exp{ \widehat{ \langle t, H_{\zeta^H } \rangle }} (\rho) | \geq |t|/ C , \ \ 
| \Im \exp{ \widehat {\langle s, H_{\bar \zeta^H } \rangle } } (\rho) |
\geq |s|/ C ,  \ \ \rho \in S \cap \bar S . \end{equation}
In fact, in view of Lemma \ref{l:zetac} at 
 $ \rho \in T^* \Lambda \cap S $ and for $ \| G \|_{C^2 } $ small, 
  we can assume 
  $ \{ \zeta_j^\Lambda, \bar \zeta_k^\Lambda \} ( \rho )/2i  $ is positive 
  definite. The  changes of variable
leading to $ \zeta_j^H $ {is a symplectomorphism} and hence we have the same property for $
\zeta_j^H $. By changing $ \zeta_j^H $ by a linear transformation 
we can then assume that $ \{ \zeta_j^H, \bar \zeta_k^H \} ( \rho )/2i  =  \delta_{kj} $. Hence we can make a linear
symplectic change of variables at any point of $ T^* \Lambda $ giving
new variables $  ( x, y, \xi,\eta )$, $ x, y , \xi , \eta \in \RR^n $, centered at $ 0 \in \RR^{4n} $,  
such that 
\[ \zeta_j^H = c ( \eta_j + i y_j ) + \mathcal O ( |x|^2 + |y|^2 + |\xi|^2 + |\eta|^2 )  , \ \ c >0 . \]
This continues to hold for the almost analytic continuations of $ \zeta_j^H $. 
That means that near $ 0 $,  
\begin{equation}
\label{eq:mathsJ}  S \cap \bar S  = \{ ( z, 0 , \zeta, 
0 ) + F ( z, \zeta ) ) : ( z, \zeta ) \in \nbhd_{ \CC^{2n} } ( 0 )  \} , \ \ F = \mathcal O ( |z|^2 + |\zeta|^2 ) ,
\end{equation}
We also note that for $ ( z, \zeta )\in \RR^{2n} $ (which corresponds to 
the interection with $ T^* \Lambda $), $ S \cap \bar S $ is real. This means
that in \eqref{eq:mathsJ}, 
\[ \Im F ( z , \zeta )  = 
\mathcal O ( ( |\Im z | + |\Im \zeta | ) ( |z| + |\zeta| ) ). \]
Hence,
\[ \begin{split} 
|\Im \exp{ \widehat{ \langle t, H_{\zeta^H } \rangle }} ( 
( z, 0 , \zeta, 
0 ) + F ( z, \zeta ) ) | & =
|( \Im z , c \Im t , \Im \zeta, c \Re t ) | \\
& \ \ \ \ + \mathcal O (  ( |\Im z | + |\Im \zeta | + |t| ) ( |z| + |\zeta| )  + |t|^2 ) 
\\ & \geq |t| /C , \ \text{ if $ |z|,|\zeta| \ll 1 $,}  \end{split}
\]  
with the corresponding estimate for $ \bar \zeta^H $. 
Lemma \ref{l:flow1} and \eqref{eq:sJt} now show the almost analyticity of $ \mathscr C_H $} and Lemma \ref{l:geof} shows that $ \mathscr C_H $ 
is Lagrangian in the almost analytic sense: 
\[  ( \pi_L^* \omega_{T^* \CC^{2n}}  - \pi_R^* \omega_{T^* \CC^{2n} } )|_{\mathscr C_H } \sim 0 .\]
(See the appendix for the review of the almost analytic machinery and notation.) Lemma~\ref{l:SSR} shows that $\Delta ((S\cap \bar{S}_{\RR}\times (S\cap \bar{S})_{\RR})=(\mc{C}_H)_{\RR}$ is a submanifold, Lemma~\ref{l:zetac} shows that $\mc{C}_H$ is therefore a strictly positive almost analytic Lagrangean submanifold and hence, using~\eqref{e:project}, 
Lemma \ref{l:generate} now gives $ \psi_H = 
\psi_H  ( \alpha, \beta ) $ such that, 
\[ d_{\bar \alpha, \bar \beta } \psi_H  ( \alpha, \beta ) = 
\mathcal O \left( |\Im \alpha |^\infty + | \Im \beta |^\infty + |\Im 
\psi_H ( \alpha, \beta )|^\infty\right) , \]
and \eqref{eq:defCH} holds in the sense of equivalence of almost analytic manifolds (that is with $ \sim $ of \eqref{eq:h1h2} replacing the equality). In addition, in view of \eqref{eq:SSR} and \eqref{eq:sJt},
\begin{equation}
\label{eq:dpsiH}   \begin{split} d_\alpha \psi_H ( \alpha , \beta )|_{\beta=\alpha}  & = \Re (\zeta \cdot dz |_\Lambda ) , \\ 
d_\beta \psi_H ( \alpha, \beta ) |_{ \beta = \alpha } & = - \Re ( \zeta \cdot dz |_\Lambda ) ,\end{split}  \ \ \ \ \ \alpha \in \nbhd_{\RR^{2n} } ( 0), 
\end{equation}
and $ d_\alpha ( \psi_H ( \alpha, \alpha ) ) \sim 0 $. Hence we can choose
$ \psi_H ( \alpha, \alpha ) = 0 $. We also see that
\[   d_\alpha  \Im \psi_H ( \alpha , \beta )|_{\beta=\alpha} = 0 , \ \ 
d_\beta \Im \psi_H ( \alpha, \beta ) |_{ \beta = \alpha } = 0 ,  \ \ 
\ \ \alpha \in \nbhd_{\RR^{2n} } ( 0) , \]
which means that $ \Im \psi_H ( \alpha, \beta ) = \mathcal O ( | \alpha - \beta |^2 ) $, $ \alpha, \beta \in \nbhd_{\RR^{2n} } ( 0)  $, and the comparison with the case of $ G = 0$ shows that
\begin{equation}
\label{eq:ImpsiH} 
\Im \psi_H ( \alpha , \beta ) \sim | \alpha - \beta|^2 . 
\end{equation} 
Finally we 
return to \eqref{eq:LaeikH}: (recall that $ \zeta_j^H $ are the almost analytic extensions of $ \zeta_j^H $ from $ T^*{\Lambda}$
and that $ \{ \zeta_j^H , \zeta_k^H \} \sim 0 $):
\[ \begin{split}  \widehat{ \langle s, H_{\pi_L^* \zeta^H} \rangle } \pi_L^* \zeta^H_j  & = 
\sum_{ k=1}^n  ( H_{ s_k \zeta_k^H} \zeta_j^H +
\overline{ H_{ s_k \zeta_k^H } } \zeta_j^H ) \\
& = 
\mathcal O ( | \Im Z |^\infty ) , \ \ Z  = ( \alpha, \beta, \alpha^*, \beta^*), \end{split}  \]
with similar estimates for $ \pi_R^* \bar \zeta_j $'s. Hence using the
definition \eqref{eq:h1h2}. This implies that
\[  \pi_L^* \zeta^H_j ,\  \pi_R^* \bar \zeta^H_j \sim 0 \ \ \text{ on 
$ \mathscr C_H $.} \]
In view of the discussion above ($ \mathscr C_H $ equivalent to 
the right hand side of \eqref{eq:defCH}) we obtain
\[   \zeta_j^H ( \alpha, d_\alpha \psi_H ( \alpha, \beta ) ) 
= 
\mathcal O ( |\Im \alpha|^\infty + |\Im \beta|^\infty + 
| \Im   d_\alpha \psi_H  |^\infty + | \Im d_\beta
\psi_H  |^\infty ) , \]
with the same estimate for $ \bar \zeta_j^H ( \beta, - d_\beta \psi_H ( \alpha, \beta ) ) $.
This and \eqref{eq:ImpsiH} give \eqref{eq:Laeik}.

This completes the construction of the phase needed in Proposition \ref{p:BLa}. The construction of $ \mathscr C_H $ satisfying 
\eqref{eq:SSR}, \eqref{eq:diagoncond1} and~\eqref{eq:defSbS}  is equivalent, in the almost analytic sense, to the construction of $\psi_H $ satisfying \eqref{eq:ImpsiH}
and \eqref{eq:diagcond} that gives uniqueness of $ \psi_H $.

We have achieved more as the definition of $ \mathscr C_H $ shows
that, in the analytic case \eqref{eq:analCH}, 
$ \mathscr C_H \circ \mathscr C_H = \mathscr C_H $ (see \cite[Lemma 2]{GZ} for a simple linear algebra case). In general we have 
$ \mathscr C_H \circ \mathscr C_H \sim \mathscr C_H $ which for real 
values of $ \alpha $ and $ \beta $ means that
\[   {\rm{c.v.}}_\gamma \left( \psi_H ( \alpha, \gamma ) + 
\psi_H ( \gamma, \beta ) \right) = \psi_H ( \alpha, \beta ) + 
\mathcal O ( | \alpha - \beta |^\infty ) . \]

We now return to our original $ \psi $ in \eqref{eq:kerBLa}, 
$ \psi( \alpha, \beta ) = - i H ( \alpha ) + \psi_{H } ( \alpha, \beta) - i H( \beta  ) $. Our construction shows that 
\begin{equation}
\label{eq:psi2H}
\text{\eqref{eq:Laeik} holds, }   \ \psi ( \alpha, \alpha ) = - 2 i H ( \alpha ) , \  \ \psi ( \alpha, \beta ) = - \overline{ \psi ( \beta, \alpha ) },  \ \ 
\alpha, \beta  \in \Lambda . 
\end{equation}
The value of $ d_\alpha \psi $ on the diagonal, $ \zeta \cdot d z|_\Lambda $ is determined by \eqref{eq:defH} and \eqref{eq:dpsiH}.
In addition, $ \psi $ is uniquely determined, up to $ \mathcal O ( | \alpha - \beta|^\infty ) $, by \eqref{eq:psi2H}. 

Returning to the original problem of solving \eqref{eq:Laeik} we record our findings in 
\begin{prop}
\label{p:eikoLa}
With the convention of \eqref{eq:convent}, suppose that $ H $ is given by \eqref{eq:defH} and $ \zeta_j^\Lambda $,
$\tilde \zeta_j^\Lambda $ are defined in \eqref{eq:tildeZLa}. Then
there exists $ \psi \in \CI ( \Lambda \times \Lambda ) $, 
$ \Lambda  = \nbhd_{\RR^{2n} } ( 0 ) $, such that
\eqref{eq:Laeik} hold and $ \psi ( \alpha, \alpha ) = - 2 i H ( \alpha ) $. The function $ \psi $ is uniquely determined modulo
$ \mathcal O ( |\alpha - \beta|^\infty ) $. 
Moreover we have,
\begin{equation}
\label{eq:proppsi}
\begin{gathered}
 {\rm{c.v.}}_\gamma \left( \psi ( \alpha, \gamma ) + 2  i H( \gamma ) + 
\psi ( \gamma, \beta ) \right) = \psi ( \alpha, \beta ) + 
\mathcal O ( | \alpha - \beta|^\infty ) , \\
- H ( \alpha ) - \Im \psi ( \alpha, \beta ) - H( \beta )
\leq - | \alpha - \beta|^2 / C , \ \ C> 0 ,\\
(d_\alpha \psi) ( \alpha , \alpha ) = \zeta \cdot dz|_\Lambda . 
\end{gathered}
\end{equation}
\end{prop}

\subsubsection{Transport equations}
\label{s:trans}

Keeping the convention \eqref{eq:convent} we now solve the transport equations arising from \eqref{eq:BLaq1}. we start with a formal discussion (valid when 
all the objects are analytic). We first note that
in view of \eqref{eq:Laeik} and \eqref{eq:proppsi}
{for any $b(\alpha,\beta)\in $} analytic in a neighbourhood of 
$ 0 $ (in the notation of \eqref{eq:tildeZLa} and \eqref{eq:convent}),
\begin{equation}
\label{eq:ZjLa} 
\begin{split}  &  Z_j^\Lambda ( \alpha, h D_\alpha) \left( 
e^{ \frac i h \psi ( \alpha, \beta ) } b ( \alpha, \beta ) \right) 
= h e^{ \frac i h \psi ( \alpha, \beta ) }(   ( V_j   + c_j  ) b ( \alpha, \beta )  + \mathcal O ( h  ) ) , \\
&  \widetilde Z_j^\Lambda ( \beta, h D_\beta) \left( 
e^{ \frac i h \psi ( \alpha, \beta ) } b ( \alpha, \beta ) \right) 
= h e^{ \frac i h \psi ( \alpha, \beta ) } ( ( \widetilde V_j   + \widetilde c_j  ) b ( \alpha, \beta )   + \mathcal O ( h ) 
).
\end{split} 
\end{equation}
Here, 
\begin{gather*}
V_j := \langle V_j ( \alpha, \beta ) , \partial_\alpha \rangle, \ \ V_j (\alpha, \beta )_\ell  :=  \partial_{ \alpha_\ell ^*}  \zeta_j^\Lambda ( \alpha, d_\alpha \psi ( \alpha, \beta ) )    , \\ 
c_j ( \alpha , \beta  ) := \tfrac12 \sum_{\ell=1}^{2n}  
\partial_{\alpha_\ell}  V_j (\alpha, \beta ) + 
\zeta_{j1} ( \alpha , d_\alpha \psi ( \alpha, \beta ) {)}  \\ 
\ \ \ \ \ \ \ \ \ \ \  \ \ \ \ \ \ \ \ \ \ \ \ \ \ \ \ \ \ \ 
- \, i \sum_{k,\ell = 1}^{2n }
\partial_{\alpha_k \alpha_\ell}^2 \psi ( \alpha, \beta  ) \partial_{\alpha^*_k \alpha^*_\ell}^2 \zeta_j^\Lambda  ( \alpha , d_\alpha \psi ( \alpha, \beta  ) ) ,
\end{gather*}
with similar expressions coming from the applications $ \widetilde Z_j^\Lambda ( \beta , h D_\beta ) $: 
$ V_j $, $ c_j $, replaced by $ \widetilde V_j $, $ \widetilde 
c_j $,  and  with the roles of $ \alpha $ and $ \beta $ switched. 

A key observation here is that the holomorphic vector fields $ H_{\zeta_j^\Lambda ( \alpha ) } $ and
$ H_{\bar \zeta_j^\Lambda ( \beta )  } $ are tangent to 
$$ \mathscr C = \{ ( \alpha, d_\alpha \psi ( \alpha, \beta ) , 
\beta , d_\beta \psi ( \alpha, \beta ) ) : \alpha , \beta  \in 
\nbhd_{\CC^{2n}} ( 0 ) \} ,
$$
and that they commute.
In the parametrization of $ \mathscr C $ by $ ( \alpha, \beta ) $,
they are given by 
$ V_j $ and $ - \widetilde V_j $, respectively.  Hence,
\begin{equation}
\label{eq:comVjk0}
[ V_j , V_k ] = 0 , \ \ [ V_j, \widetilde V_k ] = 0  , \ \
[ \widetilde V_k, \widetilde V_k ] = 0 . 
\end{equation}

Hence, we seek  $ a $ of the form
\[ a ( \alpha, \beta ) \sim \sum_{k=0}^\infty h^k a_k ( \alpha, \beta ) , \]
where, we want to solve
\begin{equation}
\label{eq:ZjLa0} \begin{gathered} 
V_j   a_k (\alpha, \beta) + c_j ( \alpha, \beta ) 
a_k ( \alpha, \beta)  = F^j_{k-1} ( a_0, \cdots, a_{k-1} ) ( \alpha, \beta)  , \ \ F^j_{-1} \equiv 0, \\
\end{gathered} \end{equation}
with the corresponding expression 
involving $ \widetilde V_j $. 

Solving~\eqref{eq:ZjLa0} means that
\begin{equation}
\label{eq:ZjLa1} 
\begin{split} 
& Z_j^\Lambda ( \alpha, hD_\alpha ) \left( e^{ \frac i h \psi ( \alpha, \beta ) } \sum_{k=0}^{K-1} h^k a_k ( \alpha, \beta ) \right) =
h^{ K+1 } e^{ \frac i h \psi ( \alpha, \beta ) } F_{ {K-1}}^j ( \alpha, \beta ) , \\
& \widetilde Z_j^\Lambda ( \beta, hD_\beta ) \left( e^{ \frac i h \psi ( \alpha, \beta ) } \sum_{k=0}^{K-1} h^k a_k ( \alpha, \beta ) \right) =
h^{ K+1 } e^{ \frac i h \psi ( \alpha, \beta ) } \widetilde F_{ {K-1}}^j ( \alpha, \beta ) . \end{split}
\end{equation}
Since 
\begin{gather*}  [ Z_j^\Lambda ( \alpha, h D_\alpha ) , Z_k^\Lambda ( \alpha, h D_\alpha )] = 0 , \ \ 
[\widetilde Z_j^\Lambda ( 
\beta, h D_\beta ) , \widetilde Z_k^\Lambda ( 
\beta, h D_\beta ) ] = 0 , \\
[ Z_j^\Lambda ( \alpha, h D_\alpha ), \widetilde Z_k^\Lambda ( 
\beta, h D_\beta ) ] = 0 , \end{gather*}
we have from \eqref{eq:comVjk0} and \eqref{eq:ZjLa},
\begin{equation}
\label{eq:comVjk}  V_j c_k = V_k c_j , \ \ 
V_k \widetilde c_j = 
\widetilde V_j c_k, \ \
\widetilde V_k \widetilde c_j = \widetilde V_j \widetilde c_k . 
\end{equation}
Similarly, \eqref{eq:ZjLa1} gives 
\begin{equation}
\label{eq:comVjk1} 
\begin{gathered}  ( V_j + c_j ) F_{ {K-1}}^\ell = ( V_k + c_k ) F_{ {K-1}}^j , \ \
( \widetilde V_j + \widetilde c_j ) \widetilde F_{ {K-1}}^\ell = ( \widetilde V_k + \widetilde c_k ) \widetilde F^j_{ {K-1}} , \\
( V_j +  c_j ) \widetilde F_{ {K-1}}^\ell = ( \widetilde V_k + \widetilde c_k )  F^j_{ {K-1}} . 
\end{gathered} \end{equation}
Equations \eqref{eq:comVjk} and \eqref{eq:comVjk1} provide compatibility 
conditions for solving \eqref{eq:ZjLa0}:
\[   ( V_j + c_j ) a_k = F_{k-1}^j ,  \ \ ( \widetilde V_\ell + 
\widetilde c_\ell ) a_k = F_{k-1}^\ell , \ \ 
a_k ( \alpha, \alpha ) = b_k ( \alpha ) , \]
where  {the} $ b_k $'s  {are} prescribed. In fact, since  {the} $ V_\ell $'s and 
$ \widetilde V_j $ {'s} are independent when $ \alpha = \beta $ (as complex vectorfields), 
\[ \begin{gathered}  \CC^{2n}  \times \CC^n \times \CC^n \ni
( \rho, t , s ) \mapsto ( \alpha , \beta ) = \left( \exp 
\langle V, t \rangle  ( \rho  ) , 
\exp  \langle \widetilde V, s \rangle  (\rho ) \right) \in 
\CC^{2n} \times \CC^{2n} ,\\
\langle V, t \rangle :=  \sum_{j=1}^n t_j V_j  , \ \
\langle \widetilde V, s \rangle  := \sum_{\ell=1}^n s_j \widetilde V_\ell , 
\end{gathered} \]
is a local bi-holomorphic map onto 
of $ \nbhd_{ \CC^{4n} } ( \diag ( \Lambda \times \Lambda) ) $ (almost analytic in the general case). In view of 
this and of \eqref{eq:comVjk0}, \eqref{eq:comVjk},  {the following integrating factor}, $ g = g ( \alpha, \beta ) $,  is well defined (in the analytic case) on $ \nbhd_{ \CC^{4n} } ( \diag ( \Lambda \times \Lambda) ) $:
\[  g ( e^{ \langle V, t \rangle  }  ( \rho) , 
e^{ \langle  \widetilde V, s \rangle }  ( \rho ) ) := - 
\sum_{ j=1}^n \int_0^1 ( 
t_j c_j + s_j \widetilde c_j )|_{( \alpha, \beta) = ( e^{ \tau \langle V, t \rangle  }  ( \rho ) , 
e^{ \tau \langle  \widetilde V, s \rangle }  ( \rho ) ) } d\tau , \]
 {and satisfies}
\[ V_j g ( \alpha , \beta ) = c_j ( \alpha, \beta ) , \ \ 
 \widetilde V_j g ( \alpha , \beta ) = \widetilde c_j ( \alpha, \beta ), \ \ j =1, \cdots , n . \] 
We then define $ a_k ( \alpha, \beta ) $ inductively as follows: at
$ ( \alpha, \beta ) = 
( e^{  \langle V, t \rangle  }  ( \rho ) , 
e^{  \langle  \widetilde V, s \rangle }  ( \rho ) ) $, 
\[ \begin{split}   a_k ( \alpha, \beta ) & =  
e^{g(\alpha, \beta  ) } 
 b_k ( \rho ) \\
& + e^{g (\alpha, \beta  ) } \int_0^1 e^{-g ( \gamma, \gamma' ) } (t_j F_{k-1}^j 
( \gamma, \gamma' )  + s_j \widetilde F_{k-1}^j ( \gamma, \gamma' ) )|_{( \gamma, \gamma') = ( e^{ \tau \langle V, t \rangle  }  ( \rho ) , 
e^{ \tau \langle  \widetilde V, s \rangle }  ( \rho ) ) } d\tau .
\end{split} \]
The compatibility relations \eqref{eq:comVjk1} then show that \eqref{eq:ZjLa0} hold.

We now modify this discussion to the $ C^\infty $ case using 
almost analytic extensions as in \S \ref{A:trans}  and that provides
solutions of \eqref{eq:ZjLa0} for $ ( \alpha, \beta ) 
\in \Lambda \times \Lambda $ valid to infinite order at 
$ \diag( \Lambda \times \Lambda ) $ with any initial data on the diagonal. 

Hence we have solved \eqref{eq:BLaq1} locally 
near $ ( \alpha, \beta ) = ( 0 , 0 ) $. We now return to the original coordinates and note the uniqueness of the local construction gives 
us $ \psi $ and $ a $ in \eqref{eq:BLaq} satisfying \eqref{eq:propsi}
and \eqref{eq:propaa}. This completes the proof of Proposition \ref{p:BLa}.

\subsection{Projection property}

It remains to choose $ a|_{\Delta } $ so that $ B_\Lambda^2 \equiv B_\Lambda$. 
From \eqref{eq:proppsi} we already know that the phase in \eqref{eq:kerBLa} has the correct composition property 
and hence we need to find the amplitude 
$a(\alpha, \beta)$. From Proposition \ref{p:BLa} {\em it is enough to determine $a$ on the diagonal}. For that we consider the kernel of $ B_\Lambda^2$ on the diagonal:
\begin{equation}
\label{eq:KB2} K_{B^2_\Lambda } ( \alpha, \alpha ) = h^{-2n} \int_\Lambda e^{ \frac i h ( \psi ( \alpha, \beta ) + 2 i H( \beta) 
+ \psi ( \beta, \alpha ) ) } a ( \alpha, \beta , h ) a ( \beta, \alpha, h ) d m_\Lambda ( \beta ) . \end{equation}
We note that the support property of $ a $ in \eqref{eq:propsi} 
implies that the integration takes place over a bounded set 
$ | \beta_\xi | \leq C \langle \alpha_\xi \rangle $.
Application of complex stationary phase to~\eqref{eq:KB2} yields
\begin{equation}
\label{eq:KB21} 
K_{B^2_\Lambda}=h^{-n} e^{\frac{i}{h}\psi(\alpha,\beta)}c(\alpha,\beta)
, \ \ 
c(\alpha,\alpha) \sim \sum_{j}h^jL_{2j}a(\alpha,\gamma, h)a(\gamma,\alpha, h )|_{\gamma=\alpha} , 
\end{equation}
where $L_{2j}$ are differential operators of order $2j$ in $\gamma$ and 
$$L_0|_{\Delta}=f(\alpha), \ \  | f ( \alpha ) | > 0 , \ \ 
\Delta := \Delta ( \Lambda \times \Lambda ) .$$
 Since $\psi(\alpha,\beta)=-\overline{\psi(\beta,\alpha)}$, $f(\alpha)\in\RR$.
(Strictly speaking we should again proceed with the rescaling 
\eqref{eq:tildealbe} and we are tacitly using the convention \eqref{eq:convent} here.)

Writing 
$
a\sim \sum_{j}h^ja_j, 
$
we have 
$$
c(\alpha,\beta )\sim \sum_j h^j c_j (\alpha, \beta )  , \ \ 
c_j ( \alpha, \alpha ) = \sum_{k+\ell+m=j}L_{2k}
 a_\ell(\alpha,\gamma)a_{m}(\gamma,\alpha)  |_{\gamma=\alpha} .
$$
We note that if $ a ( \alpha , \beta )  = \overline { a ( \beta, \alpha ) } $ then $ B_\Lambda $ is self-adjoint and hence so is
$ B_\Lambda^2 $. That means in particular that $ c ( \alpha, \alpha ) 
$ is real. 
Hence if $ a_\ell ( \alpha, \beta ) = \overline{ a_\ell ( \beta, \alpha ) }$ for $ \ell \leq M $, then
$ c_\ell |_{\Delta} \in \RR $ for $ \ell \leq M $. Since
\[ b_M ( \alpha, \alpha ) = 2 f ( \alpha ) a_0 ( \alpha, \alpha)
a_M ( \alpha, \alpha ) + \sum_{\substack{ k+\ell+m =M\\\ell , m < M} } 
L_{2k }  a_\ell ( \alpha, \gamma ) a_m ( \gamma, \alpha) |_{ \gamma = \alpha}   , \]
it follows that 
\begin{equation}
\label{eq:aell} a_\ell ( \alpha , \beta )  =  \overline{ a_\ell ( \beta , \alpha ) }
 , \ \ell < M  \ \Longrightarrow \  
 \sum_{\substack{ k+\ell+m =M\\\ell , m < M} } 
L_{2k}  a_\ell ( \alpha, \gamma ) a_m ( \gamma, \alpha) |_{ \gamma = \alpha}  \in \RR .  
 \end{equation}
We iteratively solve the following sequence of equations
\begin{equation}
\label{eq:iterate}
\sum_{k+\ell+m=j}L_{2k}a_\ell(\alpha,\gamma)a_{m}(\gamma,\alpha)|_{\gamma=\alpha}=a_j(\alpha,\alpha)
\end{equation}
with $a_j|_{\Delta}$ real. 
Proposition \ref{p:BLa} then gives us the desired $ a ( \alpha, \beta ) $.
First, let
$$
a_0(\alpha,\alpha)=\frac{1}{f(\alpha)}\in C^\infty(T^*\RR^n;\RR)
$$
so that $f(\alpha)a_0(\alpha,\alpha)^2=a_0(\alpha,\alpha)$ (i.e.~\eqref{eq:iterate} is solved for $j=0$). The proof of Proposition 
\ref{p:BLa} (see \S \ref{s:trans}) shows that we can then find 
$ a_0 ( \alpha, \beta ) $ so that \eqref{eq:ZjLa1} holds with $ K = 0 $
and $ a_0 |_\Delta = 1/f( \alpha ) $. 

Assume now that~\eqref{eq:iterate} is solved for $j\leq M-1$.  Then,~\eqref{eq:iterate} with $j=M$ reads
\begin{align*}
a_M(\alpha,\alpha)&=\sum_{k+\ell+m=M}L_{2k}a_\ell(\alpha,\gamma)a_{m}(\gamma,\alpha)|_{\gamma=\alpha}\\
&=2a_M(\alpha,\alpha)+\sum_{\substack{k+\ell+m=M\\\ell,m<M}}L_{2k}a_\ell(\alpha,\gamma)a_{m}(\gamma,\alpha)|_{\gamma=\alpha}
\end{align*}
Putting
$$
a_M(\alpha,\alpha)=-\sum_{\substack{k+\ell+m=M\\\ell,m<M}}L_{2k}a_\ell(\alpha,\gamma)a_{m}(\gamma,\alpha)|_{\gamma=\alpha}
$$ 
we solve~\eqref{eq:iterate} for $j=M$. From \eqref{eq:aell} we see that
$ a_M ( \alpha, \alpha ) $ is real.
The argument in \S \ref{s:trans} provides the construction of $ a_M  $ from 
$ a_\ell $, $ \ell < M $ and $ a_k|_\Delta $. Taking an almost analytic continuation with $a_M(\alpha,\beta)=\overline{a_M(\beta,\alpha)}$ then completes the construction of $a_M$ and hence by induction and the Borel summation lemma we have, in the notation of Proposition \ref{p:BLa}, 
\begin{equation}
\label{eq:a2b}
c =a+ \mathcal O_\infty . 
\end{equation}
This gives the following
\begin{prop}
\label{p:Bla1}
There exists a unique choice of $ b_j ( \alpha ) $ in Proposition 
\ref{p:BLa} for which the operator $ B_\Lambda $ defined by 
\eqref{eq:kerBLa} satisfies 
\begin{equation}
\label{eq:propBLa1}
B_\Lambda = B_\Lambda^{*,H} , \ \ 
B_\Lambda =   B_\Lambda^2 + \mathcal O ( h^N )_{ {\langle \xi\rangle^{N}L^2(\Lambda)\to \langle \xi\rangle^{-N}L^2(\Lambda)} }  , 
\end{equation}
for all $ N $. 
\end{prop}
\begin{proof}
In view of \eqref{eq:a2b} we need to check that for $ r = \mathcal O_\infty $ (in the notation of
Proposition \ref{p:BLa}), for all $ N $,
\[  R = \mathcal O ( h^N )_{ {\langle \xi\rangle^{N}L^2(\Lambda)\to \langle \xi\rangle^{-N}L^2(\Lambda)}  }  , \ \ 
R u ( \alpha ) := 
h^{-n} \int_\Lambda r ( \alpha, \beta, h ) e^{ \frac i h \psi ( \alpha, \beta ) } u ( \beta ) d m_\Lambda ( \beta ) . \]
But this is an immediate consequence of \eqref{eq:BLap} and 
Schur's criterion for boundedness on $ L^2 $.
\end{proof}

\subsection{Construction of the projector}
\label{s:cotp} 

We now show that the exact
orthogonal projector $ \Pi_\Lambda : L^2 ( \Lambda )  \to H( \Lambda ) $ satisfies
\begin{equation}
\label{eq:B2Pi1}
\Pi_\Lambda = B_\Lambda + \mathcal O ( h^\infty )_{{ {\langle \xi\rangle^{N}L^2(\Lambda)\to \langle \xi\rangle^{-N}L^2(\Lambda)}  }} ,
\end{equation}
for all $ N $.
For that we follow
the proof of \cite[Proposition 1.1, formula (1.46)]{Sj96} which is 
related to an earlier construction in \cite[Step 3, Proof of Corollary A.4.6]{BG}. 

We start with the exact projector $ P_\Lambda $:
\[ P_\Lambda ( L^2_\Lambda ( \Lambda ) ) = H( \Lambda )  , \ \ 
P_\Lambda^2 = P_\Lambda , \ \ P_\Lambda = \mathcal O ( 1 )_{L^2 ( \Lambda ) \to L^2 ( \Lambda ) } , \]
given by 
\[ P_\Lambda = T_\Lambda S_\Lambda . \]
For a real valued $ f \in S (  \Lambda ) $, 
\[ f ( \alpha ) \sim \sum_{ k=0}^\infty 
f_k ( \alpha ) (h/\langle \alpha_\xi \rangle )^k , \ \ 
 f_0 ( \alpha) > 1/C  , \]
  we define the following self-adjoint operator:
\[ A_f :=  P_\Lambda f P_\Lambda^{*, H } , \ \ 
A_f u ( \alpha ) =: h^{-n} \int_\Lambda e^{ \frac i h \psi_1 ( \alpha, \beta ) } a_f ( \alpha, \beta , h) u ( \beta ) e^{ - 2 H ( \beta ) /h} dm_\Lambda ( \beta ) ,\]
where $ \psi_1 $ and $ a_f $ were obtained using the method of complex stationary phase (again it is justified using using the rescaling \eqref{eq:tildealbe})

We claim that $\psi_1 = \psi + 
 \mathcal O_\infty  $ (in the notation of Proposition \ref{p:BLa}. 
Indeed, 
since $ A_f^{ *, H} = 
 {A_f} $ and $ P_\Lambda = T_\Lambda  S_\Lambda $, 
the arguments leading to \eqref{eq:Laeik} apply and $ \psi_1 $ satisfies
the same eikonal equations. Similarly, $ a_f ( \alpha, \beta , h ) $
satisfies transport equations implied by \eqref{eq:BLaq}. 
Arguing as in the proof of Lemma \ref{l:TSbounded} 
we find the value of $ \psi_1|_\Delta $ to be 
\[ \psi_1 ( \alpha, \alpha ) + 2 i H ( \alpha ) = {\rm{c.v.}}_{\beta } \left( \Psi ( \alpha, \beta )  - \overline { \Psi ( \alpha, \beta ) }  \right) = 0 . \]
We then invoke the uniqueness statement in Proposition \ref{p:BLa}.

If we can choose $ f $ so that 
$   a_f|_\Delta =  a|_{\Delta }  + \mathcal O 
( ( h/\langle \alpha_{\xi} \rangle )^\infty ) $,
with $ a $ in \eqref{eq:kerBLa}, 
then the same uniqueness statement shows that $ a_f = a + \mathcal O_\infty $. Hence
\begin{equation}
\label{eq:afaim}   a_f|_\Delta =  a|_{\Delta }  + \mathcal O 
( ( h/\langle \alpha_{\xi} \rangle )^\infty )  
\ \Longrightarrow \ A_f = B_\Lambda + \mathcal O ( h^\infty )_{{\langle \xi\rangle^{N}L^2(\Lambda)\to \langle \xi\rangle^{-N}L^2(\Lambda)} } .
\end{equation}
determined by its value on the diagonal and we find 
that using $ \Psi $ given in \eqref{e:tempProjectorPhase} and
satisfying \eqref{eq:dalPh} 
\[ \begin{split} \psi_1 ( \alpha, \alpha ) + 2 i H ( \alpha ) & = 
{\rm{c.v.}}_{\beta } \left( \Psi ( \alpha, \beta ) - 
\overline{ \Psi ( \alpha, \beta ) } + 2 i H ( \alpha ) - 
2 i H ( \beta ) \right) = 0 . 
\end{split} \]
But this means that \eqref{eq:psi2H} holds for $ \psi_1 $ and hence
 $ \psi_1 ( \alpha, \beta ) = \psi( \alpha, \beta ) +  \mathcal O ( |\alpha - \beta |^\infty ) $. 
 
  {We next choose }$ f $ so that 
 $ A_f = B_\Lambda + \mathcal O ( h^\infty)_{{\langle \xi\rangle^{N}L^2(\Lambda)\to \langle \xi\rangle^{-N}L^2(\Lambda)} } .$

Writing $ a_f ( \alpha, \beta ) \sim \sum_{k=0}^\infty 
( h / \langle \alpha_\xi \rangle )^j a_{f,j} ( \alpha, \beta ) $, 
we proceed as in \S \ref{s:trans}: with different $ L_{2k} $'s, 
$ g := L_0 |_\Delta \neq 0 $,  
\[ a_{f,j} ( \alpha, \alpha ) =
\sum_{ k + \ell = j } L_{2k } f_\ell ( \alpha ) = g ( \alpha ) f_j ( \alpha) + \sum_{\substack{k+ \ell = j\\\ell < j} } L_{2k } f_\ell ( \alpha ). \]
(In our special case, the amplitude in $ P $ is constant which is not the case in generalizations -- but the argument works easily just the same.) Using this, solving $ a_{f,j} ( \alpha ) = a_j ( \alpha ) $ 
for $ f $ is immediate. As in the construction of the amplitude of 
$ B_\Lambda $ in \S \ref{s:trans} we see that $ f $ is real valued and that $ f_0 $ is bounded from below.

We can now follow \cite{Sj96} and complete the proof of \eqref{eq:B2Pi1}.
We record this statement as 
\begin{prop}
\label{p:B2P}
Suppose that $ \Pi_\Lambda$ is orthogonal projector from 
$ L^2 ( \Lambda ) $ to $ H ( \Lambda ) $ and that $ B_\Lambda $ is given 
by Proposition \ref{p:Bla1}. Then 
\begin{equation}
\label{eq:B2Pi11}
\Pi_\Lambda = B_\Lambda + \mathcal O ( h^\infty )_{{\langle \xi\rangle^{N}L^2(\Lambda)\to \langle \xi\rangle^{-N}L^2(\Lambda)} } ,
\end{equation}
for all $ N $.
\end{prop}
\begin{proof}
To start we observe that for $ u \in H ( \Lambda ) $, 
$ \| u \|_{ L^2 ( \Lambda ) } > 0 $, 
\[ \begin{split}  \langle A_f u, u \rangle_{ L^2 ( \Lambda) }  & = 
\langle P_\Lambda f P_\Lambda^* u, u \rangle_{ L^2 ( \Lambda ) } = 
\langle  f P_\Lambda^* u, P_\Lambda^* u \rangle_{ L^2 (\Lambda) } 
\\ & \geq \min_{\alpha \in T^* \RR^n }  f 
( \alpha ) \| P_\Lambda^* u \|^2_{L^2 ( \Lambda )  } 
%\\ & 
\geq 
 \frac{| \langle P_\Lambda^* u , u \rangle |^2}{ C \| u \|^2_{L^2 ( \Lambda)  }}  =  \| u \|_{L^2 ( \Lambda ) }^2  /C .  \end{split} 
 \]
Hence, 
\begin{equation}
\label{eq:specAf}
\begin{gathered} \| u \|_{L^2 ( \Lambda)  } /C  \leq \| A_f u \|_{L^2 ( \Lambda)  } \leq {C}\| u \|_{L^2 ( \Lambda ) } , \ \ u \in H ( \Lambda )  , \\
A_f u = 0, \ \ u \in H ( \Lambda ) ^\perp , \ \ A_f^* = A_f, 
\end{gathered}
\end{equation}
and
\begin{equation}
\label{eq:contPi}
\Pi_\Lambda = \frac{1}{2 \pi} \int_\gamma ( \lambda - A_f )^{-1} d\lambda,  
\end{equation}
where $ \gamma $ is a positively oriented boundary of
an open set in $ \CC $ containing $ [ 1/C ,  {C} ] $ and excluding $ 0 $.
From \eqref{eq:afaim} and Proposition \ref{p:Bla1} we know that 
\begin{equation}
\label{eq:AfAf2}  A_f = A_f^2 + \mathcal O ( h^\infty )_{{\langle \xi\rangle^{N}L^2(\Lambda)\to \langle \xi\rangle^{-N}L^2(\Lambda)} }, 
\end{equation}
 and we want to use this property to show that $ \Pi_\Lambda $ is close to $ A_f $.
For that we note that if $ A = A^2 $ then, at first for $ |\lambda  | \gg 1 $, 
\[ ( \lambda - A )^{-1} = \sum_{j=0}^\infty  
  \lambda^{-j-1} A^j = \lambda^{-1} + \lambda^{-1} 
  \sum_{ j=0}^\infty \lambda^{-j} A = \lambda^{-1} + A \lambda^{-1} 
  ( \lambda - 1 )^{-1}  .\]
Hence, it is natural to take the right hand side as the approximate inverse in the case when $ A^2 - A $ is small: 
\[ \begin{split} ( \lambda - A_f) ( \lambda^{-1}  +  A_f \lambda^{-1}  ( \lambda - 1)^{-1} )   & = 
I - ( A_f^2 - A_f ) \lambda^{-2}  ( \lambda - 1)^{-1}  
 . \end{split} \]
In view of \eqref{eq:AfAf2} and for $ h$ small enough, the right
hand side is invertible for $ \lambda \in \gamma $ with the inverse equal to $ I + R $, $ R = \mathcal O ( h^\infty )_{{\langle \xi\rangle^{N}L^2(\Lambda)\to \langle \xi\rangle^{-N}L^2(\Lambda)} }  $. Hence for $ \lambda \in \gamma $, 
\[ (\lambda - A_f )^{-1} =  \lambda^{-1} + \lambda^{-1} ( \lambda -1)^{-1} A_f + \mathcal O ( h^\infty )_{{\langle \xi\rangle^{N}L^2(\Lambda)\to \langle \xi\rangle^{-N}L^2(\Lambda)}}  . \]
Inserting this identity into \eqref{eq:contPi} and using Cauchy's formula gives 
\[  \Pi_\Lambda = A_f + \mathcal O ( h^\infty )_{{\langle \xi\rangle^{N}L^2(\Lambda)\to \langle \xi\rangle^{-N}L^2(\Lambda)} } , \]
which combined with \eqref{eq:afaim} implies
\eqref{eq:B2Pi11}.
\end{proof} 

\section{Deformation of pseudodifferential operators}
\label{s:deforpsi}

In this section we analyse pseudodifferential operators with analytic symbols acting on spaces $ H^m_\Lambda $ defined in \S \ref{s:dFBI}. 
That means describing the action on the FBI side of operators $ P$:
\begin{equation}
\label{eq:TLPSL}  T_\Lambda P u = (T_\Lambda P S_\Lambda) (T_\Lambda u ) = 
( \Pi_\Lambda T_\Lambda P S_\Lambda \Pi_\Lambda ) ( T_\Lambda u ) .
\end{equation}

The class of pseudodifferential operators we consider is given by 
\begin{equation}
\label{eq:defP}
Pu(y):=\frac{1}{(2\pi h)^n}\int_{\RR^n}\int_{\RR^n}  e^{\frac{i}{h}\langle y-y',\eta\rangle}p(y,\eta)u(y')dy'd\eta
\end{equation}
where 
$p\in S^m ( T^* \TT^n ) $ has an analytic continuation from $T^*\mathbb{T}^n$ satisfying
\begin{equation}
\label{e:SymbolAnalyticity}
|p(z,\zeta)|\leq M\langle \zeta\rangle^m, \ \text{ for }  \ |\Im z|\leq a, \ \ |\Im \zeta|\leq  b\langle \Re \zeta\rangle .
\end{equation}
The integral in the definition \eqref{eq:defP} of $ Pu $  is considered in the sense of oscillatory integrals (see for instance \cite[\S 5.3]{zw}) 
and we extend both $ y \mapsto u ( y ) $ and $ y \mapsto p ( y, \eta ) $
to periodic functions on $ \RR^n $.

\subsection{Pseudodifferential operators as Toeplitz operators}
\label{s:ps2To}

We start with a lemma which describes the middle term in \eqref{eq:TLPSL}:
\begin{lemm}
\label{l:psdef}
Suppose $ P $ is defined by \eqref{eq:defP} with $ p $ satisfying
\eqref{e:SymbolAnalyticity}. Then, for $G$ satisfying~\eqref{e:controlDeformation} with $\epsilon_0>0$ small enough, the Schwartz kernel of 
$T_\Lambda P {S}_\Lambda$ is given by 
\begin{equation}
\label{eq:strKP}
K_P(\alpha,\beta)=c_0 h^{-n}e^{\frac{i}{h}\Psi(\alpha,\beta)}a_P(\alpha,\beta) +r(\alpha,\beta)\
\end{equation}
where $\Psi$ is as in~\eqref{e:tempProjectorPhase},
\begin{equation}
\label{eq:defaP}
a_P \sim\sum_{j=0}^\infty {h^j}{\langle \alpha_\xi\rangle^{-j}} a_j,\qquad a_0(\alpha,\alpha)=p|_{\Lambda}(\alpha),
\end{equation}
$a_j \in S^0 ( \Lambda \times \Lambda ) $ {is supported in a conic neighbourhood of $\Delta(\Lambda\times \Lambda)$} and 
\begin{equation}
\label{eq:estr}
| r(\alpha,\beta)|\leq e^{- ( \langle  \Re \alpha_\xi\rangle+\langle  \Re\beta _\xi\rangle +\langle\Re \alpha_x-\Re\beta_x \rangle  )/Ch}. 
\end{equation}
\end{lemm}
\begin{proof}
We first note that for each $\beta\in \Lambda$, $v_\beta(y')=e^{-\frac{i}{h}\varphi^*(\beta,y')}b(\beta_x-y',\beta_\xi)$ is a Schwartz function and hence the integral
$$
h^{-\frac{3n}{4}}\frac{1}{(2\pi h)^n}\int_{\mathbb{R}^{2n}} e^{\frac{i}{h}(\langle y-y',\eta\rangle-\varphi^*(\beta,y'))}p(y,\eta)b(\beta_x-y',\beta_\xi)dy'd\eta
$$
defines a Schwartz function of $y$. In particular, the kernel of $T_{\Lambda}PS_{\Lambda}$ is given by 
$$
\frac{h^{-\frac{3n}{2}}}{(2\pi h)^n}\int_{\mathbb{R}^{3n}} e^{\frac{i}{h}(\varphi(\alpha,y)+\langle y-y',\eta\rangle-\varphi^*(\beta,y'))}p(y,\eta)b(\beta_x-y',\beta_\xi)\langle \alpha_\xi\rangle^{\frac{n}{4}}dy'd\eta dy.
$$
To obtain \eqref{eq:strKP} we start by 
deforming the contour in $\eta$:
$
\eta\mapsto \eta+i\delta_1\langle \eta\rangle\frac{y-y'}{\langle y-y'\rangle}.
$
The phase $\Phi$ is then given by 
\[
\begin{split}
\Phi & =\langle \alpha_x-y,\alpha_\xi\rangle +\frac{i\langle \alpha_\xi\rangle}{2}(\alpha_x-y)^2+\frac{i\langle \beta_\xi\rangle}{2}(\beta_x-y')^2 \\
& \ \ \ \ \ \ \ +\langle y'-\beta_x,\beta_\xi\rangle +\langle y-y',\eta\rangle +i\delta_1\langle \eta\rangle\frac{(y-y')^2}{\langle y-y'\rangle}.
\end{split}
\]
We then deform the contour in $y,y'$ as follows 
$$
y\mapsto y+i\delta_1\frac{\eta-\alpha_\xi}{\langle \eta-\alpha_\xi\rangle},\qquad y'\mapsto y'+i\delta_1\frac{\beta_\xi-\eta}{\langle \beta_\xi-\eta\rangle}.
$$
The phase $\Phi$ becomes
\[ \begin{split} 
\Phi&=\langle \alpha_x-y,\alpha_\xi\rangle +\frac{i\langle \alpha_\xi\rangle}{2}(\alpha_x-y)^2+\frac{i\langle \beta_\xi\rangle}{2}(\beta_x-y')^2+\langle y'-\beta_x,\beta_\xi\rangle +\langle y-y',\eta\rangle \\
&\ \ \ \ \ \ \ \ +i\delta_1\Big[\frac{(\alpha_\xi-\eta)^2}{\langle \alpha_\xi-\eta\rangle}+\frac{(\beta_\xi-\eta)^2}{\langle \beta_\xi-\eta\rangle}+  \langle \eta\rangle\frac{(y-y')^2}{\langle y-y'\rangle}\Big]\\
&\ \ \ \ \ \ \ \ +\frac{i\langle \alpha_\xi\rangle}{2}\Big[ -2i\delta_1 \frac{\langle \alpha_\xi-\eta,\alpha_x-y\rangle}{\langle \alpha_\xi-\eta\rangle}-\delta_1^2\frac{(\alpha_\xi-\eta)^2}{\langle \alpha_\xi-\eta\rangle^2}\Big]\\
&\ \ \ \ \ \ \ \  +\frac{i\langle \beta_\xi\rangle}{2}\Big[ -2i\delta_1 \frac{\langle \eta-\beta_\xi,\beta_x-y'\rangle}{\langle \beta_\xi-\eta\rangle}-\delta_1^2\frac{(\beta_\xi-\eta)^2}{\langle \beta_\xi-\eta\rangle^2}\Big]\\
&\ \ \ \ \ \ \ \ + \mathcal O \Big(\delta_1^2\langle \eta\rangle \frac{y-y'}{\langle y-y'\rangle}\Big[\frac{|(\alpha_\xi-\eta)\langle \beta_\xi-\eta\rangle +(\eta-\beta_\xi)\langle \alpha_\xi-\eta\rangle |}{|\langle \alpha_\xi-\eta\rangle\langle \beta_\xi-\eta\rangle|}\Big]\Big)\\
&\ \ \ \ \ \ \ \  + \mathcal O \Big(\frac{\delta_1^3\langle \eta\rangle}{\langle y-y'\rangle}\Big[\frac{|(\alpha_\xi-\eta)\langle \beta_\xi-\eta\rangle +(\eta-\beta_\xi)\langle \alpha_\xi-\eta\rangle |^2}{|\langle \alpha_\xi-\eta\rangle\langle \beta_\xi-\eta\rangle|^2}\Big]\Big)
\end{split} \]
We first consider the case when $\langle \Re \alpha_\xi\rangle\geq 2\langle \Re \beta_\xi\rangle $. Then,
$$
|\Re\alpha_\xi-\eta|+|\Re \beta_\xi-\eta|\geq c(\langle \Re \alpha_\xi\rangle +\langle \Re\beta_\xi\rangle+\langle \eta\rangle).
$$
and in particular, 
$$\Im \Phi\geq c(\langle \Re \alpha_\xi\rangle +\langle\Re \beta_\xi\rangle +\langle \Re\eta\rangle +c(|\Re\alpha_x-y|+|\Re \beta_x-y' |+|y-y'|) ,$$
which produces  produces a term which can be absorbed into $ r $ satisfying
\eqref{eq:estr}.

Similar arguments, show that we can assume that $\langle\Re \alpha_\xi\rangle ,\,\langle\Re \eta\rangle ,$ and $\langle\Re\beta_\xi\rangle$ are proportional.

We now suppose that 
$$
\frac{|\Re\alpha_\xi-\Re\beta_\xi|}{\langle \Re\alpha_\xi\rangle+\langle \Re\beta_\xi\rangle}+|\Re\alpha_x-\Re \beta_x|>\delta.
$$
Then, the imaginary part of the phase is bounded below by
\[ 
\begin{split} 
\Im \Phi & \geq c(\langle \Re\alpha_\xi\rangle +\langle\Re \beta_\xi))(1+|\Re\alpha_x-\Re\beta_x|) \\
& \ \ \ \ \ \ +c(|\Re\alpha_x-y|+|y-y'|+|\Re\beta_x-y|+|\eta-\Re\alpha_\xi|+|\eta-\Re\beta_\xi|).
\end{split} \]
In particular when
$$
\frac{|\Re\alpha_\xi-\Re\beta_\xi|}{\langle\Re \alpha_\xi\rangle+\langle \Re\beta_\xi\rangle}+|\Re\alpha_x-\Re\beta_x|>\delta.
$$
the integral is bounded by 
$
Ce^{-(\langle \Re\alpha_\xi\rangle +\langle \Re\beta_\xi\rangle)(1+|\Re\alpha_x-\Re \beta_x|)/h}
$.
Hence, we can insert a cutoff 
$$
\chi\Big(\delta^{-1}\Big[\frac{|\Re\alpha_\xi-\Re\beta_\xi|}{\langle\Re \alpha_\xi\rangle+\langle\Re \beta_\xi\rangle}+|\Re\alpha_x-\Re\beta_x|\Big]\Big)
$$
into the integral.

{With this cutoff inserted,} we deform in $y,y'$ to the critical point
$$
y\mapsto y+y_c(\alpha,\beta),\qquad y'\mapsto y'+y_c(\alpha,\beta),
$$
where
$$
y_c(\alpha,\beta)=\frac{\alpha_x\langle \alpha_\xi\rangle +\beta_x\langle \beta_\xi\rangle}{\langle \alpha_\xi\rangle +\langle \beta_\xi\rangle}+i\frac{\beta_\xi-\alpha_\xi}{\langle \alpha_\xi\rangle +\langle \beta_\xi\rangle}.
$$
This contour deformation is justified since the cutoff function guarantees that
$$
\Re\Big|\frac{\alpha_\xi-\beta_\xi}{\langle \alpha_\xi\rangle+\langle \beta_\xi\rangle}\Big|\leq C\delta.
$$
The phase is then given by 
\[ \begin{split}
\Phi & =  \frac{i}{2} \left(\frac{(\beta_\xi-\alpha_\xi)^2+\langle \alpha_\xi\rangle\langle \beta_\xi\rangle(\alpha_x-\beta_x)^2}{\langle \alpha_\xi\rangle +\langle \beta_\xi\rangle}+\langle \alpha_\xi\rangle y^2+\langle \beta_\xi\rangle (y')^2\right) \\
& \ \ \ \ \ \ +\Big\langle \alpha_x-\beta_x,\frac{\alpha_\xi \langle \beta_\xi\rangle +\beta_\xi\langle \alpha_\xi\rangle}{\langle \alpha_\xi\rangle +\langle \beta_\xi\rangle}\Big\rangle 
+\langle y-y',\eta-\eta_c(\alpha,\beta)\rangle
\end{split} \]
with
$$
\eta_c(\alpha,\beta)=\frac{\alpha_\xi \langle \beta_\xi\rangle +\beta_\xi\langle \alpha_\xi\rangle}{\langle \alpha_\xi\rangle +\langle \beta_\xi\rangle}+i\frac{\langle \alpha_\xi\rangle \langle \beta_\xi\rangle(\beta_x-\alpha_x)}{\langle \alpha_\xi\rangle +\langle \beta_\xi\rangle}.
$$

We would now like to shift the countour to $\eta\mapsto \eta+\eta_c$. However, $p$ only has an analytic continuation to $|\Im \eta|\leq b\langle\Re\eta\rangle$ and $\Im \eta_c$ is not, in general, bounded. Therefore, when $|\Re\eta|\ll |\Re\eta_c|$, we cannot make this deformation. To finish the proof, we consider two cases.

We first assume that 
$ |\eta_c(\alpha,\beta)|\leq {b}/{2}. $
 Then, the contour deformation $\eta\mapsto \eta+\eta_c$ is justified, and we may perform complex stationary phase to complete the proof. 
 
We now consider the more involved case when
\[ |\eta_c|\geq \frac{b}{2}\gg\epsilon_0  > 0  \]
where $\e_0$ is as in~\eqref{e:controlDeformation}.
In that case we use the deformation 
$$
y\mapsto y+i\delta_1\frac{(\eta-\eta_c)}{\langle \eta-\eta_c\rangle},\qquad y'\mapsto y'-i\delta_1 \frac{(\eta-\eta_c)}{\langle \eta-\eta_c\rangle}
$$
to obtain  the phase
\begin{align*}
&\Big\langle \alpha_x-\beta_x,\frac{\alpha_\xi \langle \beta_\xi\rangle +\beta_\xi\langle \alpha_\xi\rangle}{\langle \alpha_\xi\rangle +\langle \beta_\xi\rangle}\Big\rangle +\frac{i}{2}\Big[ \frac{(\beta_\xi-\alpha_\xi)^2+\langle \alpha_\xi\rangle\langle \beta_\xi\rangle(\alpha_x-\beta_x)^2}{\langle \alpha_\xi\rangle +\langle \beta_\xi\rangle}+\langle \alpha_\xi\rangle y^2+\langle \beta_\xi\rangle (y')^2\Big]\\
&\qquad+\Big\langle y-y',(\eta-\eta_c)\Big(1-\frac{\delta_1}{\langle \eta-\eta_c\rangle}\Big)\Big\rangle+2i\delta_1\frac{(\eta-\eta_c)^2}{\langle \eta-\eta_c\rangle}\Big(1-\frac{\delta_1}{2\langle \eta-\eta_c\rangle}\Big).
\end{align*}
Finally, let $\chi \in C_c^\infty( (1/2,2))$ withe $\chi\equiv 1$ on $(3/4,3/2)$, and  shift contours 
$$
\eta\mapsto \eta+\eta_c\chi\Big(\frac{|\Re\eta|}{|\Re\eta_c|}\Big).
$$
 Note that this deformation is now justified since $|\Re\eta|\geq c|\Re\eta_c|$ on the deformation and $|\Im \eta_c|\leq c\epsilon_0 \langle \Re \eta_c\rangle.$ The phase is then given by
  \begin{align*}
&\ \frac{i}{2}\Big[ \frac{(\beta_\xi-\alpha_\xi)^2+\langle \alpha_\xi\rangle\langle \beta_\xi\rangle(\alpha_x-\beta_x)^2}{\langle \alpha_\xi\rangle +\langle \beta_\xi\rangle}+\langle \alpha_\xi\rangle y^2+\langle \beta_\xi\rangle (y')^2\Big]\\
&\qquad \ \  
+ \left\langle \alpha_x-\beta_x,\frac{\alpha_\xi \langle \beta_\xi\rangle +\beta_\xi\langle \alpha_\xi\rangle}{\langle \alpha_\xi\rangle +\langle \beta_\xi\rangle}\right\rangle +\Big\langle y-y',(\eta-(1-\chi)\eta_c)\Big(1-\frac{\delta_1}{\langle \eta-(1-\chi)\eta_c\rangle}\Big)\Big\rangle\\
&\qquad \ \ \ \ \ \ \ \ \ \ \ \ \ \ +2i\delta_1\frac{(\eta-(1-\chi)\eta_c)^2}{\langle \eta-(1-\chi)\eta_c\rangle}\Big(1-\frac{\delta_1}{2\langle \eta-(1-\chi)\eta_c\rangle}\Big).
\end{align*}
and, since on $|\eta_c|\geq {b}/{2}\gg\epsilon_0 $, 
$$
C\epsilon_0 |\Re (\eta-(1-\chi)\eta_c)|\geq  |\Im (\eta-(1-\chi)\eta_c)|,
$$
we have that the imaginary part of the phase satisfies
\begin{align*}
\Im \Phi &\geq \Im \Psi(\alpha,\beta)+c(|\langle \alpha_\xi\rangle| |y|^2+|\langle \beta_\xi\rangle
| |y'|^2)+c\delta_1|\eta-(1-\chi)\eta_c|
\\
& \ \ \ \ \ -|y-y'||\Im ( (1-\chi)\eta_c)|\\
&\geq \Im \Psi(\alpha,\beta) +c(|\langle \alpha_\xi\rangle ||y|^2+|\langle \beta_\xi\rangle||y'|^2)+c\delta_1|\eta-(1-\chi)\eta_c|-C\epsilon_1 \frac{|(1-\chi)\eta_c|^2}{|\langle \alpha_\xi\rangle+\langle \beta_\xi\rangle|}\\&\geq \Im \Psi(\alpha,\beta) +c(|\langle \alpha_\xi\rangle ||y|^2+|\langle \beta_\xi\rangle||y'|^2)+c\delta_1|\eta-(1-\chi)\eta_c|-C\epsilon_1 |(1-\chi)^2\eta_c|\\\
&\geq \Im \Psi(\alpha,\beta) +c(|\langle \alpha_\xi\rangle ||y|^2+|\langle \beta_\xi\rangle||y'|^2)+c\delta_1|\eta-(1-\chi)\eta_c|
\end{align*}
where we have used that $\alpha_\xi$ and $\beta_\xi$ are comparable and taken $\epsilon_0 \ll \delta_1$ small enough. Thus, we may apply the method of complex stationary phase to obtain the result.
\end{proof}

The next result gives the description on the rightmost term in 
\eqref{eq:TLPSL}. For a simpler case capturing the idea of the proof see
\cite[Theorem 2]{GZ}. 

\begin{prop}
\label{l:FBISide}
Suppose $ P $ is defined by \eqref{eq:defP} with $ p $ satisfying
\eqref{e:SymbolAnalyticity}. Then, for $G$ satisfying~\eqref{e:controlDeformation} with $\epsilon_0>0$ small enough,
$$
\Pi_\Lambda T_\Lambda P {S}_\Lambda \Pi_\Lambda = \Pi_\Lambda b_P \Pi_\Lambda +\mathcal O(h^\infty)_{\langle \xi\rangle^NL^2(\Lambda)\to \langle \xi\rangle^{-N}L^2(\Lambda)}
$$
where 
$$
b_P \sim \sum_{j=0}^\infty h^jb_j,\qquad b_j\in S^{m-j},\qquad b_0=p|_{\Lambda}.
$$
\end{prop}
\begin{proof}
Lemma \ref{l:psdef} shows that we need to prove
\begin{equation}
\label{eq:KP2b}  \Pi_\Lambda K_P \Pi_\Lambda = \Pi_\Lambda b \Pi_\Lambda + 
\mathcal O(h^\infty)_{\langle \xi\rangle^NL^2(\Lambda)\to \langle \xi\rangle^{-N}L^2(\Lambda)} , \end{equation}
where $ K_P $ is given by \eqref{eq:strKP}. Propositions \ref{p:Bla1}
and \ref{p:B2P} show that, modulo negligible terms 
the Schwartz kernel of the 
left hand side is given by 
\[  \int_\Lambda \int_\Lambda  
e^{ \frac{ i } h ( \psi ( \alpha , \gamma ) + \Psi ( \gamma, \gamma' ) 
+ \psi ( \gamma', \beta ) + 2 i H ( \gamma ) + 2 i H ( \beta ) ) } 
a ( \alpha, \gamma ) a_P ( \gamma, \gamma' ) a( \gamma' , \beta ) 
d \gamma d \gamma' , \]
where the support property of $ a $  (see \eqref{eq:propsi}) 
shows that integration is over a compact set.
An application of complex stationary phase produces a phase (with 
critical values taken for almost analytic continuation -- see 
\cite[Theorem 2.3, p.148]{mess}) 
\[ \psi_1 ( \alpha, \beta ) = 
{\rm{c.v.}}_{\gamma, \gamma' } 
( \psi ( \alpha , \gamma ) + \Psi ( \gamma, \gamma' ) 
+ \psi ( \gamma', \beta ) + 2 i H ( \gamma ) ) . \]
If we show that $ \psi_1 ( \alpha, \alpha ) =  -2 i H ( \alpha ) $ then
the uniqueness part of Proposition \ref{p:BLa} shows that (modulo negligible terms) we can take $ \psi_1 = \psi $. To see this 
we claim that for $ \alpha = \beta $ the critical point is given by 
$ \gamma = \gamma' = \alpha $, that is 
%we use the formula \eqref{e:tempProjectorPhase} for $ \Psi $ and
%the properties of $ \psi $ in \eqref{eq:proppsi} to see that 
\begin{equation} \label{eq:dgagap} \begin{split} 
& d_\gamma ( \psi ( \alpha , \gamma ) + \Psi ( \gamma, \gamma' ) 
+ \psi ( \gamma', \alpha ) + 2 i H ( \gamma ) )|_{ \gamma = \gamma' = \alpha = 0} = 0  , \\
& d_{\gamma'} ( \psi ( \alpha , \gamma ) + \Psi ( \gamma, \gamma' ) 
+ \psi ( \gamma', \alpha ) + 2 i H ( \gamma ) )|_{ \gamma = \gamma' = \alpha = 0 } = 0 . \end{split} 
\end{equation}
To see this, we first use the formula \eqref{e:tempProjectorPhase} 
for $ \Psi$ to obtain 
\begin{equation}
\label{e:PsiDiff}
d_\gamma \Psi(\gamma,\gamma')|_{\gamma=\gamma'}=\zeta dz|_{\Lambda}=-d_{\gamma'}\Psi(\gamma,\gamma')|_{\gamma=\gamma'}.
\end{equation}
This immediately gives the second equation in \eqref{eq:dgagap}.

We then consider 
\[
\begin{split}
& d_\gamma (\psi ( \alpha , \gamma ) + \Psi ( \gamma, \gamma' ) 
+ \psi ( \gamma', \beta ) + 2 i H ( \gamma ))|_{  \gamma' = \alpha }  = 
\\ & \ \ \ \ \ \ \ \ \ \ \ \ \ \    d_\gamma (\psi ( \alpha , \gamma ) + 2 i H ( \gamma ) +\psi(\gamma,\gamma')-\psi(\gamma,\gamma')+\Psi ( \gamma, \gamma' ) )|_{  \gamma' = \alpha }.
\end{split}\]
The last line in \eqref{eq:proppsi} and \eqref{e:PsiDiff} give 
$$
d_\gamma(-\psi(\gamma,\gamma')+\Psi(\gamma,\gamma')|_{\gamma=\gamma'}=0.
$$
Therefore to obtain the first equation in \eqref{eq:dgagap}, it is enough to have
$$d_{\gamma}(\psi ( \alpha , \gamma ) + 2 i H ( \gamma ) +\psi(\gamma,\gamma'))|_{\gamma=\gamma'=\alpha}=0,$$
which follows from the first line of~\eqref{eq:proppsi} together with $\psi(\alpha,\alpha)=-2iH(\alpha)$. 
Since $ \psi ( \alpha, \alpha ) = -2 i H ( \alpha ) $ the critical value 
is given by  $\psi_1 ( \alpha, \alpha ) = \psi ( \alpha, \alpha )  $. 

It follows that
\[ \begin{split}  \Pi_\Lambda K_P \Pi_\Lambda u ( \alpha ) & = 
h^{-n} \int_\Lambda e^{ \frac i h \psi ( \alpha, \beta ) } 
c ( \alpha, \beta , h ) e^{ - 2 H ( \beta ) } u ( \beta ) d \beta 
\\
& \ \ \ \ \ \ \ \ \ \ \ \ + \mathcal O ( h^\infty\| u \|_{ \langle \xi \rangle^N L^2 ( \Lambda ) }  )_{ \langle \xi \rangle^{-N} L^2 ( \Lambda ) }
, 
\end{split} \]
where $ c $ satisfies \eqref{eq:propaa} (and the support property in 
\eqref{eq:propsi}). Arguing as in \eqref{eq:ZLaB}--\eqref{eq:JZJ} 
we see that the terms in the expansion of $ c $ satisfy transport equations
of \eqref{eq:ZjLa0} and hence are determined by their values on the diagonal. 

Assume that we have obtained $ b_j $, $ j = 0, \cdots , J-1$ (the
case of $ J= 0 $, that is no $b_j $'s is also allowed as the first step) so
that 
\begin{equation}
\label{eq:PiTPTPi} \Pi_\Lambda K_P \Pi_\Lambda = \Pi_\Lambda 
\left(\sum_{j=0}^{J-1} \langle \alpha_\xi \rangle^{-j}  h^j b_j \right) \Pi_\Lambda + R^J_\Lambda, \end{equation}
where 
\[ R^J_\Lambda u ( \alpha ) = h^{J -n } \langle \alpha_\xi \rangle^{-J}  \int_{ \Lambda }
e^{ \frac i h \psi ( \alpha, \beta )}  a^J ( \alpha, \beta ) e^{ -  {2}H  ( \beta ) /h } {u(\beta)} d \beta , \ \ \  a^J \sim a_0^J + h \langle 
\alpha_\xi \rangle^{-1 } a_1^J + \cdots,  \]
with $ a^J_k $ satisfying the transport equations of \S \ref{s:trans}. If we apply the method of stationary phase to the kernel of the first term on 
right hand side of \eqref{eq:PiTPTPi} we obtain, by the inductive hypothesis, a kernel with the expansion
\[   e^{ \frac i h \psi ( \alpha, \beta ) } ( a_0 +% h \langle \alpha_\xi
%\rangle^{-1}  a_1 + 
\cdots 
+ h^{J-1} \langle \alpha_\xi
\rangle^{-J+1} a_{J-1}  + h^J \langle  \alpha_\xi
\rangle^{-J} r_0^J + h^{J+1} \langle  \alpha_\xi
\rangle^{-J-1}  r_1^J + \cdots ) , \]
where $ a_j$'s are the same as in \eqref{eq:defaP}. Again all the terms satisfy transport equations and hence are uniquely determined from their values on the diagonal. Hence, if we put 
\[  b_J ( \alpha ) := r_0^J ( \alpha, \alpha ) + a_0^J ( \alpha, 
\alpha ) , \]
we obtain \eqref{eq:PiTPTPi} with $ J $ replaced by $ J+1 $. When 
$ J = 0$, $ b_J ( \alpha ) = a_0 ( \alpha, \alpha ) = p|_\Lambda ( \alpha ) $. 
\end{proof}

\begin{prop}
\label{l:FBISideComp}
Suppose $p_1\in S^{m_1}$ and $p_2\in S^{m_2}$. Then, for $G$ satisfying~\eqref{e:controlDeformation} with $\epsilon_0>0$ small enough,
$$
\Pi_\Lambda p_1\Pi_\Lambda p_2\Pi_\Lambda = \Pi_\Lambda b \Pi_\Lambda +\mathcal O(h^\infty)_{\langle \xi\rangle^NL^2(\Lambda)\to \langle \xi\rangle^{-N}L^2(\Lambda)}
$$
where 
$$
b \sim \sum_{j=0}^\infty h^jb_j,\qquad c_j\in S^{m_1+m_2-j},\qquad b_0=p_1p_2.
$$
\end{prop}
\begin{proof}
Propositions \ref{p:Bla1} and \ref{p:B2P} show that, modulo negligible terms 
the Schwartz kernel of $\Pi_\Lambda b_1\Pi_\Lambda  b_2\Pi_\Lambda$ is given by
\begin{align*}
&h^{-3n}\int_\Lambda\int_\Lambda e^{\frac{i}{h}(\psi(\alpha,\gamma)+2iH(\gamma)+\psi(\gamma,\gamma')+2H(\gamma')+\psi(\gamma',\beta)+2iH(\beta))}\\
&\qquad\qquad\qquad\qquad \qquad\qquad\qquad \qquad \qquad p_1(\gamma)a(\alpha,\gamma)a(\gamma,\gamma')p_2(\gamma')a(\gamma',\beta)d\gamma d\gamma'
\end{align*}
where the support property of $ a $  (see \eqref{eq:propsi}) 
shows that integration is over a compact set. As in the proof of Proposition~\ref{l:FBISide}, we apply complex stationary phase to the integral resulting in the phase
\[ \psi_1 ( \alpha, \beta ) = 
{\rm{c.v.}}_{\gamma, \gamma'} 
( \psi ( \alpha , \gamma )+\psi(\gamma,\gamma')
+ \psi ( \gamma', \beta ) + 2 i H ( \gamma )+2iH(\gamma') ). \]
Then, it follows from~\eqref{eq:proppsi} that, modulo negligible terms, we may take $\psi_1(\alpha,\beta)=\psi(\alpha,\beta)$ and that, when $\alpha=\beta$, the critical point given by $\gamma=\gamma'=\alpha$ and hence that
\[ \begin{split}  \Pi_\Lambda K_P \Pi_\Lambda u ( \alpha ) & = 
h^{-n} \int_\Lambda e^{ \frac i h \psi ( \alpha, \beta ) } 
c ( \alpha, \beta , h ) e^{ - 2 H ( \beta ) } u ( \beta ) d \beta 
\\
& \ \ \ \ \ \ \ \ \ \ \ \ + \mathcal O ( h^\infty\| u \|_{ \langle \xi \rangle^N L^2 ( \Lambda ) }  )_{ \langle \xi \rangle^{-N} L^2 ( \Lambda ) }
, 
\end{split} \]
where $ c $ satisfies \eqref{eq:propaa} (and the support property in 
\eqref{eq:propsi}). Arguing as in \eqref{eq:ZLaB}--\eqref{eq:JZJ} 
we see that the terms in the expansion of $ c $ satisfy transport equations
of \eqref{eq:ZjLa0} and hence are determined by their values on the diagonal. 
Arguing as in the last paragraph of the proof of Proposition~\ref{l:FBISide} then completes the proof.

\end{proof}

\subsection{Compactness properties of the spaces $H^t(\Lambda)$}
We next study the compactness and trace class properties for operators the spaces $H^m(\Lambda)$.

We start with
\begin{lemm}
\label{l:compact}
There is $h_0>0$ such that for all $s\in \RR$ and $0<h<h_0$
\begin{equation}
\label{eq:compact1}
(hD_\alpha)^\gamma \Pi_\Lambda = \mathcal O ( 1 ) : {\langle \xi\rangle^{s-|\gamma|}L^2(\Lambda)\to \langle \xi\rangle^s L^2(\Lambda)} ,
\end{equation}
and
\begin{equation}
\label{eq:compact} 
t>s \ \Longrightarrow  \ \text{ $ H^t ( \Lambda )  \hookrightarrow H^s ( \Lambda ) $ is compact.} 
\end{equation}
\end{lemm}
\begin{proof}
To prove \eqref{eq:compact1} we show the equivalent fact that the operator
$$
\langle \xi\rangle^{-s}(hD_{\alpha})^{\gamma} \Pi_\Lambda\langle \xi\rangle^{s-|\gamma|}:L^2(\Lambda)\to L^2(\Lambda)
$$
is uniformly bounded. By Proposition~\ref{p:prelim-Projector}, the kernel of this operator is given, modulo acceptable errors, by 
\begin{gather*} 
h^{-n}e^{\frac{i}{h}\Psi(\alpha,\beta)}\left(((\partial_\alpha \Psi)^{\gamma} k(\alpha,\beta)+\mathcal O\langle \Re\alpha_\xi\rangle^{|\gamma|-1})\right) \langle\Re \beta_\xi\rangle^{s-|\gamma|}\langle\Re \alpha_\xi\rangle^{-s}
\widetilde \chi , \\
\widetilde \chi := \chi(|\Re\alpha_x-\Re\beta_x|)\chi(\min(\langle\Re\beta_\xi\rangle,\langle\Re \alpha_\xi\rangle)^{-1}|\Re\alpha_\xi-\Re\beta_\xi|), \ \ \chi \in \CIc ( \RR ) , 
\end{gather*}
where, $\Psi$ is defined in~\eqref{e:tempProjectorPhase}.
Now, on the support of the integrand, $ c\langle \Re\alpha_\xi\rangle \leq \langle \Re\beta_\xi\rangle \leq C\langle\Re \beta_\xi\rangle$ and therefore, $|\partial_\alpha \Psi|\leq C\langle\Re \beta_\xi\rangle.$ In particular, after conjugation by $e^{H/h}$, the kernel is bounded by
$$
Ch^{-n}e^{  c(\langle\Re \alpha_\xi\rangle |\Re\alpha_x-\Re\beta_x|^2+\langle\Re \alpha_\xi\rangle^{-1}|\Re\alpha_\xi-\Re\beta_\xi|^2) /h }\widetilde \chi 
$$
and hence, by Schur's test, for boundedness on $L^2$ is uniformly bounded on $L^2$. 

To see \eqref{eq:compact} we 
prove a slightly stronger statement, namely that $ T_\Lambda ( H^t(\Lambda)) \hookrightarrow  \langle \xi\rangle^{-s} L^2(\Lambda)$ is compact. For that we observe that for $u\in T_\Lambda ( H^t(\Lambda)) $, $u=\Pi_\Lambda u$
and \eqref{eq:compact1}  shows that for $ m \in \ZZ $ and $ k \in \NN $, 
\begin{gather*}  \Pi_\Lambda : \langle \xi \rangle^{-m} L^2 ( \Lambda ) \to 
H^{k,m-k} ( \Lambda ) , \ \ \ 
H^{ r, s } ( \Lambda ) := \langle hD_{x,\xi} \rangle^{-r} \langle \xi 
\rangle^{-s} L^2 ( \Lambda ) .
\end{gather*}
 Hence, by interpolation,
 $$
\Pi_\Lambda :\langle \xi\rangle^{-t}L^2(\Lambda)\to H^{r,t-r}(\Lambda), \ \
r \geq 0 , \ \ t \in \RR. 
$$ 
Setting $ r = (t - s)/2 > 0  $ we obtain continuity of $ T_\Lambda ( H^t ( \Lambda ) ) \hookrightarrow  H^{ r ,\frac{s+t}{2}}(\Lambda)$. The lemma then follows from Rellich's theorem: $H^{r ,s + r }(\Lambda)\hookrightarrow \langle \xi\rangle^sL^2(\Lambda) $, 
$ r > 0 $,  is compact.
\end{proof}

The next lemma provides trace class properties needed in the study of
determinants:
\begin{lemm}
\label{l:traceDeformed}
For $t>3n+s$ the inclusion $ H^t(\Lambda) \hookrightarrow  H^s (\Lambda) $ is of trace class. 
\end{lemm}
\begin{proof}
First, note that for all $r\in \mathbb{R}$,
$
m_r(\alpha,\alpha^*):=\langle \alpha_\xi \rangle^{\frac r2} \langle \alpha^*\rangle^r
$
is an order function in the sense of~\cite[Section 4.4.1]{zw} and for  $r<-2n$
$$
\int_{T^* \Lambda }  m_r(\alpha,\alpha^*)d\alpha d\alpha^*<\infty.
$$ 
Therefore (see \cite[(C.3.6)]{zw} or \cite[Chapter 8]{D-S}) if $\langle \xi\rangle^{-s}A\langle \xi\rangle^{s}\in \Psi(m_{r})$ for $r<-2n$, then $A:\langle \xi\rangle^{-s}L^2(\Lambda)\to \langle \xi\rangle^{-s}L^2(\Lambda)$ is of trace class. 

On the other hand,  Lemma~\ref{l:compact} shows that 
$$
A := \langle \alpha_\xi \rangle^{\frac{-r}{2}} \langle hD_{ \alpha } \rangle^{-\frac{r}{2}} = \mathcal O ( 1 ) : T ( {H}^t(\Lambda) ) \to \langle \xi\rangle^{-s}L^2(\Lambda), \ \ r = \frac{ 2 ( s - t ) }3. 
$$
Also, $ A \in \Psi ( m_{-r} ) $ 
 is elliptic and invertible and hence $A^{-1}\in \Psi(m_{r})$.
 Therefore $ A^{-1} $ is of trace if $ r=\frac{2(s-t)}{3} < -2 n $, that
 is when $ t > 3 n + s $. We conclude that
   $$
\|\Id\|_{\mc{L}^1(T(H^t(\Lambda)), \langle \xi\rangle^{-s}L^2(\Lambda))} \leq \|A^{-1}\|_{\mc{L}^1(\langle \xi\rangle^{-s}L^2(\Lambda), \langle \xi\rangle^{-s}L^2(\Lambda))}\|A\|_{T(H^t(\Lambda)) \to \langle\xi\rangle^{-s}L^2(\Lambda)}<\infty,
$$
where $\mc{L}^1$ denotes the trace class.
\end{proof}

\section{0th order operators and viscosity limits}
\label{s:vis}

Recall that the constructions in the previous sections depend only on finitely many $S^1$ norms of $G$ determining 
$$
\Lambda = \Lambda_G=\{(x+iG_\xi,\xi-iG_x)\mid (x,\xi)\in T^*\mathbb{T}^n\}.
$$ 
(Unless we worked with different $ G$'s, we suppress the dependence on 
$ G $ in $ \Lambda_G $.)
Therefore, we start by fixing $h>0$, $\e_0>0$ small enough and $N_0>0$ large enough such that if
\begin{equation}
\label{e:G}
\sup_{|\alpha|+|\beta|\leq N_0} \langle \xi\rangle^{1-|\beta|}|\partial_x^\alpha \partial_\xi^\beta G|\leq \epsilon_0,
\end{equation}
the constructions of $T_{\Lambda}$, ${S}_{\Lambda}$, are valid and 
$$
B_{\Lambda}-\Pi_{\Lambda}:\langle \xi\rangle^{N}L^2(\Lambda)\to \langle \xi\rangle^{-N}L^2(\Lambda).
$$

\subsection{Elliptic regularity in deformed spaces}
We begin with the following preliminary elliptic regularity lemma.
\begin{lemm}
\label{l:Fred2}
Suppose that  $G\in S^1(T^*\mathbb{T}^n)$ satisfies ~\eqref{e:G} and
\begin{equation}
\label{e:G2}
\sup_{|\alpha|+|\beta|\leq 1}|\langle \xi\rangle^{1-|\beta|} \partial_{x}^\alpha\partial_{\xi}^\beta G|\leq \e_1, 
\end{equation}
for a fixed $ \epsilon_1 $.
Suppose also that  $ E $ is given by  \eqref{eq:defP} with $ e $ (replacing $ p $) satisfying~\eqref{e:SymbolAnalyticity} and
$$
|e(z,\zeta)|\geq c_1 |\zeta|^m, \ \  |\zeta| \geq C, \ \ |\Im {z}|\leq \e_1,\ \  |\Im \zeta|\leq \e_1\langle \zeta\rangle.
$$
Then  $ E:H^{s}_{{\Lambda}}\to H^{s-m}_{{\Lambda}}$ is a Fredholm operator and there exists %$C_1=C_1(\e_1,P)>0$ and 
$C_1=C_1(s,\e_1,E,N)>0$ such that 
$$
\tfrac 12{c_1} \|u\|_{H^{s}_{\Lambda}}\leq  \|E u\|_{H^{s-m}_{{\Lambda}}}+C_1\|u\|_{H^{-N}_{{\Lambda}}}.
$$
\end{lemm}

\begin{proof}
The assumptions on $ e $ guarantee that 
$$
|e|_\Lambda (\alpha)| \geq c_1  |\alpha_\xi|^m, \ \ |\alpha_\xi|\geq C, \ \ \alpha\in \Lambda .
$$
Proposition 
\ref{l:FBISide} then shows 
$$
\Pi_{{\Lambda}}T_{{\Lambda}} E {S}_{{\Lambda}} \Pi_{{\Lambda}}= \Pi_{{\Lambda}} b_E\Pi_{{\Lambda}} + R'
$$
with 
$$
 \tilde{b}_E \sim \sum_j \tilde{b}_j,\qquad \tilde{b}_j\in S^{m-j},\qquad \tilde{b}_0=e|_{\Lambda}
 $$
 and
 $$
 \|R'\|_{\langle \xi\rangle^NL^2({\Lambda})\to \langle \xi\rangle^{-s+m}L^2({\Lambda})}\leq {C}'_2= {C}'_2(s,\e_1,E,N).
 $$
 Next, by Proposition~\ref{l:FBISideComp}, 
$$
 \Pi_{{\Lambda}} \overline{\tilde{b}_E}\Pi_{{\Lambda}} \tilde{b}_E\Pi_\Lambda=\Pi_{{\Lambda}} b_E\Pi_{{\Lambda}} + R''
$$
with 
$$
 b_E \sim \sum_j b_j,\qquad b_j\in S^{m-j},\qquad b_0=|e|_{\Lambda}|^2
 $$
 and
 $$
 \|R''\|_{\langle \xi\rangle^NL^2({\Lambda})\to \langle \xi\rangle^{-s+m}L^2({\Lambda})}\leq {C}''_2= {C}''_2(s,\e_1,E,N).
 $$
Since $|b_0|\geq c_1^2|\alpha_\xi |^{2m}$ on $|\alpha_\xi |\geq  C$, there is $b_\infty\in S^{-\infty}$ such that with $b:=b_E+b_\infty $
$$
|b|\geq \frac{2}{3}c_1^2\langle |\alpha_\xi|\rangle^{2m},
$$
and
\begin{gather*}
(\Pi_{{\Lambda}}T_{{\Lambda}} E {S}_{{\Lambda}} \Pi_{{\Lambda}})^*(\Pi_{{\Lambda}}T_{{\Lambda}} E {S}_{{\Lambda}} \Pi_{{\Lambda}})= \Pi_{{\Lambda}} b\Pi_{{\Lambda}} + R,\\
\|R\|_{\langle \xi\rangle^NL^2({\Lambda})\to \langle \xi\rangle^{-s+m}L^2({\Lambda})}\leq {C}_2= {C}_2(s,\e_1,E,N).
\end{gather*}
For $ u \in H^s ( \Lambda )$ we compute
\[ \begin{split}
\|Eu\|^2_{H^{s-m}_{\Lambda  }} & =\langle \Pi_{\Lambda  } T_{\Lambda  } E {S}_{{\Lambda  }} \Pi_{\Lambda  } T_{\Lambda  } u, \Pi_{\Lambda  } T_{\Lambda  } E {S}_{{\Lambda  }} \Pi_{\Lambda  } T_{\Lambda  }  u\rangle_{\langle \xi\rangle^{-s+m}L^2}\\
&=\langle \Pi_\Lambda (b+R) T_{\Lambda  } u, T_{\Lambda  }  u\rangle_{\langle \xi\rangle^{-s+m}L^2}\\
&\geq \langle  b T_{\Lambda  } u, T_{\Lambda  } u\rangle_{\langle \xi\rangle^{-s+m}L^2}- {C}_2\|u\|_{H^{-N}_{\Lambda  }}^2\\
&\geq \tfrac14 c_1^2 \|u\|^2_{H^{s}_{{\Lambda  }}}- C_1^2\|u\|^2_{H^{-N}_{{\Lambda  }}}
\end{split} \]
with $C_1=C_1(s,\e_1,E,N).$
Therefore,
\begin{equation}
\label{e:ellipticity}
\tfrac 12 c_1 \|u\|_{H^{s}_{{\Lambda  }}}\leq  \|E u\|_{H^{s-m}_{{\Lambda  }}}+C_1\|u\|_{H^{-N}_{{\Lambda  }}}.
\end{equation}
We now note that for all $ s $, 
\begin{equation}
\label{e:adjoint}
\begin{aligned}
\langle E^*u,v\rangle_{H^s_{{\Lambda  }}}&=\langle u,Ev\rangle_{H^s_{{\Lambda  }}}=\langle T_{{\Lambda  }}u,\Pi_{{\Lambda  }} T_{{\Lambda  }} E {S}_{{\Lambda  }} \Pi_{{\Lambda  }}T_{{\Lambda  }}v\rangle_{\langle \xi\rangle^{-s}L^2({\Lambda  })}\\
&=\langle T_{{\Lambda  }}u,(b  +R)T_{{\Lambda  }}v\rangle_{\langle \xi\rangle^{-s}L^2({\Lambda  })}\\
&=\langle (\bar b +R^*)T_{{\Lambda  }}u,T_{{\Lambda  }}v\rangle_{\langle \xi\rangle^{-s}L^2({\Lambda  })}.
\end{aligned}
\end{equation}
Using~\eqref{e:adjoint}, we obtain
\begin{align*}
\|E^*u\|^2_{H^{s-m}_{\Lambda  }}&= \langle E^* u,E^*u\rangle_{H_{\Lambda  }^{s-m}}=\langle (\bar{b}+R^*)T_{\Lambda  } u,T_{\Lambda  }  u\rangle_{\langle \xi\rangle^{-s+m}L^2}\\
&\geq \frac14 c_1^2\|u\|^2_{H^{s}_{{\Lambda  }}}- C_1^2\|u\|^2_{H^{-N}_{{\Lambda  }}}.
\end{align*}
Therefore
\begin{equation}
\label{e:ellipticity2}
\tfrac 12 c_1  \|u\|_{H^{s}_{{\Lambda  }}}\leq \|E ^*u\|_{H^{s-m}_{{\Lambda  }}}+C_{1}\|u\|_{H^{-N}_{{\Lambda  }}}.
\end{equation}

Combining~\eqref{e:ellipticity} and~\eqref{e:ellipticity2} with $s$ replaced by $m-s$, and applying Lemma~\ref{l:compact}, we have for $N>m-s$ that $H^{s-m}_{{\Lambda  }}\to H^{-N}_{{\Lambda  }}$ is compact. Thus, we have proved that $E:H^{s}_{{\Lambda  }}\to H^{s-m}_{{\Lambda  }}$ is  a Fredholm operator.
\end{proof}

\subsection{Zeroth order operators on deformed spaces} 
We now work in the setting of Theorem~\ref{t:meromorphy}.  Let $P\in \Psi^0$ satisfy the assumptions there, $G_0\in S^1(T^*\mathbb{T}^n)$ and $C>0$ such that 
\begin{equation}
\label{e:G0}
H_pG_0>0,\qquad \{|\xi|>C\}\cap \{p=0\}.
\end{equation}
Define the $\RR$-symplectic $I$-Lagrangian submanifold $\Lambda_\theta\subset \widetilde{T^*\mathbb{T}^n}$ by 
$$
\Lambda_\theta=\{ (x+i\theta \partial_\xi G_0, \xi-i\theta\partial_xG_0)\mid (x,\xi)\in T^*\mathbb{T}^n\}.
$$
We work with the spaces $H_{\Lambda_\theta}^m$ as defined in~\eqref{e:deformedSpace}. Observe that for $|\theta|$ small enough, $\theta G_0$ satisfies \eqref{e:G}. To avoid cumbersome notation, we will suppress the dependence of $\Lambda_\theta$ on $\theta$. 

For $u\in H^m_\Lambda$ we have 
$
T_\Lambda P u= \Pi_\Lambda T_\Lambda P {S}_\Lambda T_\Lambda u.
$
By Proposition~\ref{l:FBISide}
\begin{equation}
\label{e:mule}
\begin{gathered}
\Pi_\Lambda T_\Lambda P {S}_\Lambda \Pi_\Lambda= \Pi_\Lambda b_P\Pi_\Lambda +R_1,\qquad \Pi_\Lambda T_\Lambda \Delta  {S}_\Lambda \Pi_\Lambda= \Pi_\Lambda a_\Delta\Pi_\Lambda +R_2,
\end{gathered}
\end{equation}
where $R_i:\langle \xi\rangle^NL^2(\Lambda)\to \langle \xi\rangle^{-N}L^2(\Lambda)$, 
$$
b_P-p|_{\Lambda}\in S^{-1},\qquad a_\Delta+(\zeta)^2|_{\Lambda}\in S^{-1}.
$$
Now, 
$$
p|_{\Lambda}=p(x,\xi)-i\theta H_pG_0 +\mathcal O(\theta^2)_{S^0}.
$$
In particular, by~\eqref{e:G0}, there are $c,C>0$ such that for $\theta>0$ small enough, 
\begin{equation*}
\begin{gathered}
\Im p|_{\Lambda}\leq (-c +C \langle \xi\rangle^{-1})\theta,\qquad\text{ on }\big| \Re p|_{\Lambda}\big|\leq c.
\end{gathered}
\end{equation*}
In particular, there exists $b_\infty \in S^{-\infty}$ such that for $e_0:=b_P+b_\infty$, there are $c_0, C_0>0$ satisfying
\begin{equation}
\label{e:positivity1}
\begin{gathered}
|e_0|>c_0\theta>0, 
\qquad \Im e_0 \leq -c_0\theta \ \text{ on } \ |\Re e_0| <c_0,\qquad |\Im e_0|\leq C_0(\theta+\langle \xi\rangle^{-1}) 
\end{gathered}
\end{equation} 
By~\eqref{e:mule}, we also have
\begin{equation}
\label{e:mule2}
\Pi_\Lambda T_\Lambda P  {S}_\Lambda \Pi_\Lambda= \Pi_\Lambda e_0\Pi_\Lambda +R_0
\end{equation}
with $R_0:\langle \xi\rangle^NL^2(\Lambda)\to \langle \xi\rangle^{-N}L^2(\Lambda)$ uniformly over $0\leq \theta\leq \theta_0 $.

To analyse the contribution of the Laplacian we note that
$$
\Re (\zeta)^2|_{\Lambda} \geq (1-C\theta^2)|\xi|^2, \qquad |\Im (\zeta)^2|_{\Lambda}|\leq C\theta|\xi|^2.
$$
Therefore, for $\theta>0$ small enough, we can find $a_{\infty}\in S^{-\infty}$ such that 
\begin{equation}
\label{e:mule3}
\Re (a_\Delta+a_{\infty})\geq \frac{1}{2}\langle \xi\rangle^2,\qquad |\Im (a_\Delta +a_\infty)|\leq C\theta|\xi|^2+C|\xi|,
\end{equation}
 and we have
\begin{equation}
\label{e:compute}
\Pi_\Lambda T_\Lambda ( P+i\nu \Delta)  {S}_\Lambda \Pi_\Lambda= \Pi_\Lambda e_\nu \Pi_\Lambda +R_\nu,\qquad  e_\nu:=e_0-i\nu(a_\Delta+a_\infty)
\end{equation}
where, by~\eqref{e:mule} and~\eqref{e:mule2}, $R_\nu:\langle \xi\rangle^NL^2(\Lambda)\to \langle \xi\rangle^{-N}L^2(\Lambda)$ uniformly in $0\leq \nu \leq 1$, $0\leq\theta\leq \theta_0 $.

The next lemma gives us crucial properties of $ e_\nu $:
\begin{lemm}
\label{l:positivity}
There exist $c_1,\theta_0,\nu_0>0$ such that for all $0\leq \nu\leq \nu_0$ and $0<\theta\leq\theta_0$ 
\begin{equation}
\label{e:positivity}\begin{gathered}
|e_\nu|>c_1\theta(1+\nu|\xi|^2) >0, \\
\qquad \Im e_\nu \leq -c_1\theta(1+\nu|\xi|^2) \ \text{ on } \ |\Re e_\nu| <c_1(1+\nu|\xi|^2).
\end{gathered}
\end{equation}
\end{lemm}
\begin{proof}
We consider two cases. First, suppose $|\xi|\geq M \nu^{-1/2}$.  Then, by~\eqref{e:mule3} and~\eqref{e:positivity1}, there are $c_2,C_2>0$ such that 
$$
\Im e_\nu\leq -c_2(M^2+\nu |\xi|^2)+C_2(\theta +M^{-1}\nu^{1/2}).
$$
Therefore, setting 
$$
M=\max\Big(1,2\sqrt{{C_2}/{c_2}}\Big),$$
\eqref{e:positivity} holds on $|\xi|\geq M\nu^{-1/2}$ (uniformly in $0\leq \nu\leq 1$ and $0\leq \theta\leq 1$). 

We next consider the case $|\xi|\leq M\nu^{-1/2}$. If  ${c_0}\geq 2 |\Re e _\nu|$, then 
$$
\frac{c_0}{2}\geq |\Re e _\nu|\geq |\Re e_0|-C\theta M^2- C\nu^{1/2}M.
$$
Choosing $\theta_0$ and $\nu_0$ small enough, we obtain $|\Re e_0|\leq c_0$ and hence
$$
\Im e_\nu\leq \Im e_0\leq -c_0\theta,
$$
which completes the proof of~\eqref{e:positivity}.
\end{proof}

\subsection{Fredholm properties and meromorphy of the resolvent}
We add a localized absorbing potential to $ P + i \nu \Delta $ to 
obtain invertibility. {That is, for} $q\in C_c^\infty(\Lambda; [0,\infty))$ we define
\begin{equation}
\label{eq:Pqnu}
P_{q,\nu}:=P+i\nu \Delta  -iQ,\qquad Q:= {S}_\Lambda\Pi_\Lambda q\Pi_\Lambda T_\Lambda.
\end{equation}
This family includes the operator $ P = P_{0,0} $ and the viscous operator
$ P + i \nu \Delta = P_{ 0, \nu} $. 
We 
note that
\begin{equation}
\label{e:onion}
\begin{aligned}
\Pi_\Lambda T_\Lambda QS_\Lambda\Pi_\Lambda&=\Pi_\Lambda T_\Lambda  {S}_\Lambda\Pi_\Lambda q\Pi_\Lambda T_\Lambda S_\Lambda\Pi_\Lambda\\
&=\Pi_\Lambda P_\Lambda \Pi_\Lambda q\Pi_\Lambda P_\Lambda \Pi_\Lambda=\Pi_\Lambda q\Pi_\Lambda,
\end{aligned}
\end{equation}
where we recall that $P_\Lambda=T_\Lambda S_\Lambda$ satisfies $ \Pi_\Lambda=P_\Lambda \Pi_\Lambda$.
We record the following Lemma for use later.
\begin{lemm}
\label{l:adjoint}
The adjoint of $P:H^s(\Lambda)\to H^s(\Lambda)$ satisfies
$$
\langle P^* u,v\rangle_{H^s(\Lambda)}=\langle( \Pi_{\Lambda} \,\bar{e}_0\,\Pi_{\Lambda}+R_0^*)Tu, Tv\rangle_{\langle \xi\rangle^{-s}L^2(\Lambda)},
$$
and the adjoint of $P_{q,\nu}:H^{s+2}(\Lambda)\to H^s(\Lambda)$, $\nu>0$ satisfies
$$
\langle P_{q,\nu}^* u,v\rangle_{H^s(\Lambda)}=\langle( \Pi_{\Lambda} \,(\bar{e}_\nu-iq)\,\Pi_{\Lambda}+R_\nu^*)Tu, Tv\rangle_{\langle \xi\rangle^{-s}L^2(\Lambda)}.
$$
\end{lemm}
\begin{proof}
The lemma follows from~\eqref{e:adjoint}.
\end{proof}

We start by proving that $P-\omega$ is Fredholm on $H^s_{\Lambda}$. In particular, the next lemma proves the first part of Theorem~\ref{t:meromorphy} with $\mc{X}=H^s_{\Lambda}$.

\begin{lemm}
\label{l:Fred1}
There is $\omega_0>0$ such that  for $\Im \omega>-\omega_0 \theta $ and $|\Re\omega|<\omega_0$, 
$$ 
P -\omega:H^s_\Lambda\to H^{s}_{\Lambda}
$$  is a Fredholm operator.  For $\Im \omega\gg 1$, $P-\omega$ is invertible with inverse $\mc{R}(\omega)$ satisfying
$$
\|\mc{R}(\omega)\|_{H^s_{\Lambda}\to H^{s}_\Lambda}\leq \frac{C}{\Im \omega},\qquad \Im \omega\geq C,\quad |\Re \omega|<\omega_0.
$$
In particular, $\mc{R} (\omega):H^{s}_{\Lambda}\to H^{s}_{\Lambda}$ is a meromorphic family of operators for $\omega\in (-\omega_0,\omega_0)+i(-\omega_0\theta,\infty).$
\end{lemm}
\begin{proof}
First, observe that by~\eqref{e:compute} with $\nu=0$, and using Proposition~\ref{l:FBISideComp}
$$
\Pi_\Lambda \langle \xi\rangle^{2s}(\Pi_\Lambda T_\Lambda (P-\omega)S_\Lambda \Pi_\Lambda)^*\langle\xi\rangle^{-2s}(\Pi_\Lambda T_\Lambda (P-\omega)S_\Lambda \Pi_\Lambda)=\Pi_\Lambda |e_0-\omega|^2\Pi_\Lambda +\tilde{R}_0
$$
where
$$
\tilde{R}_1:\langle \xi\rangle^{s}L^2(\Lambda)\to \langle \xi\rangle^{s+1}L^2(\Lambda).
$$
By~\eqref{e:positivity}, there is $c_1>0$ such that $\Im e_0\leq -c_1\theta $ on $|\Re e_0|\leq c_1$. Therefore, on $|\Re e_0|\leq c_1$, 
$$
|e_0-\omega|^2\geq  (c_1\theta +\Im \omega )^2\geq \frac{c_1^2\theta^2}{4}{+\max(\Im \omega,0)^2},
$$
where we have taken $\omega_0=\frac{c_1}{2}$ and $\Im \omega\geq -\omega_0\theta$. Then, using $|\Re \omega|\leq \omega_0$, on $|\Re e_0|\geq c_1$, there is $C>0$ such that
$$
|e_0-\omega|^2\geq \frac{c_1^2}{4}+\min(|\Im \omega|^2-C,0)\geq c_1(1+|\Im \omega|^2).
$$
In particular, for $u\in H^s_\Lambda$, 
\begin{align*}
\|(P-\omega)u\|_{H^s_\Lambda}^2&= |\langle (P-\omega)u,(P-\omega)u\rangle_{H_\Lambda^s}|\\
&=\Big|\langle \Pi_\Lambda T_\Lambda (P-\omega) {S}_{\Lambda} \Pi_\Lambda T_\Lambda u, \Pi_\Lambda T_\Lambda (P-\omega) {S}_{\Lambda} \Pi_\Lambda T_\Lambda u\rangle_{\langle \xi\rangle^{-s}L^2}\Big|\\
&\geq |\langle |e_0- \omega|^2T_\Lambda u, T_\Lambda u\rangle_{\langle \xi\rangle^{-s}L^2}| -|\langle \tilde{R}_0T_\Lambda u, T_\Lambda u\rangle_{\langle \xi\rangle^{-s}L^2}|\\
%&\qquad-|\langle (e_0-\omega)T_\Lambda u, R_0T_\Lambda u\rangle_{\langle \xi\rangle^{-s}L^2}| - |\langle R_0T_\Lambda u, R_0 T_\Lambda u\rangle_{\langle \xi\rangle^{-s}L^2}|\\
&\geq c_1(1+|\Im \omega|^2)- C\|u\|_{H^{s-\frac{1}{2}}_{\Lambda}}^2.
\end{align*}

In particular, iterating this argument, we have for any $N$, there is $C_N>0$ such that
\begin{equation}
\label{e:pumpkin1}
\|u\|_{H^s_\Lambda}^2\leq C(1+|\Im \omega|^2)^{-1}[\|(P-\omega)u\|_{H^s_\Lambda}^2+ C_N\|u\|_{H^{-N}_{\Lambda}}^2].
\end{equation}

By almost exactly the same argument, using Lemma~\ref{l:adjoint}, we obtain
\begin{equation}
\label{e:pumpkin2}
\|u\|_{H^s_{\Lambda}}\leq C(1+|\Im \omega|^2)^{-1}[\|(P^*-\bar{\omega})u\|_{H^s_{\Lambda}} +C_N\|u\|_{H^{-N}_{\Lambda}}].
\end{equation}
Next, by Lemma~\ref{l:compact}, for $N>-s$, the embedding $H^{s}_{\Lambda}\to H^{-N}_{\Lambda}$ is compact.
Therefore, $ P -\omega:H^s_\Lambda\to H^s_\Lambda$ is Fredholm. 

Finally, taking $\Im \omega \gg1$, we may absorb the $H^{-N}_{\Lambda}$ error into the left hand sides of~\eqref{e:pumpkin1} and~\eqref{e:pumpkin2} to obtain that $P-\omega:H^s_\Lambda\to H^s_{\Lambda}$ is invertible with the desired estimate. The meromorphic Fredholm theorem (see e.g~\cite[Theorem C.9]{dizzy}) then shows that $\mc{R} (\omega)$ is a meromorphic family of operators on $(-\omega_0,\omega_0)+i(-\omega_0\theta,\infty).$
\end{proof}

We next study the meromorphy  of the inverse of $P_{q,\nu}-\omega$, where
$ P_{q, \nu} $ is given in \eqref{eq:Pqnu}:
\begin{lemm}
\label{l:fred3}
There exists $\e_0>0$ such that the following holds. For all $s\in \mathbb{R}$, $K\in \mathbb{N}$, $\omega \in \mathbb{C}$, $q\in C_c^\infty(T^*\mathbb{T}^n)$ and $\nu>0$ there are $C_0 = C_{s,\nu,K}$, and $C_1 = C_{s,\nu, K,N}$ such that for all $G$ satisfying~\eqref{e:G}, 
\[ (P_{q,\nu}-\omega)^K:H^{s}_{\Lambda  }\to H^{s-2K}_{\Lambda  }\] is a Fredholm operator and
\begin{equation}
\label{e:ellipticEst}
\|u\|_{H^{s}_{\Lambda  }}\leq C_0\|(P_{q,\nu}-\omega)^Ku\|_{H_{\Lambda  }^{s-2K}}+C_1 \|u\|_{H^{-N}_{\Lambda  }}.
\end{equation}
Moreover, $P_{q,\nu}-\omega :H^{s}_{\Lambda  }\to H^{s-2}_{\Lambda  }$ is invertible for $\Im \omega \gg1$  with inverse $\mc{R}_{q,\nu}(\omega)$ satisfying
$$
\|\mc{R}_{q,\nu}(\omega)\|_{H^{s}_{\Lambda  }\to H^{s+2}_{\Lambda  }}\leq C\nu^{-1},\qquad \Im \omega \geq C(1+\nu) .
$$
In particular, for all $\nu>0$, $\mc{R}_{q,\nu}(\omega):H^{s}_{\Lambda  }\to H^{s+2}_{\Lambda  }$ is a meromorphic family of operators for $\omega\in \mathbb{C}$.
\end{lemm}
\begin{proof}
We first note that
$$
|\sigma(\Delta)(z,\zeta)|=|\zeta^2|\geq |\Re \zeta|^2-|\Im \zeta|^2
$$
and hence $\Delta^K$ satisfies the hypotheses of Lemma~\ref{l:Fred2} for any $\e_1<1$. In particular, by that lemma $(i\nu \Delta)^K:H^{s}_{\Lambda  }\to H^{s-2K}_{\Lambda  }$ is a Fredholm operator and satisfies
\begin{equation}
\label{e:ellipticReg}
\|u\|_{H^{s}_{\Lambda  }}\leq C_1^{K}\nu^{-K}\|(i\nu \Delta )^{K}u\|_{H^{s-2K}_{\Lambda  }}+\nu^{-K} C_{s,N,K}\|u\|_{H^{-N}_{\Lambda  }},
\end{equation}
for any $G$ satisfying~\eqref{e:G}.

Next, observe that $(P_{q,\nu}-\omega)^K-(i\nu \Delta)^K:H^{s}_{\Lambda  }\to H^{s-2K+2}_{\Lambda  }$ and $H^{s-2K+2}_{\Lambda  }\to H^{s-2K}_{\Lambda  }$ is compact by Lemma~\ref{l:compact}. Therefore, $(P_{q,\nu}-\omega)^K:H^{s}_{\Lambda  }\to H^{s-2K}_{\Lambda  }$ is a Fredholm
operator.

Finally, using~\eqref{e:ellipticReg},
\begin{align*}
\|u\|_{H^{s}_{\Lambda  }}&\leq C_{\nu,K}\|(i\nu \Delta)^K u\|_{H^{s-2K}_{\Lambda  }}+C_{s,\nu,N,K}\|u\|_{H^{-N}_{\Lambda  }}\\
&\leq C_{\nu,K}\|(P_{q,\nu}-\omega)^K u\|_{H^{s-2K}_{\Lambda  }}+C_{s,\nu,K}\|u\|_{H^{s-2}_{\Lambda}}+C_{s,\nu,N,K}\|u\|_{H^{-N}_{\Lambda  }}.
\end{align*}
Estimating $\|u\|_{H^{s-2}_{\Lambda  }}$ by $\| (P_{q,\nu}-\omega)^K\|_{H^{s-2K-2}_{\Lambda  }}$ and iterating we obtain~\eqref{e:ellipticEst}.

For invertibility, let $\Im \omega \gg 1$ and consider
\begin{equation}
\label{e:squash}
\begin{aligned}
&\|(P_{q,\nu}-\omega)u\|_{H^{s-2}_{\Lambda  }}\|u\|_{H^{s-2}_{\Lambda  }}\\
&\geq -\Im \langle (P_{q,\nu}-\omega)u,u\rangle_{H^{s-2}_{\Lambda  }}\\
&\geq -\Im \langle \Pi_{\Lambda  }(i\nu {a}-iq-\omega)\Pi_{\Lambda  }T_{\Lambda  }u,T_{\Lambda  }u\rangle_{\langle \xi\rangle^{-s+2}L^2(\Lambda)} -C_N\|u\|_{H_{\Lambda  }^{-N}}^2
\end{aligned}
\end{equation}
where {$a=a_\Delta+a_\infty$ is as in~\eqref{e:mule3}}. In particular, for $\Im \omega >0$,
$$
\Im i\nu {a}-iq-\omega \leq -c\nu |\xi|^2+C\nu |\xi|-\Im \omega,
$$
and for $\Im \omega\geq C_N+1+C\nu$, 
$$
\Im i\nu {a}-iq-\omega \leq -C_N-1-c\nu |\xi|^2.
$$
Using this in~\eqref{e:squash}, we obtain
$$\|(P_{q,\nu}-\omega)u\|_{H^{s-2}_{\Lambda  }}\geq C\nu\|u\|_{H^s_{\Lambda}}.$$
This same argument implies that 
$$
\|(P^*_{q,\nu}-\bar{\omega})u\|_{H^{s-2}_{\Lambda  }}\geq C\nu\|u\|_{H^s_{\Lambda}}.
$$
and hence $P_{q,\nu}-\omega$ is invertible with inverse as claimed. The meromorphic Fredholm theorem (see e.g~\cite[Theorem C.9]{dizzy}) then shows that $\mc{R}_{q,\nu}(\omega)$ is a meromorphic family of operators for $\omega\in \mathbb{C}$.
\end{proof}

\subsection{A parametrix for the resolvent of $P_{0,\nu}-\omega$}

We next find $q$ so that the compact perturbation $P_{q,\nu}$ of $P_{0,\nu}$ is invertible. This inverse will be used to approximate the inverse of $P_{0,\nu}$.
\begin{lemm}
\label{l:parametrix}
There are $\omega_0,\nu_0,\theta_0>0$ so that for all $\epsilon >0$ and $\theta\in(0,\theta_0)$, there is $q=q(\epsilon,\theta)\in C_c^\infty(\Lambda;[0,\infty))$ such that for all $\nu\in (0,\nu_0]$,  and $\omega\in (-\omega_0,\omega_0)+i(-\omega_0\theta,\infty)$, the operators
$$
P_{q,\nu} -\omega: H^{s+2}_\Lambda\to H^s_\Lambda,\qquad \text{ and }\qquad P_{q,0} -\omega: H^{s}_\Lambda\to H^s_\Lambda
$$
are invertible with inverse $\mathcal{R}_{q,\nu}(\omega):=( P_{q,\nu} -\omega)^{-1}$
satisfying
$$
\|\mathcal{R}_{q,\nu}(\omega)\|_{H^{s}_\Lambda\to H^{s-\epsilon}_\Lambda}\leq 1
$$
\end{lemm}
\begin{proof}
We assume without loss of generality that $\e<\frac{1}{2}$. Observe that by~\eqref{e:compute} and ~\eqref{e:onion}
\begin{align*}
\Pi_\Lambda T_\Lambda P_{q,\nu} {S}_{\Lambda} \Pi_\Lambda& =\Pi_\Lambda( e_\nu-iq )\Pi_\Lambda +R_\nu.
\end{align*}
and thus, by Proposition~\ref{l:FBISideComp}, 
$$
\Pi_\Lambda \langle \xi\rangle^{2s}(\Pi_\Lambda T_\Lambda (P_{q,\nu}-\omega) {S}_{\Lambda} \Pi_\Lambda)^*\langle \xi\rangle^{-2s}\Pi_\Lambda T_\Lambda (P_{q,\nu}-\omega) {S}_{\Lambda} \Pi_\Lambda=\Pi_\Lambda |e_\nu -iq-\omega |^2\Pi_\Lambda +\tilde{R}_\nu
$$
where for all $s\in \mathbb{R}$,
$$
\|\tilde{R}_\nu v\|_{H_\Lambda^s}\leq C_0 (\|v\|_{H_{\Lambda}^{s-1}}+\nu \|v\|_{H_{\Lambda}^{s+1}}+\nu^2\|v\|_{H_{\Lambda}^{s+3}}),\qquad 0\leq \nu\leq 1.
$$
Therefore, 
\begin{equation}
\label{e:estimateMe}
 \begin{split}
&\|(P_{q,\nu}-\omega )u\|_{H^{s}_\Lambda}^2 \\
%&=\Big|\langle \Pi_\Lambda T_\Lambda( P_{q,\nu}-\omega ) {S}_{\Lambda} \Pi_\Lambda T_\Lambda u, \Pi_\Lambda T_\Lambda (P_{q,\nu}-\omega) {S}_{\Lambda} \Pi_\Lambda T_\Lambda u\rangle_{\langle \xi\rangle^{-s}L^2}\Big|\\
&\geq |\langle \langle \xi\rangle^{\epsilon}|e_\nu-iq-\omega|^2\langle\xi\rangle ^{\epsilon} \langle\xi\rangle^{-\epsilon} T_\Lambda u, \langle\xi\rangle^{-\epsilon}T_\Lambda u\rangle_{\langle \xi\rangle^{-s}L^2}| \\
&\qquad -|\langle \tilde{R}_\nu\langle\xi\rangle^{\epsilon}\langle\xi\rangle^{-\epsilon}T_\Lambda u,\langle \xi\rangle^{\epsilon} \langle\xi\rangle^{-\epsilon} T_\Lambda u\rangle_{\langle \xi\rangle^{-s}L^2}|\\
&\geq |\langle \langle \xi\rangle^{\epsilon}|e_\nu-iq-\omega|^2\langle\xi\rangle ^{\epsilon} \langle\xi\rangle^{-\epsilon} T_\Lambda u, \langle\xi\rangle^{-\epsilon}T_\Lambda u\rangle_{\langle \xi\rangle^{-s}L^2}| \\
&\qquad -C_0(\|\langle \xi\rangle^{-\epsilon}T_\Lambda u\|^2_{\langle \xi\rangle^{s-\frac{1}{2}+\e} L^2(\Lambda)}+\nu\|\langle \xi\rangle^{-\epsilon}T_\Lambda u\|^2_{\langle \xi\rangle^{s+\frac{1}{2}+\e} L^2(\Lambda)}\\
&\qquad\qquad\qquad\qquad\qquad\qquad+\nu^2\|\langle \xi\rangle^{-\epsilon}T_\Lambda u\|^2_{\langle \xi\rangle^{s+\frac{3}{2}+\e} L^2(\Lambda)})
\end{split} 
\end{equation}

Let $\theta_0,\nu_0,$ and $c_1$ be as Lemma~\ref{l:positivity} and fix $\chi=\chi_{\epsilon}\in C_c^\infty(\Lambda;[0,1])$ with 
$$
\chi\equiv 1\text{ on } \langle \xi\rangle^{2\epsilon}<\frac{8\max(C_0,1)}{c_1^2\min(1,\theta^2)}.
$$
Then, let $q=M\chi_\epsilon$ for $M$ to be chosen later and $\omega_0\leq {c_1}/{2}$. On $\supp (1-\chi_\epsilon)$, 
\begin{align*}
  \langle \xi\rangle^{2\epsilon}|e_\nu-iq- \omega|^2&\geq   {\langle \xi\rangle^{2\epsilon}}(|\Re e_\nu-\Re\omega|^2+|\Im e_\nu -q-\Im \omega|^2)\\
&\geq  {\langle \xi\rangle^{2\epsilon}}(\min((c_1(1+\nu|\xi|^2)-|\Re \omega|)^2,|c_1\theta(1+\nu|\xi|^2)+\Im \omega|^2)\\
&\geq \tfrac 14 {\langle \xi\rangle^{2\epsilon}c_1^2}\min(1,\theta^2)(1+\nu|\xi|^2)^2\\
&\geq (1+ C_0)(1+\nu|\xi|^2)^2
\end{align*}

On $\chi_\epsilon \equiv 1$, we have
\begin{align*}
\langle \xi\rangle^{2\epsilon}|e_\nu-iq- \omega|^2&\geq   {\langle \xi\rangle^{2\epsilon}} (M^2-4(\omega_0\theta+| e_\nu|)^2)\geq \tfrac{1}{4}(M^2-C)
\end{align*}
for some $C>0$ independent of $\nu,\omega_0 ,\theta \in[0,1]$.
Therefore, for  $\omega_0:=\min({c_1}/{2},1)$ and 
$$
M^2:=C+4(1+C_0)\Big(1 +\frac{8\max(C_0,1)}{c_1^2\min(1,\theta^2)}\Big)^2.
$$ 
we have
$$
\inf_{\Lambda}\langle \xi\rangle^{2\epsilon}|e_\nu-iq- \omega|^2\geq ( 1+ C_0)(1+\nu|\xi|^2)^2,\qquad 0\leq \nu\leq 1.
$$
In particular, using this in~\eqref{e:estimateMe} yields
\begin{equation}
\label{e:squid}
\|u\|_{H^{s-\epsilon}_\Lambda}\leq \|(P_{q,\nu}-\omega)u\|_{H^{s}_{\Lambda}}.
\end{equation}
As in the proofs of Lemma~\ref{l:Fred1}, an identical argument using 
$\|(P_{q,\nu}^*-\bar{\omega})u\|^2_{H^{s}_{\Lambda}}$ 
implies 
\begin{equation}
\label{e:octopus}
\|u\|_{H^{s-\epsilon}_\Lambda}\leq \|(P_{q,\nu}^*-\bar{\omega})u\|_{H^{s}_{\Lambda}}.
\end{equation}
Since, $P_{0,\nu}-P_{q,\nu}:H^{s+2}_{\Lambda}\to H^{N}_{\Lambda}$ for any $N$, $P_{q,\nu}-\omega$ is a Fredholm operator. In particular,~\eqref{e:squid} and~\eqref{e:octopus} imply that $\mc{R}_{q,\nu}(\omega)$ exists and satisfies the requisite bounds.
\end{proof}

\subsection{Convergence of the poles of $\mc{R}_{0,\nu}(\omega)$}
We {now finish the proof of Theorem~\ref{t:viscoscity}.}
\begin{proof}[Proof of Theorem~\ref{t:viscoscity}]
First, observe that by Lemma~\ref{l:parametrix} for $\omega\in (-\omega_0,\omega_0)+i(-\omega_0\theta,\infty)$, and $\nu\in [0,\nu_0]$ the inverse $\mc{R}_{q,\nu}:H^s_{\Lambda}\to H^{s}_{\Lambda}$ exists and satisfies
$$
(\Id+i\mc{R}_{q,\nu}(\omega)Q)=\mc{R}_{q,\nu}(\omega)(P_{0,\nu}-\omega).
$$
Moreover, by Lemmas~\ref{l:Fred1} and~\ref{l:fred3}, there is $C_\nu>0$ such that for $\omega \in (-\omega_0,\omega_0) +i(C_\nu,\infty)$, $\mc{R}_{0,\nu}(\omega):H^s_\Lambda\to H^s_{\Lambda}$ exists. Therefore, for $\omega$ in this region, the inverse
$$
(\Id +iR_{q,\nu}(\omega)Q)^{-1}=\mc{R}_{0,\nu}(\omega)(P_{q,\nu}-\omega):H^s_{\Lambda}\to H^s_{\Lambda}
$$
exists. 

Now, for any $N>0$, $Q:H^s_\Lambda\to H^{s+N+\epsilon}_{\Lambda}$ and $\mc{R}_{q,\nu}(\omega):H^{s+N+\epsilon}_{\Lambda}\to H^{s+N}_{\Lambda}$, with uniform bounds in $\nu\geq 0$. Therefore, Lemma~\ref{l:traceDeformed} implies that for any $s$
$$
\mc{R}_{q,\nu}(\omega)Q:H^s_{\Lambda}\to H^s_{\Lambda},
$$
is trace class with uniformly bounded trace class norm. In particular, for $\omega\in (-\omega_0,\omega_0)+i(-\omega_0\theta,\infty)$ the operator
$$
\Id+i\mc{R}_{q,\nu}(\omega)Q:H^s_{\Lambda}\to H^s_{\Lambda},
$$
is Fredholm with index 0. Thus, by the meromorphic version of
Fredholm analyticity (see for instance \cite[Theorem C.10]{dizzy}) 
$$
(\Id+i\mc{R}_{q,\nu}(\omega)Q)^{-1}:H^s_{\Lambda}\to H^s_{\Lambda}
$$
is a meromorphic family of operators satisfying
\begin{equation}
\label{e:resolve}
\mc{R}_{0,\nu}(\omega)=(\Id+i\mc{R}_{q,\nu}(\omega)Q)^{-1}\mc{R}_{q,\nu}(\omega).
\end{equation}
For $ q $ chosen in Lemma \ref{l:parametrix}, $R_{q,\nu}(\omega)$ is analytic in $(-\omega_0,\omega_0)+i(-\omega_0\theta,\infty)$. Hence, the eigenvalues of $P_{0,\nu}$ on $H^0_{\Lambda}$ agree, with multiplicity, with the zeroes of 
$$
f_{\nu}(\omega):={\det}_{H_{\Lambda}^0}(\Id+i\mc{R}_{q,\nu}(\omega)Q).
$$

\begin{lemm}
\label{l:converge}
We have 
$$ f_\nu(\omega)\underset{\nu\to 0}\longrightarrow f_0(\omega) $$ 
uniformly on compact subsets of  $\omega \in (-\omega_0,\omega_0)+i(-\omega_0\theta,\infty)$.
\end{lemm}
\begin{proof}
First, note that
$$
\nu^{-1}[(\mc{R}_{q,\nu}(\omega)-\mc{R}_{q,0}(\omega))Q]=-i \mc{R}_{q,\nu}(\omega)\Delta \mc{R}_{q,0}(\omega)Q.
$$
Since $Q:H^s_{\Lambda}\to H^{s+N}_{\Lambda}$ for any $N$, and, by Lemma~\ref{l:parametrix}, $\mc{R}_{q,\nu}:H^{s}_\Lambda\to H^{s-\e}_{\Lambda}$ with uniform bounds in $\nu$, $R_{q,\nu}(\omega)\Delta R_{q,0}(\omega)Q:H^{s}_{\Lambda}\to H^{s+N}_{\Lambda}$ is uniformly bounded in $\nu$ for any $N$. In particular, Lemma~\ref{l:traceDeformed} implies
$$
\nu^{-1}\|(R_{q,\nu}(\omega)-R_{q,0}(\omega))Q\|_{\mc{L}^1(H^s_\Lambda\to H^s_{\Lambda})}\leq C.
$$
By \cite[Proposition B.29]{dizzy}
$$
|{\det}_{H_{\Lambda}^s} (I+A)-{\det}_{H_\Lambda^s} (I+B)|\leq \|A-B\|_{\mc{L}^1(H^s_\Lambda\to H^s_{\Lambda})}e^{1+\|A\|_{\mc{L}^1(H^s_\Lambda\to H^s_{\Lambda})}+\|B\|_{\mc{L}^1(H^s_\Lambda\to H^s_{\Lambda})}}.
$$
Therefore, since $\mc{R}_{q,\nu}Q:H^s_{\Lambda}\to H^N_{\Lambda}$ is uniformly bounded in $\nu$ for any $N$, the lemma is proved.
\end{proof}

Finally, we show that the eigenvalues of $P_{0,\nu}$ on $H^s_{\Lambda}$ agree with those on $L^2$. Together with Lemma~\ref{l:converge}, this will complete the proof of Theorem~\ref{t:viscoscity}.
\begin{lemm}
Let $\nu>0$, and $G_0,G_1$ satisfy~\eqref{e:G}. Suppose that $u\in H^{s}_{\Lambda_{G_1}}$ and 
\begin{equation}
\label{e:L2Eig}
(P_{0,\nu}-\omega)^Ku=0
\end{equation}
Then $u\in H^{k}_{\Lambda_{G_0}}$ for any $k$. In particular, the spectrum of $P_{0,\nu}$ on $L^2(\mathbb{T}^n)$ agrees with that on $H^s_{\Lambda}$.
\end{lemm}
\begin{proof}
Let $G_\e=(1-\chi(\e|\xi|))G_1+\chi(\e|\xi|)G_0$ with $\chi\in C_c^\infty(\mathbb{R})$, $\chi\equiv 1$ on $[-1/2,1/2]$ and $\supp \chi\subset(-1,1)$. Note that $G_\e$ satisfies~\eqref{e:G} and $\lim_{\e\to 0}G_\e=G_0$ pointwise. In particular, $G_\e=G_0$ on $2|\xi|<\e^{-1}$ and $G_\e=G_1$ on $|\xi|>\e^{-1}$.

Suppose that $u\in H^s_{\Lambda_{G_1}}$ satisfies~\eqref{e:L2Eig}. Then,
\begin{align*}
\|u\|_{H^s_{\Lambda_{G_\e}}}&=\|\langle \xi\rangle^s T_{\Lambda_{G_\e}}u\|_{L^2(\Lambda_{G_\e})}\\
&\leq \|1_{|\xi|\leq \e^{-1}}\langle \xi\rangle^k T_{\Lambda_{G_\e}}u\|_{L^2(\Lambda_{G_\e})}^2+\|1_{|\xi|>\e^{-1}}\langle \xi\rangle^s T_{\Lambda_{G_1}}u\|_{L^2(\Lambda_{G_1})}^2\\
&\leq \|1_{|\xi|\leq \e^{-1}}\langle \xi\rangle^s T_{\Lambda_{G_\e}} {S}_{\Lambda_{G_1}}T_{\Lambda_{G_1}}u\|_{L^2(\Lambda_{G_\e})}^2+\|1_{|\xi|>\e^{-1}}\langle \xi\rangle^s T_{\Lambda_{G_1}}u\|_{L^2(\Lambda_{G_1})}^2\\
&\leq C_\e\|\langle \xi\rangle^sT_{\Lambda_{G_1}}u\|_{L^2(\Lambda_{G_1})},
\end{align*}
where in the last line we use Lemma~\ref{l:changeLagrangian}.
In particular, $u\in H^s(\Lambda_{G_\e})$ for each fixed $\e>0$.  

Since $u\in H^s_{\Lambda_{G_\e}}$, we can apply~\eqref{e:ellipticEst} together with~\eqref{e:L2Eig} to obtain
$$
\|u\|_{H^{k}_{\Lambda_{G_\e}}}\leq C_{k,N,\nu}\|u\|_{H^{-N}_{\Lambda_{G_\e}}}.
$$
where $C_{s,N,\nu}$ does not depend on $\e$. Writing this on the FBI transform side, we have
\begin{align*}
\|\langle \xi\rangle^k T_{\Lambda_{G_\e}}u\|_{L^2(\Lambda_{G_\e})}& \leq C_{k,N,\nu}\|\langle \xi\rangle^{-N}T_{\Lambda_{G_\e}}u\|_{L^2(\Lambda_{G_\e})}\\
&\leq  C_{k,N,\nu}\big(\|1_{|\xi|\leq M}T_{\Lambda_{G_\e}}u\|_{L^2(\Lambda_{G_\e})}+M^{k-N}\|\langle \xi\rangle^{k}T_{\Lambda_{G_\e}}u\|_{L^2(\Lambda_{G_\e})}\big).
\end{align*}
Now, choosing $M\geq (2C_{k,N,\nu})^{\frac{1}{N-k}}$ large enough, and subtracting the last term to the left hand side, we obtain
\begin{align*}
\|\langle \xi\rangle^k T_{\Lambda_{G_\e}}u\|_{L^2(\Lambda_{G_\e})}&\leq C_{k,N,\nu}\|1_{|\xi|\leq M}T_{\Lambda_{G_\e}}u\|_{L^2(\Lambda_{G_\e})}=C_{k,N,\nu}\|1_{|\xi|\leq M}T_{\Lambda_{G_0}}u\|_{L^2(\Lambda_{G_0})}\\
&=C_{k,N,\nu}\|1_{|\xi|\leq M}T_{\Lambda_{G_0}} {S}_{\Lambda_{G_1}}T_{\Lambda_{G_1}}u\|_{L^2(\Lambda_{G_0})}\\
&\leq C_{k,s,N,\nu}\|\langle \xi\rangle^s T_{\Lambda_{G_1}}u\|_{L^2(\Lambda_{G_1})},
\end{align*}
where in the last line we apply Lemma~\ref{l:changeLagrangian}.

In particular, sending $\e\to 0^+$, we have that 
$$
\limsup_{\e \to 0}\|\langle \xi\rangle^k T_{\Lambda_{G_\e}}u\|_{L^2(\Lambda_{G_\e})}\leq C\|u\|_{H^s_{\Lambda_1}}.
$$
Finally, by Fatou's lemma together with the fact that $G_\e\to G_0$, this implies
$$
\|u\|_{H^k_{\Lambda_{G_0}}}=\|\langle \xi\rangle^k T_{\Lambda_{G_0}}u\|_{L^2(\Lambda_{G_0})}\leq C\|u\|_{H^s_{\Lambda_1}},
$$
and in particular, $u\in H^k_{\Lambda_{G_0}}$ as claimed.
\end{proof}
This completes the proof of Theorem \ref{t:viscoscity}.
\end{proof}

\subsection{Poles of the resolvent $\mc{R}_{0,0}(\omega)$ in the upper half plane}
Finally, we study the behavior of $ \mc{R}_{0,0}(\omega) = 
 \mc{R} ( \omega ) $ 
 for $\Im \omega\geq 0$ and complete the proof of Theorem~\ref{t:meromorphy}. (That resolvent was defined in 
 Lemma \ref{l:Fred1}.)
\begin{lemm}
\label{l:poleL2}
Suppose that 
$$
\omega_1\in \{|\Re\omega|<\omega_0,\,\Im \omega \geq 0\}\setminus 
{\rm{spec}}_{{\rm{pp}}, L^2} (P).$$ Then, $\mc{R} (\omega)$ is analytic near $\omega_1$. Moreover, for $s\in \mathbb{R}$, $\Im \omega> 0$ and $u\in L^2(\mathbb{T}^n)\cap H^s_\Lambda$, $\mc{R} (\omega)u=\mc{R}^{L^2}(\omega)u$ where $\mc{R}^{L^2}$ denotes the $L^2$ resolvent for $P$.

Conversely, if $\omega_1\in (-\omega_0,\omega_0)\cap \specppL(P)$, then $\omega_1$ is a pole of $\mc{R}(\omega)$ and 
$$
\mc{R}(\omega)=A(\omega)+\frac{\Pi_{\omega_1}}{\omega-\omega_1},
$$
where $A(\omega):H^s_{\Lambda}\to H^{s}_{\Lambda}$ is analytic near $\omega_1$ and $\Pi_{\omega_1}$ is the orthogonal projection onto the $L^2$ eigenspace of $P$ at $\omega_1$. 
\end{lemm}
\begin{proof}
For $\Im \omega>0$, the spectral theorem shows that the resolvent 
of $ P $, 
\[ \mc{R}^{L^2}(\omega):= ( P - \omega)^{-1} : L^2(\mathbb{T}^n)\to L^2(\mathbb{T}^n) \]
 exists and is analytic. Also, Lemma ~\ref{l:Fred1} 
shows that $ \mc R ( \omega ) = \mc R_{0,0} ( \omega) $ 
is a meromorphic family of operators in $(-\omega_0,\omega_0)+i(-c_0 \theta ,\infty)$.  

Lemma \ref{l:B5} implies that  for $ u\in \mathscr A_{\delta} $ 
(defined in~\eqref{e:analytic})
and $\Im \omega \gg 1$,  $ R^{L^2} ( \omega ) u\in \mathscr {A}_\delta$
Since $ P : H_\Lambda \to H_\Lambda $ and $ \mathscr A_\delta \subset H_\Lambda $, 
$ \mathcal R (\omega )  ( P - \omega ) |_{\mathscr A_\delta}=\Id_{\mathscr A_\delta} $.  Therefore, for $ u \in \mathscr A_\delta $,
\[ \mc R^{L^2} ( \omega ) u= \left[ \mc R ( \omega ) ( P - \omega ) \right] \mc R^{L^2} u = \mc R ( \omega ) \left[ ( P - \omega ) 
\mc R^{L^2} ( \omega ) \right] u  = \mc R ( \omega ) u .\]
Since $\mathscr{A}_\delta$  are dense in both $L^2(\mathbb{T}^n)$ and $H^s_\Lambda$, 
$$
\mc{R}(\omega)u=\mc{R}^{L^2}(\omega)u, 
\ \ \Im \omega>0, \ \ \Re\omega \in(-\omega_0,\omega_0), \ \ 
u\in L^2(\mathbb{T}^n)\cap H^s_{\Lambda}.
$$
This proves the first part of the lemma. 

To prove the secod part, let $\omega_1\in (-\omega_0,\omega_0)$ and $\Pi_{\omega_1}:L^2(\mathbb{T}^n)\to L^2(\mathbb{T}^n)$ be the orthogonal projection onto the $\omega_1$ eigenspace for $P$ (possibly the zero operator if $\omega_1$ is not an embedded eigenvalue for $P$). By~\cite[Lemma 3.2]{DyZw19}, the $\omega_1$ eigenfunctions of $P$ are smooth and hence $\widetilde{P}=P+\Pi_{\omega_1}$ has the same symbol as $P$ and no embedded eigenvalue at $\omega_1$.  Moreover, we may choose $\e>0$ so small that $\omega_1$ is the only embedded eigenvalue for $P$ in $|\omega-\omega_1|<\e$. Then, for $0<|\omega-\omega_1|<\e$, $\Im \omega> 0$,
\begin{align*}
\mc{R}^{L^2} (\omega)&= (\widetilde{P}-\omega)^{-1}+(P-\omega)^{-1}\Pi_{\omega_1}(\widetilde{P}-\omega)^{-1}\\
&= (\widetilde{P}-\omega)^{-1}+\frac{\Pi_{\omega_1}}{(\omega_1-\omega)(1+\omega_1-\omega)}.
\end{align*}
Note that by~\cite[Lemma 3.3]{DyZw19} for $\omega\in(\omega_1-\e,\omega_1+\e)$, the limiting absorption resolvent $(\widetilde{P}-\omega-i0)^{-1}:H^{1/2+0}\to H^{-1/2-0}$ exists. 

The meromorphy of $\mc{R}(\omega):H^s_\Lambda\to H^s_{\Lambda}$ (Lemma 
\ref{l:Fred1}) gives
\begin{equation}
\label{e:meroPole}
\mc{R}(\omega)=A(\omega)+\sum_{j=1}^K\frac{B_j}{(\omega-\omega_1)^j}
\end{equation}
where $A:H^{s}_{\Lambda}\to H^{s}_{\Lambda}$ is holomorphic near $\omega_1$. 
Therefore, for $|\omega-\omega_1|<\e$, $\Im \omega>0$, and $u\in L^2(\mathbb{T}^n)\cap H^s_{\Lambda}$,
\begin{equation}
\label{e:poles}
\begin{split}
&(\omega-\omega_1)^K A(\omega)u+\sum_{j=1}^{K}(\omega-\omega_1)^{K-j}B_ju\\
&\qquad=(\omega-\omega_1)^K\Big[(\widetilde{P}-\omega)^{-1}-\frac{\Pi_{\omega_1}}{1+\omega_1-\omega}\Big]u -(\omega-\omega_1)^{K-1}\Pi_{\omega_1}u.
\end{split}
\end{equation}

Let $u\in \mathscr{A}_\delta\subset H_{\Lambda}^s\cap H^{\frac{1}{2}+0}$. Then, using~\eqref{e:poles} with $\omega=\omega_1+ir$, we obtain 
\begin{align*} 
B_Ku&=\lim_{r\to 0^+}[(ir)^KA(\omega_1+ir)u+\sum_{j=1}^K (ir)^{K-j}B_ju]\\
&=\lim_{r \to 0^+}(ir)^K\Big[(\widetilde{P}-\omega_1-ir)^{-1}-\frac{\Pi_{\omega_1}}{1-ir}\Big]u-(ir)^{K-1}\Pi_{\omega_1}u\\
&=\delta_{K1}\Pi_{\omega_1}u
\end{align*}
Since $\mathscr{A}_{\delta}$ is dense in both $H^{\frac{1}{2}+0}$ and $H^s_{\Lambda}$, $B_K=\delta_{K1}\Pi_{\omega_1}$. In particular, we may write~\eqref{e:meroPole} with $K=1$ and by the same argument obtain $B_1=\Pi_{\omega_1}$. 
\end{proof}

\appendix
\section{Review of some almost analytic constructions}

%\vspace{0.5cm}
%
%\begin{center}
%\noindent
%{\sc  Appendix: Review of some almost analytic constructions.}
%\end{center}
%\vspace{0.4cm}
%\renewcommand{\theequation}{A.\arabic{equation}}
%\refstepcounter{section}
%\renewcommand{\thesection}{A}
%\setcounter{equation}{0}

Here we include some facts about almost analytic functions and 
manifolds. For an in-depth presentation see \cite[\S 1-3]{mess} 
and \cite[Chapter X]{tre}.

\subsection{Almost analytic manifolds} Let $ U $ be an open subset of $ \CC^m $ and let $ U_\RR := U \cap \RR^m $.
We define an almost analytic function as follows:
\[  f \in C^{\rm{aa}} ( U ) \ \Longleftrightarrow \ 
\partial_{\bar z } f ( z ) = \mathcal O_K ( | \Im z |^\infty ) , \ \ 
z \in K \Subset U . \]
This definition is non-trivial only for $ U_\RR \neq \emptyset $.
We write $ f \sim 0 $ in $ U $ if $ f ( z )  = \mathcal O_K ( | \Im z |^\infty ) $, $ z \in K \Subset U \subset \CC^m $.
We note that (see \cite[Lemma X.2.2]{tre}) that for $ f \in \CI $ that
implies $ \partial^\alpha f \sim 0 $ in $ U $.

We also need the notion of an almost analytic manifold. Let $\Lambda\subset \mathbb{C}^m$ be a smooth manifold and $\Lambda_\RR:=\Lambda\cap \RR^m$. We say that $\Lambda$ is almost analytic if near any point $z_0\in \Lambda_\RR$, there exist a neigbourhood $U$ of $z_0$ in $\CC^m$ and  functions $f_1,\dots,f_k\in  \CI (\CC^m)$ such that: 
\begin{gather*} \Lambda \cap U =\{ z : f_j ( z ) = 0 , 1 \leq j \leq k \}, \ \ \partial_z f_j (z_0) \ \text{ are linearly independent,} \\
|\partial_{\bar z} f_j ( z ) | = \mathcal 
O ( | \Im z |^\infty + | \sup_{1 \leq \ell\leq k} f_\ell ( z ) |^\infty ),
\end{gather*}
see \cite[Theorem 1.4]{mess}.

A special case is given by 
\begin{equation}
\label{eq:fjhj} 
 f_j ( z ) = z_j - h_j (z') ,  \ \ \ z' := ( z_{k+1} , \cdots, z_n ) . 
\end{equation} 

Equivalence of two almost analytic manifolds can be defined as follows (see \cite[Definition 1.6, Proposition 1.7]{mess}): suppose $ \Lambda_1 \cap \RR^m = 
\Lambda_2 \cap \RR^m $ and that $ \Lambda_k $ is defined by 
\eqref{eq:fjhj} with $ h = h_{k} $, $ k = 1, 2$, respectively. Then 
\[    \text{$ \Lambda_1 $ and $ \Lambda_2 $ are equivalent as almost analytic submanifolds (denoted $\Lambda_1\sim \Lambda_2$)} \]
if and only if, on compact subsets,
\begin{equation}
\label{eq:h1h2} 
\begin{gathered}
| h_1  ( z') - h_2 ( z' ) | = \mathcal O ( |\Im h_1 ( z' ) |^\infty ) 
\\
{\text{or, equivalently,}}
\\
| h_1  ( z') - h_2 ( z' ) | = \mathcal O ( | \Im z'|^\infty + |\Im h_1 ( z' ) |^\infty ) .
\end{gathered}
\end{equation}

We now consider {\em almost analytic vector fields}:
\[ V = \sum_{ j=1}^m a_j ( z ) \partial_{z_j} , \ \ a_j \in C^{\rm{aa}} ( \CC^n ) , \]
which we identify with {\em real vector fields} $ \widehat V $
such that for $ u $ holomorphic $ \widehat V f = V $:
\[ \begin{split} 
\widehat V & :=  V + \bar V = 2 \Re V \\
& = \sum_{ j=1}^m  \Re a_j ( z ) ( \partial_{z_j} +
\partial_{\bar z_j } ) + i \Im a_j ( z ) ( \partial_{z_j} -
\partial_{\bar z_j } )  \\
& = \sum_{ j=1}^m  \Re a_j ( z ) \partial_{\Re z_j} +  \Im a_j ( z )  \partial_{\Im z_j} . \end{split} \]

\noindent
{\bf Example.} Suppose $ M \subset \CC^m $, $ \dim_{\RR} M = 2k $ is almost analytic. Then 
vector fields tangent to $ M $ are spanned by almost analytic vector fields, $ V_j = a_j ( z) \cdot \partial_z $, $ \partial_{\bar z }
a_j ( z )  = \mathcal O ( |\Im z |^\infty )$, $ z \in M $, $ j = 1, 
\cdots k $. In fact, using \cite[Theorem 1.{4}, 3$^{\circ}$]{mess} 
we can write  {$M$ locally near any $z\in M\cap \RR^m$} as $ \{ ( z' , h ( z' ) ) : z' \in 
\CC^k \} $, $ h = (h_{k+1}, \cdots , h_m ) : \CC^k \to \CC^{m-k} $, $ \partial_{\bar z } h =
\mathcal O ( |\Im z'|^\infty + |\Im h ( z' ) |^\infty ) $. We then 
put
\begin{equation}
\label{eq:Vj}  V_j = \partial_{z_j} + \sum_{ \ell = k+1}^m \partial_{z_j} h_\ell (z') \partial_{z_\ell} .
\end{equation}
The real vector fields $ \widehat V_j $ then span the vector fields tangent to $M $. \qed

\medskip

Following \cite{mess} and \cite{Sj74} we define the (small complex time)
flow of $ V $ as follows for $ s \in \CC $, 
$ |s | \leq \delta $ 
\begin{equation} 
\label{eq:defesV}  \Phi_s  (z ) := \exp { \widehat {sV } } ( z ) .
\end{equation}
The right hand side is the flow out at time $ 1 $ of the real 
vector field $ \widehat{sV} $. Unless the coefficients in $ V $ are
holomorphic $ [ \widehat V, \widehat { i V } ] \neq 0 $ which means
that $ \exp ( s + t ) V \neq \exp sV \exp t V $ for $ s, t \in \CC $.
However, we still have $ [ \widehat { i V } , \widehat V ] \sim 0 $.

\begin{lemm}
\label{l:flow1}
Suppose that $ \Gamma \in \CC^m $ is an embedded almost 
analytic submanifold {of real dimension $2k$} and $ V $ is an almost analytic vector field. 
Assume that, 
\begin{equation}
\label{eq:ViVin} \text{ $ \widehat V $, $ \widehat {iV } $ are linearly independent with span transversal to $ \Gamma $,} \end{equation}
and that, 
in the notation of \eqref{eq:defesV},
\begin{equation}
\label{eq:assimP}
| \Im \Phi_t ( z ) | \geq |t|/C_K , \ \ z \in K \Subset \Gamma . 
\end{equation}
%represented (locally) by
%\begin{gather*}  \Gamma = \{ ( z', h ( z') ) : z' \in \CC^k \} , \ \ 
%h: \CC^k \to \CC^{m-k} , h ( 0 ) = 0 \\
%\partial_{\bar z'} h ( z' ) = \mathcal ( |\Im z'|^\infty + 
%| \Im h ( z')|^\infty ) . \end{gather*}
%Suppose $ V = a' \partial_{z'} + a'' \partial_{z''} $ is an almost analytic vector field such that
%\begin{gather*}   \Re a'' ( 0 )  ,\  \Im  a'' ( 0 ) \in 
%RR^{n-k} \te
%xt{ are linearly independent ,} \\
%\Im a' ( 0 ) = \Re a' ( 0 ) = 0 . \end{gather*}
%$ \dim_{\RR } \Gamma = 2k $, $ k < m $, and that 
%\begin{equation}
%\label{eq:ind_tra}  \text{ $\widehat { V }$ and $ \widehat { i V } $ are linearly independent and transversal to $ \Gamma $ and $ \RR^n \subset \CC^n $ at $ \Gamma_\RR $.} 
%\end{equation}
Then 
for  any $ U \Subset \CC^m $, there exists $ \delta $ such that 
\[ \Lambda  := \left\{ \exp \widehat {t V} ( \rho ) 
%\left( \Im s \, \widehat {i V } \right) \circ \exp \left( \Re s \, \widehat V \right)  ( \rho) 
: \rho \in \Gamma \cap U  , \ \ |t| < \delta  , \ t \in \CC \right\}  \]
is an almost analytic manifold, $ \Lambda_\RR = \Gamma_\RR $ and  
$ \dim_{\Re} \Lambda = 2k +2$. \end{lemm} 

We will use the following geometric lemma:
\begin{lemm}
\label{l:geof}
Suppose $ Z_j \in \CI ( \RR^m ; T^* \RR^m ) $, $ j = 1, \cdots , J $, are smooth vector fields
and, for $ s \in \RR^J $, 
\[ \langle s, Z \rangle := \sum_{j=1}^J s_j Z_j \in \CI ( \RR^m ; T^* \RR^m ) .\]
Then for $ f \in \CI ( \RR^m ) $
\begin{equation}
\label{eq:pullbyexp}
%\partial_s^\alpha 
f ( e^{ \langle s , Z \rangle} ( \rho ) ) =
\sum_{ p=1}^{P} \frac{1}{p!} ( \langle s, Z \rangle)^k f ( \rho ) + 
\mathcal O_K ( |s|^{P+1} ), \ \ \rho \in K \Subset \RR^m . 
\end{equation}
while for $ Y \in \CI ( \RR^m ; T^* \RR^m ) $, 
\begin{equation}
\label{eq:pushbyexp}
e^{ \langle s , Z \rangle }_* Y ( \rho ) = 
\sum_{p=1}
^{P} \frac{1}{p!} \ad_{ \langle s, Z \rangle}^k Y ( \rho ) + 
\mathcal O_K ( |s|^{P+1} ), \ \ \rho \in K \Subset \RR^m .
\end{equation}
\end{lemm}
For a proof see for instance \cite[Appendix A]{fred}. We recall that
$ F_* Y ( F ( \rho ) ) := dF ( \rho ) Y ( \rho ) $.

%\noindent
%{\bf Remark.} This definition of the flow out is not unique as we could 
%have, for instance, chosen a different order of (non-commuting) vector fields. However, different definitions produce equivalent 
%almost analytic manifolds. We recall here that two almost analytic 
%submanifolds are equivalent, $ \Lambda \sim \Lambda' $ if
% $ \dim_\RR \Lambda = \dim_\RR
%\Lambda' $ and for any $ K  \Subset \Lambda' $, $ d ( z , \Lambda ) = \mathcal O_K  ( |\Im z |^\infty )$, 
%$ z \in K $ -- see \cite[Proposition 1.7]{mess} for conditions equivalent to this. 

\begin{proof}[Proof of Lemma \ref{l:flow1}] 
Let $ \iota : \Gamma \hookrightarrow \CC^m $ be the inclusion
map. Then 
\[  \partial  \exp ( t_1  \widehat { V } + t_2  \, \widehat { i V } )  \circ \iota (\rho ) 
: T_{(0,\rho)} (\RR^2_t \times \Gamma ) \to T_{\rho} \CC^{m} \]
is given by 
$ ( T, X ) \mapsto T_1 \widehat{  V } + T_2 \widehat { i V } + \iota_* X , $
which, thanks  {to} our assumptions, is surjective onto a $ 2k +2 $ (real) dimensional subspace of $ T^* \CC^m $. Hence, by the implicit function theorem $ \Lambda $ is a $ 2k+2 $ dimensional embedded submanifold of $ \CC^m $. 

To fix ideas we start with the simplest case of $
\Gamma = \{ 0 \} \subset \CC^n $. In that case
$ \{ \Lambda = \{ \Phi_t ( 0 ) : t \in \CC , |t| < \delta \}$, and 
from our assumption $ |\Im \Phi_t ( 0 )| \sim |t_1 \widehat V + 
t_2 \widehat { i V } | \sim | t | $.
The tangent space is given by 
\[ T_{ \Phi_t ( 0 ) } \Lambda = \{ \partial_t \Phi_t ( 0 ) 
T + \partial_{\bar t } \Phi_t ( 0 ) \bar T : T \in \CC \}
\subset \CC^2  .\]
If we show that  
\begin{equation}
\label{eq:parttP} \partial_{\bar t }\Phi_t ( 0 ) = \mathcal O ( |t|^\infty) 
\end{equation} then  
$ d (  T_{ \Phi_t ( 0 ) } \Lambda , i  T_{ \Phi_t ( 0 ) } \Lambda ) = 
\mathcal O ( t^\infty ) $ and almost analyticity of $ \Lambda $ 
follows from \cite[Theorem 1.4, 1$^\circ$]{mess}. 
The estimate \eqref{eq:parttP} will follow from showing that for
any  
holomorphic function $ f $, $  {\partial_{t_1}^{\alpha_1}\partial_{t_2}^{\alpha_2}}\partial_{\bar t} f ( \Phi_t ( 0 ) ) |_{t= 0 } = 0 $. But this follows from 
\eqref{eq:pullbyexp} and the fact that $ [ \widehat V , \widehat { i V } ] \sim 0 $ at $ 0$. Indeed, 
\begin{equation}
\label{eq:partf}  \begin{split} 
 {\partial_{t_1}^{\alpha_1}\partial_{t_2}^{\alpha_2}} \partial_{\bar t} f (\Phi_t ( 0 ) ) |_{t=0} & = 
 {\partial_{t_1}^{\alpha_1}\partial_{t_2}^{\alpha_2}}  \partial_{\bar t } \left( \sum_{k=0}^\infty \frac{1}{k!}
\left( t_1 \widehat V + t_2 \widehat { iV } \right)^k f (0 ) \right)|_{t=0} \\
& =    {\partial_{t_1}^{\alpha_1}\partial_{t_2}^{\alpha_2}}
\left( \sum_{k=0}^\infty \frac{1}{k!}
\left( t_1 \widehat V + t_2 \widehat { iV } \right)^k 
( \widehat V + i  \widehat { i V } )f (0 ) \right)|_{t = 0 } 
\\
& = \widehat V^{\alpha_1 } \widehat { i V}^{\alpha_2} ( 
\widehat V + i \, \widehat { i V } ) f ( 0 ) = 
\widehat V^{\alpha_1 } \widehat { i V}^{\alpha_2}  ( V - V ) f ( 0 ) = 0.
\end{split}
\end{equation} 
The fact that $ \widehat V $ and $ \widehat {i V }$ commute to infinite 
order at $ 0 $ was crucial in this calculation. Holomorphy of $ f $
was used to have $ \widehat W f = W f $. 

We now move the general case. For $ z \in \Gamma $, 
$ T_{\Phi_t ( z ) } \Lambda $ is spanned by 
\begin{equation}
\label{eq:TanPht}   \partial_t \Phi_t ( z ) T + \partial_{\bar t} \Phi_t ( z ) \bar T , \ T \in \CC, \ \ 
d \Phi_t ( z) X , \ \  X \in T_z \Gamma . \end{equation}
%(Here as elsewhere we consider $ X $ as a vector in $ T_z \CC^n \simeq 
%\CC^n $.) 
We can repeat the calculation \eqref{eq:partf} with $ 0 $ replaced by 
$ z $ to see that, using the assumption \eqref{eq:assimP} and
the fact that $ \Im \Phi_t ( z ) = \Im z + \mathcal O ( t )  $, 
\begin{equation} 
\label{eq:bartphi} \partial_{\bar t } \Phi ( z ) = \mathcal O ( |t|^\infty + 
|\Im z |^\infty ) = \mathcal O ( |\Im \Phi_t ( z ) |^\infty ) . 
\end{equation}
To consider $ d \Phi_t ( z ) X = (\Phi_t)_* Y ( \Phi_t ( z ) )  $ we choose a vector field 
tangent to $ \Gamma $, $ Y $, $ Y_c ( z ) = X $. We choose 
\begin{equation}
\label{eq:YcWc} 
 Y_c  = \widehat 
W_c , \ \ \  W_c = \sum_{ j=1}^k c_j V_j , \ \  c \in \CC^k ,
\end{equation}  
a constant coefficient linear combination of vector fields \eqref{eq:Vj}. 
Then $ d \Phi_t ( z ) X = (\Phi_t)_* Y_c ( \Phi_t ( z ) ) $ and we 
want to show that
\begin{equation}
\label{eq:almostca}   c \mapsto (\Phi_t)_* Y_c ( \Phi_t ( z ) )  \ 
\text{ is complex linear modulo errors  
$\mathcal O ( |\Im \Phi_t ( z ) |^\infty ) $.} \end{equation}
 In view of \eqref{eq:TanPht}
that shows that $ d ( T_{ \Phi_t ( z ) } \Lambda , i T_{\Phi_t ( z ) }
\Lambda ) = 
\mathcal O ( |\Im \Phi_t ( z ) |^\infty ) $ and from 
\cite[Theorem 1.4, 1$^\circ$]{mess} we conclude that $ \Lambda $ is
almost analytic.

To establish \eqref{eq:almostca} we use \eqref{eq:pushbyexp}
with $ \langle s , X \rangle = s_1 \widehat V + s_2 \widehat{i V } $, $ s_1 = \Re t$, $ s_2 = \Im t $. Since $ [ \widehat V , 
\widehat {i V} ] \sim 0 $  and $ \widehat V \sim \widehat { i V }/i $
at $ \Im w = 0 $, we see that
\begin{equation}
\label{eq:Phit}  (\Phi_t )_* Y_c ( w ) = \sum_{ p=0}^\infty \frac {t^p} {p!} 
{\ad_{\widehat V}^p W_c } ( w ) + \mathcal O ( |t|^{K+1} +
|\Im w |^\infty ) . \end{equation}
Because of the form of $ W_c $ (see \eqref{eq:Vj} and \eqref{eq:YcWc}) 
\[ \ad_{\widehat V}^p W_c ( w) = \widehat {\ad_{V}^p W_c } ( w ) 
+ \mathcal O ( |\Im w' |^\infty + |\Im h(w')|^\infty ) , \]
and 
\[ c \mapsto  {\ad_{V}^p W_c } ( w )  \ \text{ is complex linear}. \] 
Since $ w = \Phi_t ( z ) $, $ z \in \Gamma $, 
\[ \begin{split}  | \Im w' | + | \Im h ( w') | & = \mathcal O ( |\Im z' | + | \Im h ( z' ) | + |t| ) \\
& =  
\mathcal O  ( | \Im z | + |t | ) = 
\mathcal O  ( | \Im w | + |t | ) = \mathcal O ( | \Im w| ), 
\end{split} \]
since $ |\Im  w| = |\Im \Phi_t ( z ) | \geq |t|/C $. Combining this estimates with \eqref{eq:Phit} gives \eqref{eq:almostca}.
\end{proof}

\subsection{Almost analytic generating functions}

We now recall how to obtain 
generating functions for almost analytic strictly positive Lagrangian manifolds. We recall that 
an almost analytic submanifold of $ T^* \CC^n $, $ \Lambda $,
is Lagrangian if
\[  ( \omega_\CC )|_\Lambda \sim 0 , \ \ \omega_\CC := \sum_{ j=1}^n 
d\zeta_j \wedge d z_j . \]
In addition 
we say that $ \Lambda $ is \emph{strictly positive} (\cite[Definition 3.3]{mess}) if $\Lambda_\RR$ is a submanifold of $ T^*\RR^m$ and for all $\rho\in \Lambda_\RR$,
\begin{equation}
\label{eq:strpos}
\frac 1 i \sigma(V,\bar{V})(\rho)>0, \ \ \text{for all } \  V\in T_\rho \Lambda\setminus (T_\rho\Lambda_\RR)^{\mathbb{C}}.
\end{equation}
\begin{lemm}
\label{l:pos}
Suppose that $\Lambda$ is a strictly positive almost analytic 
Lagrangian submanifold of $ \nbhd_{T^* \CC^m } ( 0 ) $ given by 
$$
\{(z,\zeta(z))\mid z\in \nbhd_{\CC^m}(0)\}.
$$
Then there is $C>0$ such that for $x\in \nbhd_{\RR^m}(0) $, 
\begin{equation} 
\label{eq:LemA3}
\frac{1}{C}d(x,\pi({\Lambda_\RR}))\leq |\Im \zeta(x)|\leq Cd(x,\pi({\Lambda_\RR})).
\end{equation}
\end{lemm}
\begin{proof}
Since $\Lambda_\RR$ is a submanifold of $T^*\mathbb{R}^m$, we may choose real coordinates on $\RR^m$ such that (near $ 0 $)
$$
\Lambda_\RR=\{(x',0, \zeta(x',0))\mid (x',x'')\in \mathbb{R}^k\times \mathbb{R}^{m-k}\}.
$$
Then, with  $\rho=(x',0,\zeta(x',0))$,
$
(T_\rho\Lambda_\RR)^\CC=\{(\delta_{z'},0, \partial_{x'}\zeta(x',0) \delta_{z'})\mid \delta_{z'}\in \mathbb{C}^k\}
$, 
and it follows that for all $\delta_{x''}\in \mathbb{R}^{m-k}$,
$$
(0,\delta_{x''},\partial_{x''}\zeta(x',0)\delta_{x''})\in T_\rho \Lambda\setminus (T_\rho\Lambda_\RR)^\CC.
$$
Strict positivity of $\Lambda$ then implies that 
\[ 
\frac 1 i \sigma\big((0,\delta_{x''},\partial_{x''}\zeta(x',0)\delta_{x''}),\overline{(0,\delta_{x''},\partial_{x''}\zeta(x',0)\delta_{x''})}\big)=2\langle \Im \partial_{x''}\zeta\delta_{x''},\delta_{x''}\rangle >0. 
\]
Since in our coordinates,
$$
\frac{1}{C}|x''|\leq d(x,\pi(\Lambda_\RR))\leq C |x''|,
$$
\eqref{eq:LemA3} follows.
\end{proof}

With this lemma in place we can find generating functions in the 
almost analytic setting:
\begin{lemm}
\label{l:generate}
Suppose that $ \Lambda $ is a {strictly positive} almost analytic Lagrangian 
submanifold of $ \nbhd_{T^* \CC^m } ( 0 ) $ and that
$ \pi_*:  T_{(0,0)} \Lambda \to T_0 \CC^m $ is onto. Then 
there exists $ \Psi \in \CI ( \nbhd_{\CC^m } ( 0 ) ) $ satisfying
\begin{equation}
\label{eq:generate1}
 \partial_{\bar z } \Psi = \mathcal O ( 
|\Im z |^\infty + | \Im \Psi ( z) |^\infty ) ,
\end{equation}
such that, as almost analytic manifolds, 
\begin{equation}
\label{eq:paraLa} \Lambda \sim \{  ( z , \Psi_z ( z ) ) : |z| < \epsilon \}  , \ \ 
\Psi_z ( 0 ) = 0 . \end{equation}
\end{lemm}
\begin{proof}
Since $ \Lambda $ is an almost analytic Lagrangian, we have 
$ \sigma|_\Lambda \sim 0 $ (vanishes to infinite order at $ \Lambda_\RR $) while the projection property shows that,
near $ z = 0 $, $ \Lambda = \{ ( z, \zeta ( z ) ) : z \in \CC^m \}$, 
$ \zeta ( 0 ) = 0 $.
Hence $ d ( \zeta ( z ) d z ) \sim 0 $ and (see \cite[Theorem 1.4, 3$^\circ$]{mess})
\[ \partial_{\bar z } \zeta ( z ) = \mathcal O ( |\Im z |^\infty + 
|\Im \zeta ( z ) |^\infty ) .
\]
We note that for $ z =x \in \RR^n $, {Lemma~\ref{l:pos} together with} the strict positivity at
$ \Lambda_\RR = \{ ( 0 , 0 )\} $ show that 
\begin{equation}
\label{eq:imzeta}
| x''|/C \leq | \Im \zeta ( x ) | \leq C |x''|, \ \ x \in \RR^n, \ |x| < 
\epsilon . \end{equation}
where $\pi(\Lambda_\RR)$ is given by $\{|x''|=0\}$.
We now see that 
\[  0 \sim \sigma|_\Lambda = \sum_{j=1}^n \partial_z \zeta_j ( z ) \wedge 
d z_j + \mathcal O ( | \Im z|^\infty + |\Im \zeta ( z ) |^\infty )_{
C^\infty ( \CC^n ;\wedge^{2n} \CC^n )} , \]
and in view of \eqref{eq:imzeta} 
\[  \partial_{z_k} \zeta_j ( x ) - \partial_{z_j} \zeta_k ( x ) 
= \mathcal O ( |x'' |^\infty ) , \ \ x \in \RR^n, \ |x| < 
\epsilon . \]
For $ x \in \RR^n $, define $ \Psi $ by a simple version of the Poincar\'e lemma:
$ \Psi ( x ) = \int_0^1 \zeta ( t x ) \cdot x dt $. 
Then
 \begin{equation}
\label{eq:aaPoin} \begin{split} \partial_{x_j} \Psi  ( x ) & =  \int_0^1 \left(\sum_{k=1}^n t z_k \partial_{x_j} \zeta_k ( t x )  + \zeta_j ( t x ) \right) dt \\
& =  \int_0^1\left( \sum_{k=1}^n t z_k \partial_{x_k} \zeta_j ( t x )  
+  \zeta_j ( t x ) \right) dt + \mathcal O ( |{x''}|^\infty  ) \\
& = \int_0^1 \partial_t ( t \zeta_j ( t x ) ) dt + 
 \mathcal O ( |{x''}|^\infty )  = \zeta_j ( x ) + \mathcal O ( | \Im \zeta ( x ) | ^\infty) ,
\end{split} 
\end{equation}
in the last argument we used \eqref{eq:imzeta} again. 
We now define $ \Psi ( z ) $ as an almost analytic extension of 
$ \Psi $. From \cite[Proposition 1.7(ii)]{mess} we obtain \eqref{eq:paraLa}. 
\end{proof}

\subsection{Integration of almost analytic vector fields}
\label{A:trans}

Here we show how to solve transport equations arising in \S \ref{s:trans}. 
For clarity we present a simpler case (see also \cite[\S 5.2.2]{GZ}). Thus 
we assume that $ V $ is an almost analytic vector field on $ \CC^n $ (
$ w = ( w_1, w' ) \in \CC^n $, $ w_1 \in \CC$, $ w' \in \CC^{n-1} $)
satisfying
\[  | \Im  \exp ( \widehat {t V } ) ( 0 , w' ) | \geq |t|/C, \ \ 
w' \in B_{\CC^{n-1} } ( 0 , \epsilon) ,  
 \ \ t \in \CC, \ \ |t| < \epsilon , \ \ dw_1 ( V ) \neq 0.  \]
Then $ ( t, w' ) \mapsto \exp ( \widehat { t V } ( 0 , w' ) ) =: z ( t, w' ) $ is a diffeomorphism for $ \epsilon $ small enough.
We solve 
\begin{equation}
\label{eq:Vab}   V a \sim b, \ \ a ( 0, w' ) = a_0 ( w' ) , \end{equation}
by putting
\[ a ( z) := a_1 ( z ) + a_2 (z ), \ \  a_1 ( z ) = a_0 ( w' ( z ) ) , \ \ 
a_2 ( z ( t, w')) :=  \int_0^1 t b ( z ( ts, w' ) )ds . \]
We calculate the action of $ V $ on $ a_1 $
using almost analyticity of $ b  $, the properties of 
$ z ( t, w' ) $ and \eqref{eq:pullbyexp}:
{\begin{align*}
( V a_2 ) ( z ( t, w' )) & \sim\int_0^1 \sum_{k=0}^\infty \frac{s^k}{k!}V\widehat{tV}^ktb (z(0,w'))ds +\mathcal O(|t|^\infty) \\
& \sim \int_0^1 \sum_{k=0}^\infty \frac{s^k}{k!}\widehat{tV}^{k+1}b (z(0,w')) ds + \mathcal O(|t|^\infty)+\mathcal O(|\Im z(0,w')|^\infty) \\
&\sim \sum_{k=0}^\infty \frac{1}{(k+1)!}\widehat{tV}^{k+1}b(z(0,w'))+\mathcal O(|t|^\infty)+ \mathcal O(|\Im z(0,w')|^\infty)\\
&= b (z(t,w'))+\mathcal O(|t|^\infty+|\Im z(0,w')|^\infty)\\
&= b (z(t,w'))+\mathcal O(|\Im z|^\infty).
\end{align*}}
Similarly, $ V a_1 \sim 0 $ and we obtain \eqref{eq:Vab}.

\section{Physical deformations and numerical results}

%In some circumstances, we may consider physical deformations of $\mathbb{T}^n$ rather than the more complicated phase space deformations. In particular, when there is a $G(x,\xi)$ that is linear in $\xi$ satisfying $H_pG<0$, this is possible. For numerical illustration we consider the simple example
%$$
%P= \langle D\rangle^{-1}D_{x_2}+\sin(x_1)(\Id-V(D_{x_1}))+(\Id-V(D_{x_1}))\sin(x_1).
%$$
%with $V(\xi_1)$ satisfying 
%$$
%|V(\xi_1)|\leq C,\qquad |\Im \xi_1|<b\langle \Re \xi_1\rangle.
%$$
%$$
%G(x,\xi)=\langle G_0(x),\xi\rangle ,\qquad G_0(x)=G_0(x_1)=(2\cos(x_1),0).
%$$
%Then, we consider the complex deformed operator, $P_\theta$, such that, when $u$ is analytic, 
%$$
%P_\theta u|_{\Gamma_\theta}=(Pu)|_{\Gamma_\theta},\qquad \Gamma_\theta:=\{ x-i\theta G_0(x)\mid x\in \mathbb{T}^n\}. 
%$$
%Let $\gamma_\theta(x_1)=x_1-i\theta G_0(x_1).$
%
%Next, observe that 
%$$
%(D_{x_1})_{\theta}=\frac{1}{\gamma_\theta'(x_1)}D_{x_1},\qquad (D_{x_2})_{\theta}=D_{x_2}.
%$$
%
%Note that 
%$$
%(\Id+D_{x_2}^2+D^2_{x_1})_{\theta}= \Id+D^2_{x_2}+\frac{1}{\gamma_\theta'(x_1)}D_{x_1}\frac{1}{\gamma_\theta(x_1)}D_{x_1}=\Id +D^2+ R
%$$
%where 
%$$
%R=\big(\frac{1}{\gamma_\theta'(x_1)^2}-1\big)D_{x_1}^2-\frac{\gamma_{\theta}''}{i\gamma_\theta'^3}D_{x_1}=O(\theta)_{H^{s}\to H^{s-2}}.
%$$
%We need
%{\bf{There is no functional calculus here!!! How do we define $\langle D\rangle^{-\frac{1}{2}}_\theta$} easily?}

The purpose of this appendix is to illustrate our results by numerical examples. We have not yet implemented the general theory numerically. However, in some circumstances, it is enough to consider physical deformations of $\mathbb{T}^n$ rather than the more complicated phase space deformations. In particular, this is possible when there exists $G(x,\xi)$ linear in $\xi$ satisfying $H_pG>0$ on $\{p=0\}\cap\{|\xi|\geq C\}$. This type of deformation is analogous to the method of {\em complex scaling} (rediscovered 
as the method of {\em perfectly matched layers} in numerical analysis) -- see \cite[\S\S 4.5,4.7]{dizzy} for an introduction and references.

\subsection{Deformations of analytic pseudodifferential operators}

For $ u \in \mathscr D' ( \TT^n ) $ 
we extend $u$ to be $2\pi \mathbb{Z}^n$ periodic on $\mathbb{R}^n$.
We consider 
\begin{equation}
\label{e:PAssume}
\begin{gathered}
(Pu)(x)=\lim_{\e\to 0^+}\lim_{\delta\to 0^+}\frac{1}{(2\pi)^n}\int e^{i\langle x-y,\xi\rangle -\e|\xi|^2-\delta|x-y|^2}p\big(x,\xi\big)u(y)dyd\xi,\\
 |p(z,\zeta)|\leq C\langle \Re \zeta\rangle^m,\ \  |\Im z|\leq a,\ \ 
 |\Im \zeta | \leq b\langle \Re \zeta\rangle,  \\ 
p(x,\xi)=p(x+2\pi k,\xi), \ \ k\in \mathbb{Z}^n.
\end{gathered}
\end{equation}
and 
$G(x,\xi)\in S^1(T^*\mathbb{T}^n)$ such that 
$$
G(x,\xi)=\langle G_0(x),\xi\rangle,\qquad G_0\in C^\infty(\mathbb{T}^n;\mathbb{R}^n).
$$
\noindent{\bf{Remark:}} Observe that ~\eqref{e:PAssume} agrees with the definition of the standard left quantization of the symbol $p$ as in~\eqref{eq:defP}. 

We consider the complex deformed operator, $P_\theta$, defined by the property that when $u$ is analytic in a sufficiently large neighbourhood of $ \TT^n $ (or simply for $ u $ being a trigonometric polynomial),
$$
P_\theta\big(  u|_{\Gamma_\theta}\big)=(Pu)|_{\Gamma_\theta},\qquad \Gamma_\theta:=\{ x+i\theta G_0(x)\mid x\in \mathbb{T}^n\}. 
$$

We start by deriving a formula for the kernel of $P_\theta$: 

\begin{lemm}
\label{l:analytic}
Suppose $u\in C^{\omega}(\mathbb{T}^n)$ extends analytically to $|\Im z|<a$. Then, for $|\Im z|<a$, the limit 
$$
v(z)=\lim_{\e\to 0^+}\lim_{\delta\to 0^+}\frac{1}{(2\pi)^n}\int e^{i\langle z-y,\xi\rangle -\e|\xi|^2-\delta (z-y)^2}p(z,\xi)u(y)dyd\xi,
$$
exists and, moreover, $v(z)$ is the analytic continuation of $Pu$.
\end{lemm} 
\begin{proof}
For each fixed $\e,\delta$, the resulting function of $z$ is manifestly analytic in a neighbourhood of $ \TT^n $ (or $ \RR^n $ if we think of
periodic functions). Therefore, in order to see that $v$ itself is analytic, we need only show that the limit exists and the convergence is uniform on compact subsets of $|\Im z|<a$. For this, we deform the contour in $y$ to 
$$
\Gamma(z): y\mapsto y+i\Im z,
$$
so that we have
\begin{align*}
v(z)&=\lim_{\e\to 0^+}\lim_{\delta\to 0^+}\frac{1}{(2\pi)^n}\int \int_{\Gamma(z)} e^{i\langle z-y,\xi\rangle -\e|\xi|^2-\delta|z-y|^2}p(z,\xi)u(y)dyd\xi\\
&=\lim_{\e\to 0^+}\lim_{\delta\to 0^+}\frac{1}{(2\pi)^n}\int e^{i\langle \Re z-y,\xi\rangle -\e|\xi|^2-\delta|\Re z-y|^2}p(z,\xi)u(y+i\Im z)dyd\xi.
\end{align*}
This contour deformation is justified for each fixed $\delta$ since the integrand is super exponentially decaying in $y$. Now, integrating by parts in $\xi$ using 
$$
L_{\e}=\frac{1+\langle \Re z-y -2i\e, D_\xi\rangle  }{1+|\Re z-y|^2+4\e|\xi|^2},
$$
we have for any $N>0$
\begin{equation*}
v(z)=\lim_{\e\to 0^+}\lim_{\delta\to 0^+}\frac{1}{(2\pi)^n}\int e^{i\langle \Re z -y,\xi\rangle -\e|\xi|^2-\delta|\Re z-y|^2}(L_{\e}^t)^N\big(p(z,\xi)\big)u(y+i\Im z)dyd\xi.
\end{equation*}
In particular, since $|\partial_{\xi}^\alpha p|\leq C_\alpha\langle \Re \xi\rangle^{m-|\alpha|},$ the integrand is bounded by $C_N\langle \Re z-y\rangle^{-N}\langle \xi\rangle^me^{-\e|\xi|^2}$ uniformly in $\delta>0$ and compact subsets of $|\Im z|<a$. Hence, for $N$ large enough, the limit in $\delta$ exists and is uniform in compact subsets of $|\Im z|<a$. It is given by
\begin{equation*}
v(z)=\lim_{\e\to 0^+}\frac{1}{(2\pi)^n}\int e^{i\langle \Re z-y,\xi\rangle -\e|\xi|^2}(L_{\e}^t)^N\big[p(z,\xi)\big]u(y+i\Im z)dyd\xi.
\end{equation*}
Defining 
$
B:= ( 1+\langle \xi,D_y\rangle)/({1+|\xi|^2}) $ we have, for any $N>0$,
\begin{equation*}
v(z)=\lim_{\e\to 0^+}\frac{1}{(2\pi)^n}\int_{\Gamma} e^{i\langle \Re z-y,\xi\rangle -\e|\xi|^2}(B^t)^N\Big[(L_{\e}^t)^N\big[p(z,\xi)\big]u(y+i\Im z)\Big]dyd\xi.
\end{equation*}
Since $|\partial^{\alpha}_y u(y+i\Im z)|\leq C_\alpha$ and $|\partial^{\alpha}_x\partial_{\xi}^\beta p(x,\xi)|\leq C_{\alpha\beta}\langle \Re\xi\rangle^{m-|\beta|}$, the integrand is bounded by 
$$
C_N\langle \Re z-y\rangle^{-N}\langle \xi\rangle^{m-N},
$$
uniformly in $\e>0$ and compact subsets of $|\Im z|<a$. In particular, the limit in $\e$ exists, is uniform in compact subsets of $|\Im z|<a$, and is given by 
$$
v(z)=\frac{1}{(2\pi)^n}\int_{\Gamma} e^{i\langle \Re z-y,\xi\rangle}(B^t)^N\Big[(L_{0}^t)^N\big[p(z,\xi)\big]u(y+i\Im z)\Big]dyd\xi.
$$

Thus, we see that $v$ is analytic on $|\Im z|<a$ and agrees with $Pu$ on $\Im z=0$. In particular, by uniqueness of analytic continuation, $v(z)=(Pu)(z)$.
\end{proof}

We now move to the representation for the Scwartz kernel of $P_\theta$:
\begin{lemm}
\label{l:ker}
The kernel of $P_\theta$ acting on $2\pi \mathbb{Z}^n$ periodic functions on $\mathbb{R}^n$ is given by
\begin{equation}
\label{e:kernel}
K_{\theta}(x,y)=\frac{1}{(2\pi)^n}\int e^{i\langle x-y, \xi\rangle} p\big(\gamma_{\theta}(x),(e_\theta(x,y)^{-1})^t\xi\big)\frac{ \det (\partial_y\gamma_{\theta}(y))}{\det e_\theta(x,y)}d\xi
\end{equation}
where the integral is interpreted as an oscillatory integral, 
$$
\gamma_\theta:=x+i\theta G_0(x),
$$
and $e_{\theta}(x,y)$ satisfies 
$$
e_{\theta}(x,y)(x-y)=\gamma_\theta(x)-\gamma_\theta(y).
$$
In particular, $P_{\theta}\in \Psi^m$ and its principal symbol is given by 
$$
\sigma(P_{\theta})=p(\gamma_\theta(x), (\partial\gamma_\theta (x)^{-1})^t\xi).
$$
\end{lemm}
\noindent {\bf{Remark:}} Note that the symbol in~\eqref{e:kernel} is \emph{not} $2\pi \mathbb{Z}^n$ periodic in $x$. However, it is of the form $a(x,x-y,\xi)$ where $a$ is $2\pi \mathbb{Z}^n$ periodic in the first variable. Therefore, it still maps periodic functions to periodic functions.
\begin{proof}
By Lemma~\ref{l:analytic}, for $u$ analytic on $|\Im z|<a$ and $\theta$ small enough
\begin{align*}
&(P_\theta u|_{\Gamma_\theta})(\gamma_{\theta}(x))\\
&=\lim_{\e\to 0^+}\lim_{\delta\to 0^+}\frac{1}{(2\pi)^n}\int e^{i\langle \gamma_{\theta}(x)-y,\xi\rangle-\e|\xi|^2-\delta (\gamma_\theta(x)-y)^2}p(\gamma_{\theta}(x),\xi)u(y)dyd\xi.
\end{align*}
Now, since for each fixed $\delta>0$, the integrand is super exponential decaying in $y$, we may deform the contour in $y$ to $\Gamma_\theta$, to obtain 
\begin{equation}
\label{e:discreteUse}
\begin{aligned}
(P_\theta u|_{\Gamma_\theta})(\gamma_{\theta}(x))&=\lim_{\e\to 0^+}\lim_{\delta\to 0^+}\frac{1}{(2\pi)^n}\int e^{i\langle \gamma_{\theta}(x)-\gamma_{\theta}(y),\xi\rangle-\e|\xi|^2-\delta (\gamma_\theta(x)-\gamma_{\theta}(y))^2}\\
&\qquad\qquad \qquad\qquad\qquad p(\gamma_{\theta}(x),\xi) \det (\partial_y\gamma_{\theta}(y))u(\gamma_{\theta}(y))dyd\xi.
\end{aligned}
\end{equation}
Next, using that for each fixed $\e>0$, the integrand is super exponentially decaying in $\xi$ and that, with $e_\theta(x,y)(x-y)=\gamma_\theta(x)-\gamma_\theta(y)$, we have $e_\theta = \Id+O(\theta\langle x-y\rangle^{-1})$, we can deform the contour in $\xi$ to $\Gamma_1=\xi \mapsto (e_\theta(x,y)^t)^{-1}\xi$, to obtain
\begin{align*}
(P_\theta u|_{\Gamma_\theta})(\gamma_{\theta}(x))&=\lim_{\e\to 0^+}\lim_{\delta\to 0^+}\frac{1}{(2\pi)^n}\int e^{i\langle x-y ,\xi\rangle-\e((e_\theta^t)^{-1}\xi)^2-\delta (\gamma_\theta(x)-\gamma_{\theta}(y))^2}\\
&\qquad\qquad\qquad p\big(\gamma_{\theta}(x),(e_\theta(x,y)^{-1})^t\xi\big)\frac{ \det (\partial_y\gamma_{\theta}(y))}{\det e_\theta(x,y)}u(\gamma_{\theta}(y))dyd\xi.
\end{align*}
Now, integrating by parts as in the proof of Lemma~\ref{l:analytic}, results in the formula~\eqref{e:kernel}.

To prove the final claim, let $\chi\in C_c^\infty(-1,1)$ with $\chi\equiv 1$ near $0$. Then, for any $\delta>0$, we have by~\eqref{e:kernel} that with 
$$
P_{\theta}'=\frac{1}{(2\pi)^n}\int e^{i\langle x-y, \xi\rangle} p\big(\gamma_{\theta}(x),(e_\theta(x,y)^{-1})^t\xi\big)\frac{ \det (\partial_y\gamma_{\theta}(y))}{\det e_\theta(x,y)}\chi(\delta^{-1}|x-y|)d\xi,
$$
the error $P_\theta-P_\theta'$ is smoothing and maps periodic functions to periodic functions. In particular, $P_{\theta}$ is a pseudodifferential operator on $\mathbb{T}^n$ with symbol
$$\sigma(P_\theta)=p(\gamma_\theta(x), (\partial\gamma_\theta (x)^{-1})^t\xi),$$
proving the last claim in the lemma.
\end{proof}

\subsection{The resolvent of the deformed operator}

We now consider the setting of Theorem~\ref{t:viscoscity}. Namely, we assume that $P$ is a self-adjoint 0th order pseudodifferential operator and study the properties of $P_\theta$.
\begin{prop}
\label{p:Deform}
Suppose that $P\in \Psi^0$ is self adjoint and satisfies~\eqref{e:PAssume} and that $G=\langle G_0,\xi\rangle \in S^1$ has $H_pG>0$ on $\{p=0\}\cap \{|\xi|>C\}$.  Then there are $\omega_0,\theta_0>0$ such that for $0<\theta<\theta_0$, $\omega\in (-\omega_0,\omega_0)+i(-\omega_0\theta,\infty)$, and all $s\in \mathbb{R}$,
$$
\mc{R}_\theta(\omega):=(P_\theta-\omega)^{-1}:H^s(\mathbb{T}^n)\to H^s(\mathbb{T}^n)
$$
is meromorphic with finite rank poles. 
\end{prop}
\begin{proof}
First, note that there is $\omega_1>0$ such that $H_pG>c>0$ on $\{|p|<\omega_0\}\cap\{|\xi|>C\}$. We compute
\begin{align*}
\sigma(P_\theta)(x,\xi)&=p(x,\xi)+i\theta \langle \partial_xp(x,\xi) ,G_0(x)\rangle+\langle \partial_\xi p(x,\xi)),(\partial\gamma_\theta (x)^{-1})^t-\Id)\xi\rangle+O(\theta^2)\\
&=p(x,\xi)+i\theta(\langle \partial_xp(x,\xi) ,G_0(x)\rangle -\langle \partial_\xi p(x,\xi),(\partial_x G_0)^t(x)\xi\rangle+O(\theta^2)\\
&=p(x,\xi)-i\theta H_pG(x,\xi) +O(\theta^2).
\end{align*}
Therefore, for $\theta>  0$ small enough, and $\omega_0=\min(c,\omega_1)$, $P_\theta-\omega $ is elliptic when $\Im \omega \geq -\omega_0\theta$, $|\Re \omega|<\omega_0.$
 
 In particular, $P_{\theta}-\omega:H^s\to H^s$ is Fredholm for $\omega\in [-\omega_0,\omega_0]+i(-\omega_0\theta,\infty)$. Moreover, since $P_\theta:H^s\to H^s$ is bounded, if $\Im \omega\gg 1$, $P_\theta-\omega$ is invertible by Neumann series and hence $P_\theta-\omega$ has index 0. By the meromorphic Fredholm theorem~\cite[Theorem C.9]{dizzy}, its inverse $\mc{R}_{\theta}(\omega)=(P_\theta-\omega)^{-1}:H^s\to H^s$ is a meromorphic family of operators with finite rank poles for $\omega\in (-\omega_0,\omega_0)+i(-\omega_0\theta,\infty)$.
 \end{proof}
 
 \begin{prop}
 \label{p:sharedPoles}
Let $P$ and $G$ as in Proposition~\ref{p:Deform}. There are $\theta_0,\omega_0>0$ such that for $0<\theta<\theta_0$, the poles of $\mc{R}_{\theta}(\omega)$ for $\omega\in (-\omega_0,\omega_0)+i(-\omega_0\theta,\infty)$ agree with multiplicity with those of $\mc{R}^{H_{\Lambda}}(\omega)$ where $\mc{R}^{H_{\Lambda}}(\omega)$ the resolvent of $P$ on $H^s_{\Lambda}$ from Lemma~\ref{l:Fred1}.
 \end{prop}
 
We will need the following lemma.
\begin{lemm}
\label{l:B5} 
 For $P$ as in~\eqref{e:PAssume}, there are $C>0$ and $\delta>0$ such that for $\Im\omega\geq C$, 
 $$
 \mc{R}^{L^2}(\omega):\mc{A}_{\delta}\to \mc{A}_{\delta}.
 $$
 \end{lemm}
 \begin{proof}
We start by showing that $P:\mathscr{A}_{\delta}\to \mathscr{A}_{\delta}$ is bounded. For this, note that  for $j\in \mathbb{Z}^n$,
 $$
\widehat{Pu}(j)=\sum_{k\in\mathbb{Z}^n} \widehat{u}(k) \widehat{p}(j-k,k)
 $$
 where $\widehat{p}(j,\xi)$ denotes the Fourier series for $p(x,\xi)$ in the $x$ variable.
  Note that by~\eqref{e:PAssume}, there is $C>0$ such that 
 $$
| \widehat{p}(k,\xi)|\leq C\langle \xi\rangle^m e^{-a|k|}.
 $$
 Therefore, for $\delta<\frac{a}{2}$,
 \begin{align*}
 \|Pu\|^2_{\mathscr{A}_{\delta}}&=\sum_j\Big|\sum_k \widehat{u}(k)\widehat{p}(j-k,k)\Big|^2 e^{4|j|\delta}\\&\leq \sum_j\Big(\sum_k |\widehat{u}(k)|^2e^{4|k|\delta}\Big)\Big(\sum_k|\widehat{p}(j-k,k)|^2 e^{4(|j|-|k|)\delta}\Big)\\
 &= \|u\|^2_{\mathscr{A}_{\delta}}\Big(\sum_{k,j}e^{-2a|j-k|} e^{4(|j|-|k|)\delta}\Big)\leq C \|u\|^2_{\mathscr{A}_{\delta}}\Big(\sum_{j}e^{(4\delta -16\frac{a\delta}{4\delta+2a})|j|}\Big)\leq C\|u\|_{\mathscr{A}_{\delta}}^2.
 \end{align*}
Since $\Im \omega>\|P\|_{\mathscr{A}_{\delta}\to \mathscr{A}_{\delta}}$, 
 $
 \mc{R}^{L^2}(\omega)=-\sum_{k=0}^\infty \omega^{-k-1}P^k $
  the proof is complete.
 \end{proof}
 
 \begin{proof}[Proof of Proposition~\ref{p:sharedPoles}]
Let $\omega_0$ be the minimum of $\omega_0$ from Proposition~\ref{p:Deform} and Lemma~\ref{l:Fred1} and  suppose that $\omega\in  (\omega_0,\omega_0)+i(-\omega_0\theta,\infty)$ with $\Im \omega\gg 1$.  Then, $\mc{R}_{\theta}(\omega):H^s(\Gamma_\theta) \to H^s(\Gamma_\theta)$ and $\mc{R}^{L_2}(\omega):\mathscr{A}_{\delta}\to \mathscr{A}_{\delta}$.
 
 Let $u\in \mathscr{A}_{\delta}$. Then we have $\mc{R}^{L_2}(\omega)u\in \mathscr{A}_\delta$ and
 $$(P_\theta -\omega) \big([\mc{R}^{L_2}(\omega)u]|_{\Gamma_\theta}\big)=((P-\omega)\mc{R}^{L_2}(\omega)u)|_{\Gamma_\theta}=u|_{\Gamma_\theta}.$$ 
 In particular, $[\mc{R}^{L_2}(\omega)u]|_{\Gamma_\theta}=\mc{R}_{\theta}(\omega)(u|_{\Gamma_\theta})$.  
 
 For $u,\,v\in \mathscr{A}_{\delta}$, and $\Im \omega \gg 1$,
 \begin{align*}
 \langle \mc{R}_\theta(\omega)u|_{\Gamma_\theta},v|_{\Gamma_\theta}\rangle_{L^2(\Gamma_\theta)}&=\langle \big[\mc{R}^{L^2}(\omega)u|_{\mathbb{T}^n}\big]\big|_{\Gamma_\theta},v|_{\Gamma_\theta}\rangle_{L^2(\Gamma_\theta)}.
 \end{align*}
 By Lemma~\ref{l:poleL2}, when $\Im \omega>0$ and $u\in \mathscr{A}_{\delta}$, we have $\mc{R}^{H_{\Lambda}}(\omega)u=\mc{R}^{L_2}(\omega)u$.
 Then, deforming the contour of integration in the inner product to $\mathbb{T}^n$,
 \begin{align*}
  \langle \mc{R}_\theta(\omega)u|_{\Gamma_\theta},v|_{\Gamma_\theta}\rangle_{L^2(\Gamma_\theta)}&=\langle \big[\mc{R}^{L^2}(\omega)u|_{\mathbb{T}^n}\big]\big|_{\Gamma_\theta},v|_{\Gamma_\theta}\rangle_{L^2(\Gamma_\theta)}\\&=\langle \mc{R}^{L^2}(\omega)u,v\rangle_{L^2(\mathbb{T}^n)}\\
 &=\langle \mc{R}^{H_{\Lambda}}(\omega)u,v\rangle_{L^2(\mathbb{T}^n)}=\langle \mc{R}^{H_{\Lambda}}(\omega)u,v\rangle_{\mathscr{A}_{-\delta}(\mathbb{T}^n),\mathscr{A}_{\delta}(\mathbb{T}^n)}.
 \end{align*}

 Since $\mathscr{A}_{\delta}\subset H_{\Lambda}^0\subset \mathscr{A}_{-\delta}$, both sides of this equality continue meromorphically from $\Im \omega\gg1 $ to $\omega\in  (-\omega_0,\omega_0)+i(-\omega_0\theta,\infty)$ and the equality continues to hold for $\omega$ in this set. Finally, since $\mathscr{A}_{\delta}$ is dense in $H^s_{\Lambda}$ and $\mathscr{A}_{\delta}|_{\Gamma_\theta}$ is dense in $L^2(\Gamma_\theta)$, this equality implies that the poles of $\mc{R}_\theta$ and $\mc{R}^{H_{\Lambda}}$ coincide with an agreement of multiplicities.
 \end{proof}

 \subsection{Numerical examples and discretization}

\begin{figure}
\includegraphics[width=13cm]{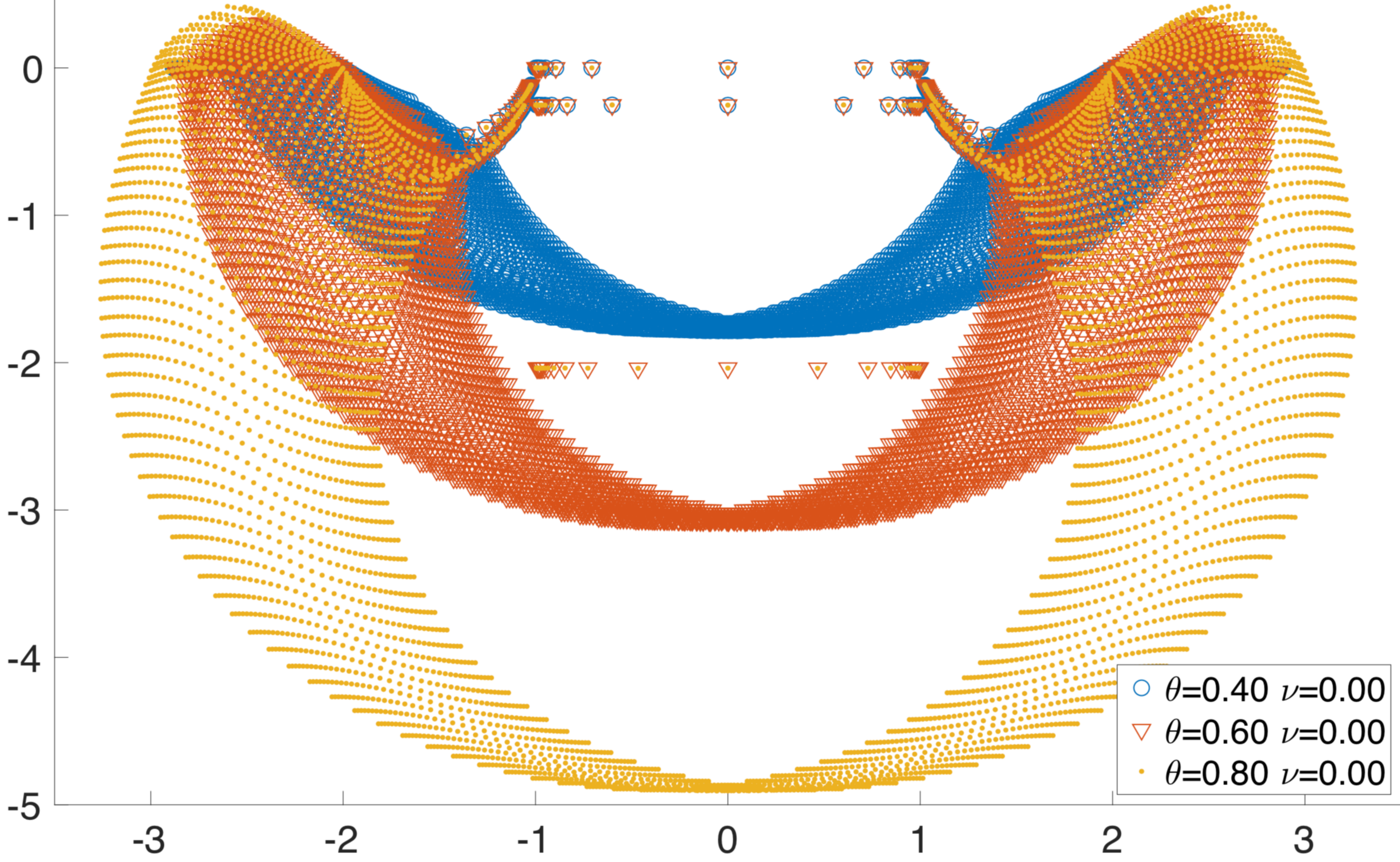}
\caption{\label{f:2}We display the eigenvalues of $P_\theta$ acting on $e^{inx_2}L^2_{x_1}$ for three different values of $\theta$ and $-20\leq n\leq 20$. $P$ is chosen as in~\eqref{e:numeric} with $V_a=0$ and $V_m=((1-\xi^2)+e \xi^2)e^{-\xi^2}$. These choices guarantee the existence of an embedded eigenvalue at 0~\cite[Example 1]{ZTao}. Note that once the eigenvalues emerge from the continues spectrum, they are independent of the choice of $\theta$. }
\end{figure}

\begin{figure}
\includegraphics[width=13.5cm]{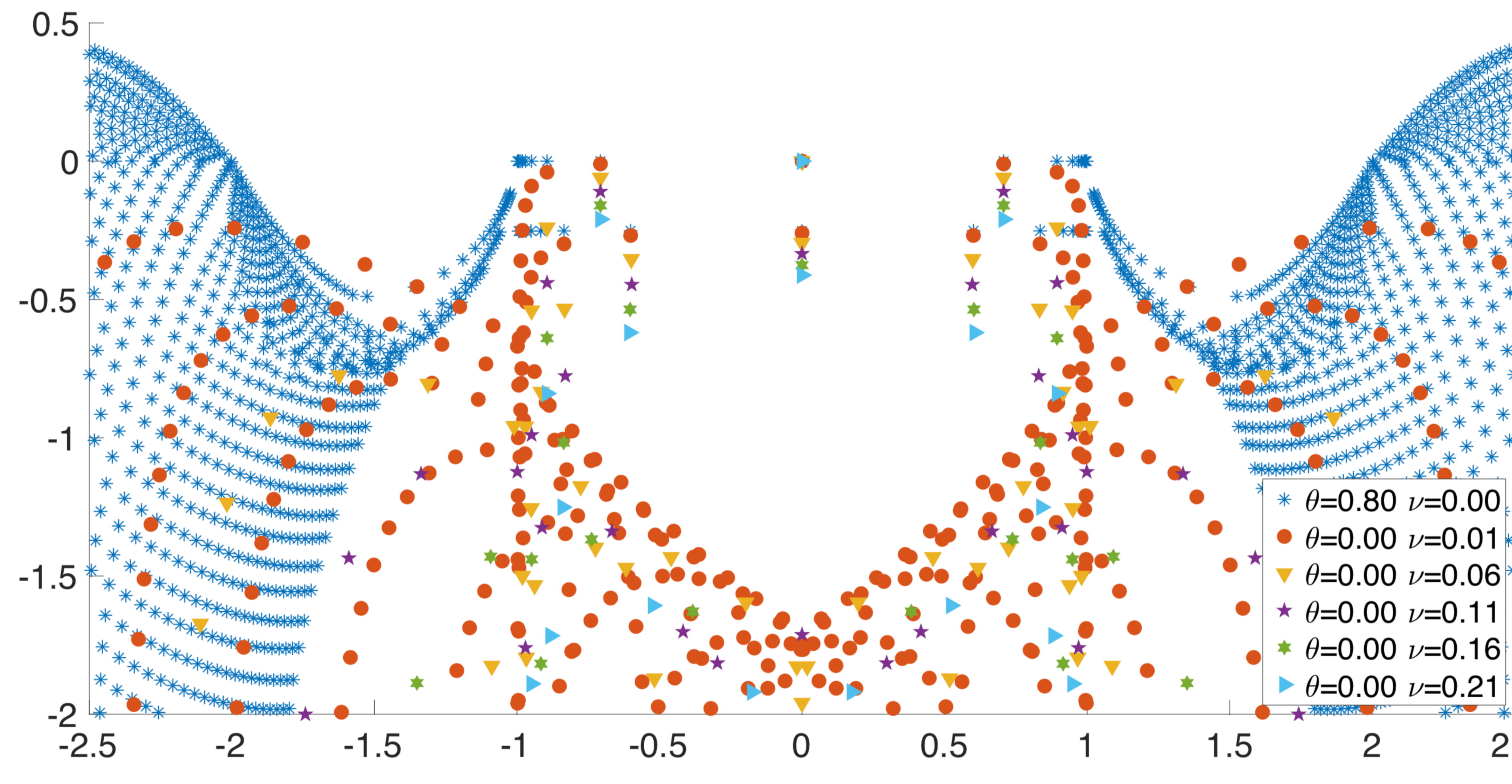}
\caption{\label{f:3}We display the eigenvalues of $P_\theta$ on $e^{inx_2}L^2_{x_1}$ with $\theta=0.8$ as well as those for $P+i\nu\Delta$ for five values of $\nu$ approaching zero and $-20\leq n\leq 20$. $P$ is chosen as in~\eqref{e:numeric} with $V_a=0$ and $V_m=((1-\xi^2)+e \xi^2)e^{-\xi^2}$. These choices guarantee the existence of an embedded eigenvalue at 0~\cite[Example 1]{ZTao}. }
\end{figure}
 
 In our numerical study, we consider operators of the form 
 \begin{equation}
 \label{e:numeric}
 P= \langle D\rangle^{-1}D_{x_2}+\sin(x_1)(\Id-V_m(D_{x_1}))+(\Id-V_m(D_{x_1}))\sin(x_1) +V_a(D_{x_1}),
 \end{equation}
 with $V_{\bullet}(\xi_1)$, $ \bullet = a, m $, satisfying 
\begin{equation}
\label{e:rapidDecay}
|V_{\bullet}(\xi_1)|\leq Ce^{-c|\Re \xi_1|^2}, \ \ \ |\Im \xi_1|<b\langle \Re \xi_1\rangle, \ \ \ \bullet = a, m .
\end{equation}
Then, $P$ satisfies the assumptions of Proposition~\ref{p:Deform} with $G_0=(-2\cos(x_1),0).$

Since the deformation $G$ does not involve $x_2$, we may decompose 
$$
L^2(\mathbb{T}^2)=\bigoplus_{n=-\infty}^\infty e^{inx_2}L^2(x_1)
$$ 
and use that $D_{x_2}|_{e^{inx_2}L^2_{x_1}}=n\Id.$

In order to discretize the operator $P_\theta|_{e^{inx_2}L^2_{x_1}}$, we replace $\mathbb{T}_{x_1}$ by $\frac{2\pi}{N}(\mathbb{Z}/N\mathbb{Z})$ and denote by $X_1,Y_1\in [\frac{2\pi}{N}(\mathbb{Z}/N\mathbb{Z})]^N$ vectors with $j^{\text{th}}$ entry $X_1(j)=\tfrac{2\pi j}{N}$. We will represent the Fourier dual to $\frac{2\pi}{N}(\mathbb{Z}/N\mathbb{Z})$ as $-\frac{N}{2}+\mathbb{Z}/N\mathbb{Z}$ and index vectors on the Fourier side with $K\in \{-\frac{N}{2},-\frac{N}{2}+1,\dots, \frac{N}{2}-1\}.$  We then compute a matrix which approximates the action of $P_\theta$ on the Fourier series side. 

Letting $\gamma_\theta(x):=x+i\theta G_0(x)$, and $\mc{F}_N$ the matrix for the Discrete Fourier transform,
$$
(\mc{F}_N)_{K,j}:=-\frac{e^{-2\pi i jK/N}}{\sqrt{N}} ,\qquad  -\frac{N}{2}\leq K\leq \frac{N}{2}-1,\, \qquad0\leq j\leq N-1
$$ 
we obtain for a vector $u\in \mathbb{C}^N$,
$$
\mc{F}_N\frac{u}{\gamma_{\theta}'}={\bf{\Gamma}} \mc{F}_N u,\qquad {\bf{\Gamma}}:=\mc{F}_N\operatorname{Diag}\Big(\frac{1}{\gamma'_\theta(X_1)}\Big)\mc{F}_N^*,
$$
so that 
$$
[\mc{F}_N (D_{x_1})_{\theta}u] \approx   [{\bf{\Gamma} \bf{K}} \mc{F}_Nu]
$$
where
$$
{\bf{K}}=\operatorname{Diag}(K).
$$
At this point we discretize $(\langle D\rangle^{-\frac{1}{2}})_{\theta}$ using the functional calculus of~\cite[Section 4]{SjZwDist}. In particular, writing 
$$
{\bf{\langle D\rangle_{\theta}}}=(1+n^2){\bf{\Id}}+( {\bf{\Gamma K}})^2,
$$ 
we have 
$$
\mc{F}_N[(\langle D\rangle^{-\frac{1}{2}})_{\theta}u]\approx (\sqrt{{\bf{\langle D\rangle_{\theta}}}})^{-1}\mc{F}_Nu
$$
where the square root is taken in the sense of matrices and all eigenvalues are taken with non-negative real part. 

\noindent{\bf{Remark:}} Note that the functional calculus definition of $(\langle D\rangle^{-1/2})_{\theta}$ agrees with the definition~\eqref{e:discreteUse} with $p=(1+n^2+\xi^2)^{-\frac{1}{2}}$. To see this observe that at $\theta=0$, the two operators agree and hence their analytic extensions to $\Gamma_\theta$ agree. 

The operators $[V_{a/m}(D)]_\theta$ are computed by using~\eqref{e:discreteUse} to write 
$$
(V_{a/m}(D))_{\theta} u(x)=\sum_k\frac{1}{2\pi}\int  e^{i(\gamma_\theta(x)-\gamma_\theta(y))k}V_{a/m}(k)u(y)dy.
$$
This sum converges since $V_{a/m}$ satisfies~\eqref{e:rapidDecay}. We then evaluate $x$ and $y$ on $\frac{2\pi}{N}(\mathbb{Z}/N\mathbb{Z})$ in the above kernel to obtain the matrix approximation on the Fourier transform side. More precisely, putting 
$$
({{\bf \widehat{V}}^{\theta}_{\bullet }})_{ij}:=\sum_k\frac{1}{2\pi}\int  e^{i(\gamma_\theta(X_1(i))-\gamma_\theta(Y_1(j)))k}V_{\bullet}(k), \ \ 
\bullet = a , m , 
$$
we have
$$
\mc{F}_NV_{\bullet}(D)_{\theta} u\approx  \mathbf{V}^{\theta}_{\bullet}
\mc{F}_Nu,\qquad {\bf{V}}^{\theta}_{\bullet}= \mc{F}_N {\bf \widehat{V}}^{\theta}_{\bullet}\mc{F}_N^*.
$$
Note that this approximation is valid since we take $|V(\xi)|\leq  Ce^{-C|\xi|}$ and therefore the sum in $k$ converges rapidly. Finally, writing 
$$
{\bf{S}^{\theta}}=\mc{F}_N\operatorname{Diag}(\sin(\gamma_\theta(X_1)))\mc{F}_N^*,
$$ 
Our total operator is then approximated by
$$
P_{\theta}|_{e^{inx_2}L^2_{x_1}}\approx \mc{F}_N^*{\bf{P}}_N^\theta\mc{F}_N,\ \ {\bf{P}}_N^\theta:= n (\sqrt{{\bf{\langle D\rangle_{\theta}}}})^{-1}+{\bf{S}^{\theta}}({\bf{\Id}}-{\bf{V}}^{\theta}_{m})+({\bf{\Id}}-{\bf{V}}^{\theta}_{m}){\bf{S}^\theta} +{\bf{V}}^{\theta}_a  .
$$ 
Since $\mc{F}_N$ is unitary, we compute the eigenvalues of ${\bf{P}}^\theta_N$ to approximate the eigenvalues of $P_\theta|_{e^{inx_2}L^2_{x_1}}$.

When approximating $P+i\nu \Delta$, these computations are much simpler and we use the standard Fourier series approximations
$$
P+i\nu \Delta \approx\mc{F}_N^*{\bf{P}}_N^\nu\mc{F}_N,\ \  {\bf{P}}_N^{\nu}:= n {\bf{\langle D\rangle^{-1/2}}}+{\bf{S}}({\bf{\Id}}-{\bf{V}}_{m})+({\bf{\Id}}-{\bf{V}}_{m}){\bf{S}} +{\bf{V}}_a  -i\nu {\bf{K^2}},
$$
where 
\begin{gather*}
{\bf{\langle D\rangle}}^{-1/2}:=\operatorname{Diag}(\langle K\rangle^{-1/2}),\ \  {\bf{V}_{\bullet}}:=\operatorname{Diag}(V_{\bullet}(K)),\ \ \mathbf{S}=\mc{F}_N\operatorname{Diag}(\sin(X_1))\mc{F}_N^*.
\end{gather*}

The results of several numerical experiments are displayed in Figures~\ref{f:1},~\ref{f:2}, and~\ref{f:3}.

\end{document}